\documentclass[reqno,12pt]{amsart}
\usepackage[colorlinks=true, linkcolor=blue, citecolor=blue]{hyperref}

\usepackage{amssymb}
\usepackage{amsmath, graphicx, rotating}
\usepackage{color}
\usepackage{soul}
\usepackage[dvipsnames]{xcolor}

\usepackage{ifthen}
\usepackage{xkeyval}
\usepackage{todonotes}
\usepackage[T1]{fontenc}
\usepackage{lmodern}
\usepackage[english]{babel}

\usepackage{ upgreek }
\usepackage{stmaryrd}
\SetSymbolFont{stmry}{bold}{U}{stmry}{m}{n}
\usepackage{amsthm}
\usepackage{float}

\usepackage{ bbm }
\usepackage{ stmaryrd }
\usepackage{ mathrsfs }
\usepackage{ frcursive }
\usepackage{ comment }

\usepackage{pgf, tikz}
\usetikzlibrary{shapes}
\usepackage{varioref}
\usepackage{enumitem}

\definecolor{rouge}{rgb}{0.7,0.00,0.00}
\definecolor{vert}{rgb}{0.00,0.5,0.00}
\definecolor{bleu}{rgb}{0.00,0.00,0.8}
\usepackage[margin=1in]{geometry}
\newtheorem{theorem}{Theorem}[section]
\newtheorem*{theorem*}{Theorem}
\newtheorem{lemma}[theorem]{Lemma}
\newtheorem{definition}[theorem]{Definition}

\newtheorem{proposition}[theorem]{Proposition}

\labelformat{hypothesis}{\textbf{M\kern-0.1mm#1}}

\newtheorem{condition}{Condition}

\newtheorem{conditionA}{A\kern-0.1mm}
\labelformat{conditionA}{\textbf{A\kern-0.1mm#1}}

\theoremstyle{definition}

\newtheorem{remark}[theorem]{Remark}

\def \eref#1{\hbox{(\ref{#1})}}

\numberwithin{equation}{section}

\def\geq{\geqslant}
\def\leq{\leqslant}

\def\RR{\mathbb{R}}
\def\PP{\mathbb{P}}
\def\EE{\mathbb{E}}

\def\vare{{\varepsilon}}
\def \eref#1{\hbox{(\ref{#1})}}

\def\EE{\mathbb{ E}}

\def\ep{{\varepsilon}}

\begin{document}

\title[Diffusion Approximation for Multi-Scale MVSDE]
{Diffusion Approximation for Multi-Scale McKean-Vlasov SDEs Through Different Methods}
\author{Wei Hong}
\curraddr[Hong, W.]{Center for Applied Mathematics, Tianjin University, Tianjin 300072, China}
\email{weihong@tju.edu.cn}

\author{Shihu Li}
\curraddr[Li, S.]{School of Mathematics and Statistics, Jiangsu Normal University, Xuzhou 221116, China}
\email{shihuli@jsnu.edu.cn}

\author{Xiaobin Sun}
\curraddr[Sun, X.]{ School of Mathematics and Statistics, Jiangsu Normal University, Xuzhou 221116, China }
\email{xbsun@jsnu.edu.cn}

\begin{abstract}
In this paper, we aim to study the diffusion approximation for multi-scale McKean-Vlasov stochastic differential equations. More precisely, we prove the weak convergence of slow process $X^\varepsilon$  in $C([0,T];\RR^n)$ towards the limiting process $X$ that is the solution of a distribution dependent stochastic differential equation in which some new drift and diffusion terms compared to the original equation appear. The main contribution is to use two different methods
to explicitly characterize the limiting equations respectively. The obtained diffusion coefficients in the limiting equations have different form through these two methods, however it will be asserted that they are essential the same by a comparison.

\end{abstract}

\date{\today}
\subjclass[2000]{ Primary 60H10; Secondary 60F05,~60F10}
\keywords{Diffusion approximation; McKean-Vlasov equation; Multi-scale; Martingale problem approach; Martingale representation theorem}

\maketitle

\tableofcontents 
\section{Introduction}

The diffusion approximation problem mainly concerns the convergence of singularly perturbed ordinary differential equations with random inputs, which involves a homogenization term, tends to   stochastic differential equations (SDEs for short). To our knowledge, the diffusion approximation for multi-scale stochastic differential equations was first studied by Papanicolaou, Stroock and Varadhan \cite{PSV} on a compact state space.
A simple example is the following
\begin{eqnarray}
d X^{\varepsilon}_t =b\left(X^{\varepsilon}_t, X^{\varepsilon}_{t}/\sqrt{\vare}\right)dt+\frac{1}{\sqrt{\varepsilon}}K\left(X^{\varepsilon}_t, X^{\varepsilon}_{t}/\sqrt{\vare}\right)dt+\sigma\left(X^{\varepsilon}_t, X^{\varepsilon}_{t}/\sqrt{\vare}\right)d W_t,~ X^{\varepsilon}_0=x\in\RR^n,\label{simple Ex}
\end{eqnarray}
where  the small enough parameter $\varepsilon>0$, and $\frac{1}{\sqrt{\varepsilon}}K\left(X^{\varepsilon}_t, X^{\varepsilon}_{t}/\sqrt{\vare}\right)$ is the homogenization term which has its own interest in the theory of partial differential equations (cf.~\cite{HP08,HP1}) and is also very useful in many physical systems (cf.~\cite{MR83,Ne67,PS08}).
Denote $\mu^{\vare}$ by the distribution of solution $X^{\varepsilon}$ in $C([0,\infty);\RR^n)$. Under some proper conditions on the coefficients, it was proved that
$$\mu^{\vare}\rightarrow \mu~~\text{weakly},~\text{as}~ \vare\rightarrow 0,$$
where  $\mu$ is the distribution of the limiting process (see \cite[Theorem 6.1]{BLP}). In the diffusion approximation theory, due to the appearance of homogenization term in \eref{simple Ex}, an extra diffusion term will appear in the limiting equation typically.

Note that if we denote $Y^{\varepsilon}_{t}=X^{\varepsilon}_{t}/\sqrt{\vare}$, then  \eref{simple Ex} is equivalent to the following coupled stochastic system
\begin{equation*}\left\{\begin{array}{l}\label{Example}
\displaystyle
d X^{\varepsilon}_t = b\left(X^{\varepsilon}_t, Y^{\varepsilon}_{t}\right)dt+\frac{1}{\sqrt{\varepsilon}}K\left(X^{\varepsilon}_t, Y^{\varepsilon}_{t}\right)dt+\sigma\left(X^{\varepsilon}_t, Y^{\varepsilon}_{t}\right)d W_t, \\
\displaystyle d Y^{\varepsilon}_t =\frac{1}{\varepsilon}K\left(X^{\varepsilon}_t, Y^{\varepsilon}_{t}\right)dt+\frac{1}{\sqrt{\vare}}b\left(X^{\varepsilon}_t, Y^{\varepsilon}_{t}\right)dt+\frac{1}{\sqrt{\vare}}\sigma\left(X^{\varepsilon}_t, Y^{\varepsilon}_{t}\right)d W_t,\\
\displaystyle X^{\varepsilon}_0=x,~Y^{\varepsilon}_0=x/\sqrt{\vare}.
\end{array}\right.
\end{equation*}
Hence, the classical diffusion approximation problem can be reduced to a more general type
\begin{equation}\left\{\begin{array}{l}\label{eq40}
\displaystyle
d X^{\varepsilon}_t = b(X^{\varepsilon}_t, Y^{\varepsilon}_t)dt+\frac{1}{\sqrt{\varepsilon}}K(X^{\varepsilon}_t, Y^{\varepsilon}_t)dt+\sigma(X^{\varepsilon}_t, Y^{\varepsilon}_t)d W_t, \\
\displaystyle d Y^{\varepsilon}_t =\frac{1}{\varepsilon}f(X^{\varepsilon}_t, Y^{\varepsilon}_t)dt+\frac{1}{\sqrt{\varepsilon}}h( X^{\varepsilon}_t, Y^{\varepsilon}_t)d t+\frac{1}{\sqrt{\varepsilon}}g( X^{\varepsilon}_t, Y^{\varepsilon}_t)d W_t,\\
\displaystyle X^{\varepsilon}_0=x,~Y^{\varepsilon}_0=y.
\end{array}\right.
\end{equation}
When $h\equiv0$, Pardoux and Veretennikov \cite{PV1,PV2} have studied the asymptotic behavior of the solution $X^{\vare}$ to (\ref{eq40}) in $C([0,\infty);\RR^n)$, the main technique is based on the auxiliary Poisson equation and applying the martingale problem approach to characterize the limiting process. Recently, R\"{o}ckner and Xie \cite{RX21} developed the regularity of Poisson equation and established a more general result of diffusion approximation for the multi-scale SDEs. For more results on this subject, we refer to \cite{BK04,BS,FW2021, GM, KY05,LWX,XY21,WR15,YW19}.

In this paper, we are interested in the theory of diffusion approximation for a class of multi-scale McKean-Vlasov SDEs.
The McKean-Vlasov SDEs in the form of
\begin{equation}\label{eq01}
dX_t=b(X_t,\mathscr{L}_{X_t})dt+\sigma(X_t,\mathscr{L}_{X_t})dW_t,
\end{equation}
where $\mathscr{L}_{X_t}$ stands for the distribution of $X_t$, have attracted considerable attention in recent years.
This type of models arose in \cite{M1} by McKean, which could be view as
the limit of $N$-interacting particle systems in a mean-field way (called the {\it propagation of chaos}) while $N$ goes to
infinity, we refer the reader to the lecture notes \cite{S91}. Another important application of McKean-Vlasov SDEs is it can be used to characterize the nonlinear Fokker-Planck-Kolmogorov (FPK for short) equations (cf.~\cite{S91,WFY}). More specifically, the distribution density (denoted by $\rho_t$) of solution to (\ref{eq01})  solves the following nonlinear PDE
\begin{equation*}
\partial_t\rho_t=L^*\rho_t,~~t\geq 0,
\end{equation*}
where $L$ is a second order nonlinear operator and $L^*$ denotes its adjoint operator. In particular, it can be applied to solve the homogenous Landau equations,~i.e.,
\begin{eqnarray*}
\partial_tf_t=\frac{1}{2}{\rm{div}}\Big\{\int_{\RR^d}a(\cdot-z)(f_t(z)\nabla f_t-f_t\nabla f_t(z))dz\Big\},
\end{eqnarray*}
for some reference coefficients $a:\RR^n\rightarrow\RR^n\otimes\RR^n$. It is well-known that the Landau equation can be seen as an approximation of the Boltzmann equation in the asymptotic of grazing collisions. In the classical paper \cite{FU}, Funaki investigated the diffusion approximation problem of the homogenous Boltzmann equation with soft potentials and showed that  the Boltzmann martingale problem towards to the Landau martingale problem as the scale parameter $\varepsilon\to 0$, see also \cite{FU1} for the case of the Boltzmann equation of Maxwellian molecules.

From the standpoint of McKean-Vlasov SDEs, we consider the following multi-scale McKean-Vlasov stochastic systems
\begin{equation}\left\{\begin{array}{l}\label{eq4}
\displaystyle
d X^{\varepsilon}_t = b(X^{\varepsilon}_t, \mathscr{L}_{X^{\varepsilon}_t}, Y^{\varepsilon}_t)dt+\frac{1}{\sqrt{\varepsilon}}K(X^{\varepsilon}_t, \mathscr{L}_{X^{\varepsilon}_t}, Y^{\varepsilon}_t)dt+\sigma(X^{\varepsilon}_t, \mathscr{L}_{X^{\varepsilon}_t},Y^{\varepsilon}_t)d W_t, \\
\displaystyle d Y^{\varepsilon}_t =\frac{1}{\varepsilon}f(X^{\varepsilon}_t, \mathscr{L}_{X^{\varepsilon}_t}, Y^{\varepsilon}_t)dt+\frac{1}{\sqrt{\varepsilon}}h( X^{\varepsilon}_t, \mathscr{L}_{X^{\varepsilon}_t}, Y^{\varepsilon}_t)d t+\frac{1}{\sqrt{\varepsilon}}g( X^{\varepsilon}_t, \mathscr{L}_{X^{\varepsilon}_t}, Y^{\varepsilon}_t)d W_t,\\
\displaystyle X^{\varepsilon}_0=\xi,~Y^{\varepsilon}_0=\zeta,
\end{array}\right.
\end{equation}
where $\{W_t\}_{t\geq 0}$ is a $d$-dimensional standard Brownian motion on a complete filtration probability space $(\Omega, \mathscr{F}, \{\mathscr{F}_{t}\}_{t\geq0}, \mathbb{P})$, $\xi$ and $\zeta$ are $\mathscr{F}_0$-measurable $\RR^n$ and $\RR^m$-valued random variables respectively, $\varepsilon$ is a small positive parameter describing the ratio of the time-scale between the slow component $X^{\varepsilon}_t$ and fast component $Y^{\varepsilon}_t$. Multi-scale systems are ubiquitous in many fields of sciences and engineering such as chemistry, fluids dynamics and climate dynamics,   the reader can see \cite{An00,PS08} and the reference therein for more precise background and applications.

Since the widely separated time-scales and the cross interactions of slow and fast modes, it is often difficult to study the multi-scale system directly. Hence the asymptotic behavior (specifically, the averaging principle) of the system and a simplified equation which governs the evolution of the system for small $\varepsilon$ are widely studied in the literature. Very recently, the averaging principle of multi-scale McKean-Vlasov stochastic systems
\begin{equation*}\left\{\begin{array}{l}\label{E01}
\displaystyle
d X^{\varepsilon}_t = b(X^{\varepsilon}_t, \mathscr{L}_{X^{\varepsilon}_t}, Y^{\varepsilon}_t)dt+\sigma(X^{\varepsilon}_t, \mathscr{L}_{X^{\varepsilon}_t},Y^{\varepsilon}_t)d W^{1}_t, \\
\displaystyle d Y^{\varepsilon}_t =\frac{1}{\varepsilon}f(X^{\varepsilon}_t, \mathscr{L}_{X^{\varepsilon}_t}, Y^{\varepsilon}_t)dt+\frac{1}{\sqrt{\varepsilon}}g( X^{\varepsilon}_t, \mathscr{L}_{X^{\varepsilon}_t}, Y^{\varepsilon}_t)d W^{2}_t,
\end{array}\right.
\end{equation*}
where $\{W^{1}_t\}_{t\geq 0}$ and $\{W^{2}_t\}_{t\geq 0}$ are mutually independent $d_1$ and $d_2$-dimensional Brownian motions, was established by  R\"{o}ckner et al.~\cite{RSX1}, which
can be seen as the classical functional law of large numbers. For more averaging principle results, we refer the reader to e.g.~\cite{BM,C1,GP2,HLL2,K1,PIX,SXX} and reference therein for the classical SDEs or SPDEs (i.e.~distribution independent case) and to \cite{HLL4,SXW,XLLM} for the case of distribution dependence.

As the continuation of \cite{RSX1},  we are going one step further in the present paper. Namely, we will investigate the theory of diffusion approximation, which is closely relate to the averaging principle, for more general type of multi-scale McKean-Vlasov stochastic systems (\ref{eq4}) compared to \cite{RSX1}. As a powerful tool of characterizing the limiting process, some regularities on the Wasserstein space of solution to the auxiliary Poisson equation depending on  parameter measures are derived.  
Furthermore, it is worth  pointing out an independent interest of the paper. More precisely, we shall use two different methods to characterize the limiting process. One is the {\it martingale problem approach}, and the other is {\it martingale characteriziation}. As known to all, the former method is used frequently to characterize the limiting process is the solution of corresponding martingale problem (see \cite{PV1,PV2}), however it seems nontrivial to characterize the explicit form of limiting equation, which could be completed in a straightforward way by our second method.
 Here we give the details and present our strategy for the readers briefly.

 Firstly, we prove that $X^{\varepsilon}$ is tight in $C([0,T],\RR^n)$, then there exists a subsequence  of any sequence $\{\varepsilon_k\}_{k\geq1}$, which we keep denoting by $\{\varepsilon_{k}\}_{k\geq1}$, tending to $0$ such that $X^{\varepsilon_{k}}$ converges weakly to the limit denoted by $X$  in $C([0,T];\RR^n)$. Secondly, we identify the weak limiting process of sequence $\{X^{\varepsilon_{k}}\}_{k\geq 1}$ as $k \to \infty$.
To do this, we denote by $\Phi$ the solution of the following Poisson equation with measure dependence, i.e.,
$$
-\mathscr{L}_{2}(x,\mu)\Phi(x,\mu,y)=K(x,\mu,y),
$$
where the operator $\mathscr{L}_{2}(x,\mu)$ is generator of the frozen equation (see (\ref{FEQ2}) below) corresponding to the fast component of (\ref{eq4}). We shall use two mentioned methods above to present our main results.
 \begin{itemize}
\item{
\textbf{Martingale problem approach:} The limiting process is the solution of the
martingale problem associated to the operator $L_{\mu}$ given by
$$L_{\mu}:=\sum^n_{i=1}\Theta_i(x,\mu)\partial_{x_i}+\frac{1}{2}\sum^{n}_{i=1}\sum^n_{j=1} (\Sigma\Sigma^{\ast})_{ij}(x,\mu)\partial_{x_i}\partial_{x_j},$$
where
\begin{eqnarray*}
&&\Theta(x,\mu):=\overline{b+\partial_x \Phi_K+\partial_y \Phi_h+\text{Tr}\left[\partial^2_{xy}\Phi_{\sigma g^{*}}\right]}(x,\mu),\\
&&\Sigma(x,\mu):= \Big(\overline{(K\otimes\Phi)+(K\otimes\Phi)^{\ast}+(\sigma g^{*})\partial_{y}\Phi+[(\sigma g^{*})\partial_{y}\Phi]^{\ast}+(\sigma\sigma^{\ast})}\Big)^{\frac{1}{2}}(x,\mu).
\end{eqnarray*}
Then such limiting process is also the weak solution of following equation
\begin{eqnarray*}
dX_t= \!\!\!\!\!\!\!\!&& \Theta(X_t,\mathscr{L}_{X_t})dt +\Sigma(X_t,\mathscr{L}_{X_t})d\hat{W}_t,\quad X_0=\hat{\xi},~~~
\end{eqnarray*}
where initial value $\hat{\xi}$ coincides in law with $\xi$, $\hat{W}_t$ is a $n$-dimensional standard Brownian motion.

}
\end{itemize}

Note that there is a gap about the diffusion coefficient $\Sigma$, that is, it is unclear whether the term
$$\overline{(K\otimes\Phi)+(K\otimes\Phi)^{\ast}+(\sigma g^{*})\partial_{y}\Phi+[(\sigma g^{*})\partial_{y}\Phi]^{\ast}+(\sigma\sigma^{\ast})}$$ is positive semi-definite so that the square root makes sense. In order to fill this gap, we use the following second method to characterize the limiting process.
\begin{itemize}
\item{
\textbf{Martingale characterization:} The limiting process is the weak solution of following equation
\begin{eqnarray*}
dX_t= \!\!\!\!\!\!\!\!&&\Theta(X_t,\mathscr{L}_{X_t})dt +\tilde{\Sigma}(X_t,\mathscr{L}_{X_t})d\tilde{W}_t,\quad X_0=\tilde{\xi},~~~
\end{eqnarray*}
where initial value $\tilde{\xi}$ coincides in law with $\xi$, $\tilde{W}_t$ is a $n$-dimensional standard Brownian motion and
\begin{eqnarray*}
&&\Theta(x,\mu):=\overline{b+\partial_x \Phi_K+\partial_y \Phi_h+\text{Tr}\left[\partial^2_{xy}\Phi_{\sigma g^{*}}\right]}(x,\mu),\\
&&\tilde{\Sigma}(x,\mu):= \Big(\overline{(\partial_y\Phi_g+\sigma)(\partial_y\Phi_g+\sigma)^*}\Big)^{\frac{1}{2}}(x,\mu).
\end{eqnarray*}
}
\end{itemize}
The meaning of these notations in $\Theta, \Sigma, \tilde{\Sigma}$ can be founded in section 2.

Note that, by using these two methods, the obtained drift coefficients in the limiting equation have the same form but the diffusion coefficients are different, however it will be asserted that they are essential the same, see Remark \ref{re0} for the detailed explanations.

This manuscript is organized as follows. In section 2, we first introduce some notations and assumptions, then we state our main results more details through by two different methods, see Theorem \ref{main result 1} and Theorem \ref{main result 2} respectively. In section 3, we give some a priori estimates of the solution $(X^{\vare}, Y^{\vare})$ and study the regularity of the solution of Poisson equation. Finally,  the detailed proof of Theorem \ref{main result 1} is given in subsection 4.1, and the detailed proof of Theorem \ref{main result 2} is given in subsection 4.2. Note that throughout this paper $C$ and $C_T$  denote positive constants which may change from line to line, where the subscript $T$ is used to emphasize that the constant depends on certain parameter.

\section{Framework and main results}\label{sec.prelim}

We first recall some notations that will be frequently used throughout the present paper. We denote by $|\cdot|$ the Euclidean vector norm and by $\langle\cdot, \cdot\rangle$ the usual Euclidean inner product. Let $\|\cdot\|$ be the matrix norm or the operator norm if there is no confusion possible. For a vector-valued or matrix-valued function $u(x,y)$ defined on $\RR^n\times\RR^m$, for any $v,q\in \{x,y\}$, we use $\partial_v u$ to denote the first order partial derivative of $u$ with respect to (w.r.t.) the component $v$, and $\partial^2_{v q} u$ to denote its mixed second order derivatives of $u$ w.r.t. the components $v$ and $q$.

\vspace{0.1cm}
Let $\mathscr{P}(\RR^n)$ be the set of all probability measures on $(\RR^n, \mathscr{B}(\RR^n))$ and $\mathscr{P}_2(\RR^n)$ be
$$
\mathscr{P}_2(\RR^n):=\Big\{\mu\in \mathscr{P}(\RR^n): \mu(|\cdot|^2):=\int_{\RR^n}|x|^2\mu(dx)<\infty\Big\},
$$
 then $\mathscr{P}_2(\RR^n)$ is a Polish space under the Wasserstein distance
$$
\mathbb{W}_{2}(\mu_1,\mu_2):=\inf_{\pi\in \mathscr{C}_{\mu_1,\mu_2}}\left[\int_{\RR^n\times \RR^n}|x-y|^2\pi(dx,dy)\right]^{1/2},~ \mu_1,\mu_2\in\mathscr{P}_2(\RR^n),
$$
where $\mathscr{C}_{\mu_1,\mu_2}$ is the set of all couplings for $\mu_1$ and $\mu_2$.

In the sequel, we recall the notion of Lions derivative on Wasserstein space. For any $u: \mathscr{P}_2(\RR^n)\rightarrow \RR$, we denote by $U$ its "extension" to $L^2(\Omega, \PP;\RR^n)$, which is defined by
$$
U(X)=u(\mathscr{L}_{X}),\quad X\in L^2(\Omega,\PP;\RR^n).
$$
We say  $u$ is differentiable at $\mu\in\mathscr{P}_2(\RR^n)$ if there exists $X\in L^2(\Omega,\PP;\RR^n)$ such that $\mathscr{L}_{X}=\mu$ and $U$ is Fr\'echet differentiable at $X$. By the Riesz representation theorem, the Fr\'echet derivative $DU(X)$, which can be seen as an element of $L^2(\Omega,\PP;\RR^n)$, can be represented by
$$
DU(X)=\partial_{\mu}u(\mathscr{L}_{X})(X),
$$
where $\partial_{\mu}u(\mathscr{L}_{X}):\RR^n\rightarrow \RR^n$ is called the Lions derivative of $u$ at $\mu= \mathscr{L}_{X}$. Moreover, $\partial_{\mu}u(\mu)\in L^2(\mu;\RR^n)$ for $\mu\in\mathscr{P}_2(\RR^n)$. In addition, if $\partial_{\mu}u(\mu)(\cdot):\RR^n\rightarrow \RR^n$ is differentiable, we denote its derivative by $\partial_{z}\partial_{\mu}u(\mu)(\cdot):\RR^n\rightarrow \RR^n\times\RR^n$. In addition,
we say a vector-valued or matrix-valued function $u(\mu)=(u_{ij}(\mu))$ is differentiable at $\mu\in\mathscr{P}_2(\RR^n)$, if all its components are  differentiable at $\mu$, and set
$$ \partial_{\mu}u(\mu)=(\partial_{\mu}u_{ij}(\mu)), ~~   \|\partial_{\mu}u(\mu)\|^2_{L^2(\mu)}=\sum_{i,j}\int_{\RR^n}|\partial_{\mu}u_{ij}(\mu)(z)|^2\mu(dz). $$
 Similarly, we call that $\partial_{\mu}u(\mu)(\cdot)$ is differentiable if all its components are differentiable, and set $$\partial_{z}\partial_{\mu}u(\mu)(z)=(\partial_{z}\partial_{\mu}u_{ij}(\mu)(z)), ~~ \|\partial_{z}\partial_{\mu}u(\mu)(\cdot)\|^2_{L^2(\mu)}=\sum_{i,j}\int_{\RR^n}\|\partial_{z}\partial_{\mu}u_{ij}(\mu)(z)\|^2\mu(dz).$$

\vspace{0.1cm}
For sake of simplicity, we recall the following definitions.

 \begin{definition} For a  map $u(\cdot): \mathscr{P}_2(\RR^n) \to \RR$, we say $u\in C^{(1,1)}_b(\mathscr{P}_2(\RR^n); \RR)$, if this map is continuously differentiable at any $\mu\in\mathscr{P}_2(\RR^n)$ and its derivative $\partial_{\mu}u(\mu)(z):\RR^n\rightarrow \RR^n$ is continuously differentiable at any $z\in\RR^n$, moreover, the derivatives $\partial_{\mu}u(\mu)(z)$ and $\partial_z\partial_{\mu}u(\mu)(z)$ are jointly continuous at any $(\mu,z)$, and uniformly bounded,~i.e. $\sup_{\mu\in\mathscr{P}_2(\RR^n),z\in\RR^n}\|\partial_{\mu}u(\mu)(z)\|<\infty$ and $\sup_{\mu\in\mathscr{P}_2,z\in\RR^n}\|\partial_z\partial_{\mu}u(\mu)(z)\|<\infty$. For a vector or matrix-valued map $u(\cdot): \mathscr{P}_2(\RR^n) \to \mathbb{K}$, where $\mathbb{K}$ is a vector or matrix space, we say $u\in C^{(1,1)}_b(\mathscr{P}_2(\RR^n);\mathbb{K})$ if all the components belong to $C^{(1,1)}_b(\mathscr{P}_2(\RR^n);\RR)$.
\end{definition}

\begin{definition} We say $u\in C^{(2,1)}_b(\mathscr{P}_2(\RR^n); \RR)$, if $u\in C^{(1,1)}_b(\mathscr{P}_2(\RR^n); \RR)$ and its Lions derivative of order two $\partial_{\mu}\partial_{\mu}u(\mu)(z)(z')$ exists and is jointly continuous at any $(\mu,z,z')$ and uniformly bounded in $L^2(\mu)$-sense,~i.e. $\sup_{\mu\in\mathscr{P}_2(\RR^n),z\in\RR^n}\|\partial_{\mu}\partial_{\mu}u(\mu)(z)\|_{L^2(\mu)}<\infty$. For a vector or matrix-valued map we say $u\in C^{(2,1)}_b(\mathscr{P}_2(\RR^n);\mathbb{K})$ if all the components belong to $C^{(2,1)}_b(\mathscr{P}_2(\RR^n);\RR)$.
\end{definition}

\begin{definition}  For a  map $u(\cdot,\cdot): \RR^n\times\RR^m \to \RR$, we say $u\in C^{3,3}_b(\RR^n\times\RR^m; \RR)$, if all the $k$-order (mixed) partial derivatives of $u$ are jointly continuous and uniformly bounded w.r.t.~$(x,y)\in\RR^n\times\RR^m$, for $1\leq k\leq 3$. For a vector or matrix-valued map $u(\cdot,\cdot): \RR^n\times\RR^m \to \mathbb{K}$, we say $u\in C^{3,3}_b(\RR^n\times\RR^m;\mathbb{K})$ if all the components belong to $C^{3,3}_b(\RR^n\times\RR^m;\RR)$.
\end{definition}

\begin{definition} For a vector or matrix-valued map $u(\cdot,\cdot,\cdot): \RR^n\times\mathscr{P}_2(\RR^n) \times\RR^m \to \mathbb{K}$, we say $u\in C^{3,(2,1),3}_b(\RR^n\times\mathscr{P}_2(\RR^n)\times\RR^m;\mathbb{K})$ if $u(x,\cdot,y)\in C^{(2,1)}_b(\mathscr{P}_2(\RR^n);\mathbb{K})$ for any $(x,y)\in\RR^n\times \RR^m$ and $u(\cdot,\mu,\cdot)\in C^{3,3}_b(\RR^n\times\RR^m;\mathbb{K})$ for any $\mu\in\mathscr{P}_2(\RR^n)$, moreover, all the $k$-order (mixed) partial partial derivatives of $u$ w.r.t.~$(x,y)$, for $1\leq k\leq 3$, are uniformly bounded w.r.t.~$(x,\mu,y)\in\RR^n\times\mathscr{P}_2(\RR^n) \times\RR^m$, and $\partial_{\mu}u(x,\mu,y)(z)$, $\partial_z\partial_{\mu}u(x,\mu,y)(z)$ and $\|\partial_{\mu}\partial_{\mu}u(x,\mu,y)(z)\|_{L^2(\mu)}$ are uniformly bounded  w.r.t.~$(x,\mu,y,z)\in\RR^n\times\mathscr{P}_2(\RR^n) \times\RR^m\times \RR^n$.
\end{definition}

\vspace{0.1cm}
Suppose that the coefficients
\begin{eqnarray*}
&&b: \RR^n\times\mathscr{P}_2(\RR^n)\times\RR^m \rightarrow \RR^{n};\\
&&K:\RR^n\times\mathscr{P}_2(\RR^n)\times\RR^m \rightarrow \RR^{n};\\
&& \sigma: \RR^n\times \mathscr{P}_2(\RR^n)\times\RR^m\rightarrow \RR^{n}\otimes\RR^{d};\\
&&f:\RR^n\times\mathscr{P}_2(\RR^n)\times\RR^m\rightarrow \RR^{m};\\
&&h:\RR^n\times\mathscr{P}_2(\RR^n)\times\RR^m\rightarrow \RR^{m};\\
&&g:\RR^n\times\mathscr{P}_2(\RR^n)\times\RR^m\rightarrow \RR^{m}\otimes\RR^{d}
\end{eqnarray*}
satisfy the following assumptions.


\smallskip
\noindent
\begin{conditionA}\label{A1} There exist constants $C$ and $\gamma>0$ such that for all $x_1,x_2\in\RR^n$, $\mu_1,\mu_2\in \mathscr{P}_2(\RR^n)$ and $y_1,y_2\in\RR^m$,
\begin{eqnarray}
&&|b(x_1, \mu_1, y_1)-b(x_2, \mu_2, y_2)|+|K(x_1, \mu_1, y_1)-K(x_2, \mu_2, y_2)|\nonumber\\
&&+\|\sigma(x_1,\mu_1,y_1)-\sigma(x_2,\mu_2, y_2)\|\leq C\big[|x_1-x_2|+|y_1-y_2|+\mathbb{W}_{2}(\mu_1, \mu_2)\big]; \label{A11}
\end{eqnarray}
\begin{eqnarray}
\!\!\!\!\!\!&&|f(x_1,\mu_1, y_1)-f(x_2, \mu_2,y_2)|+|h(x_1, \mu_1,y_1)-h(x_2, \mu_2, y_2)|\nonumber\\
\!\!\!\!\!\!&&~~~+\|g(x_1, \mu_1,y_1)-g(x_2, \mu_2, y_2)\|\leq C\big[|x_1-x_2|+|y_1-y_2|+\mathbb{W}_{2}(\mu_1, \mu_2)\big];\label{A21}
\end{eqnarray}
%
\begin{eqnarray}
2\langle f(x,\mu, y_1)-f(x, \mu,y_2), y_1-y_2\rangle
+\|g(x, \mu,y_1)-g(x,\mu, y_2)\|^2
\leq -\gamma|y_1-y_2|^2.\label{sm}
\end{eqnarray}

\end{conditionA}

\smallskip
\noindent
\begin{conditionA}\label{A3}
There exist constants $C$ and $\varsigma\in(0,1)$ such that for all $y\in\RR^m$,
\begin{eqnarray}
&&\sup_{x\in\RR^n,\mu\in\mathscr{P}_2(\RR^n)}\big[|K(x,\mu,y)|\!+|f(x,\mu,y)|+\!|h(x,\mu,y)|+\!\|\sigma(x,\mu,y)\|\big]\leq \!C(1+|y|),\label{Bound Condition}\\
&&\sup_{x\in\RR^n,\mu\in\mathscr{P}_2(\RR^n)}\|g(x, \mu,y)\|\leq C(1+|y|^\varsigma).\label{A6}
\end{eqnarray}
\end{conditionA}

\smallskip
\noindent
\begin{conditionA}\label{A2}
Suppose that $K\in C_b^{3,(2,1),3}(\RR^n\times\mathscr{P}_2(\RR^n)\times\RR^m;\RR^n)$,  $f\in C^{3,(2,1),3}_b(\RR^n\times \mathscr{P}_2(\RR^n)\times\RR^m;\RR^m)$ and $g\in C^{3,(2,1),3}_b(\RR^n\times \mathscr{P}_2(\RR^n)\times\RR^m;\RR^{m}\otimes\RR^{d})$. Moreover, there exist constants $C>0$  and $\gamma_1\in (0,1]$ such that for all $y_1,y_2\in\RR^m$,
\begin{eqnarray*}
\sup_{x,z\in\RR^n,\mu\in\mathscr{P}_2(\RR^n)}\|\partial_{z}\partial_{\mu} F(x, \mu, y_1)(z)-\partial_{z}\partial_{\mu} F(x, \mu, y_2)(z)\|\leq C|y_1-y_2|^{\gamma_1} , \label{A44}
\end{eqnarray*}
\begin{eqnarray*}
\sup_{x\in\RR^n,\mu\in\mathscr{P}_2(\RR^n)}\|\partial_{\mu}\partial_{y} F(x, \mu, y_1)-\partial_{\mu}\partial_{y} F(x, \mu, y_2)\|_{L^2(\mu)}\leq C|y_1-y_2|^{\gamma_1} ,
\end{eqnarray*}
\begin{eqnarray*}
\sup_{x,z\in\RR^n,\mu\in\mathscr{P}_2(\RR^n)}\|\partial_{y}\partial_{\mu} F(x, \mu, y_1)(z)-\partial_{y}\partial_{\mu} F(x, \mu, y_2)(z)\|\leq C|y_1-y_2|^{\gamma_1} ,
\end{eqnarray*}
\begin{eqnarray*}
\sup_{x,z\in\RR^n,\mu\in\mathscr{P}_2(\RR^n)}\|\partial_{\mu}\partial_{\mu} F(x, \mu, y_1)(z)-\partial_{\mu}\partial_{\mu} F(x, \mu, y_2)(z)\|_{L^2(\mu)}\leq C|y_1-y_2|^{\gamma_1} ,
\end{eqnarray*}
\begin{eqnarray*}
\!\!\!\!\!\!\!\!\!\!\!\!\!\!\!\!&&\sup_{x,z\in\RR^n,y\in\RR^m,\mu\in\mathscr{P}_2(\RR^n)}\big[\|\partial_{\mu}\partial_y F(x,\mu,y)\|_{L^2(\mu)}+\|\partial_{y}\partial_{\mu} F(x,\mu,y)(z)\|
\nonumber\\
\!\!\!\!\!\!\!\!\!\!\!\!\!\!\!\!&&~~~~~~~~~~~~~~~~~~~~~~~~~~
+\|\partial_{\mu}\partial^2_{yx} F(x,\mu,y)\|_{L^2(\mu)}+\|\partial_{\mu}\partial^2_{yy} F(x,\mu,y)\|_{L^2(\mu)}\big]\leq C,
\end{eqnarray*}
where $F$ represents $K, f,g$ respectively.
\end{conditionA}

Note that condition \eref{sm} is a standard dissipative condition that guarantees the existence and uniqueness of invariant measures (denoted by $\nu^{x,\mu}$) to the frozen equation (\ref{FEQ2}) below, see subsection \ref{sec3.11} for details. Based on this, we further assume the following centering condition for the coefficient $K$.
\smallskip
\noindent
\begin{conditionA}\label{A4}
Suppose that $K$ satisfies the centering condition, i.e.,
$$\int_{\RR^m}K(x,\mu,y)\nu^{x,\mu}(dy)=0.$$
\end{conditionA}

Let us give some comments for the conditions $\ref{A1}$-$\ref{A4}$ for reader's understanding.

\begin{remark}\label{R2.1}
(i) If \eref{A11} and \eref{A21} in condition $\ref{A1}$ hold, the system (\ref{eq4}) admits a unique  solution $(X_t^\varepsilon,Y_t^\varepsilon)$ (cf.~\cite{RSX1}), i.e., for any $\varepsilon>0$, $\mathscr{F}_0$-measurable initial values $\xi\in L^2(\Omega;\RR^n), \zeta\in L^2(\Omega;\RR^m)$, there exists a unique solution $\{(X^{\varepsilon}_t,Y^{\varepsilon}_t)\}_{t\geq 0}$ to system \eref{eq4} such that for any $T>0$,
$(X^{\varepsilon},Y^{\varepsilon})\in C([0,T]; \RR^n)\times C([0,T]; \RR^m), \PP$-a.s.,
\begin{equation*}\left\{\begin{array}{l}\label{mild solution}
\displaystyle
X^{\varepsilon}_t=\xi+\int^t_0b(X^{\varepsilon}_s, \mathscr{L}_{X^{\varepsilon}_s}, Y^{\varepsilon}_s)ds+\frac{1}{\sqrt{\varepsilon}}\int^t_0 K(X^{\varepsilon}_s, \mathscr{L}_{X^{\varepsilon}_s}, Y^{\varepsilon}_s)ds +\int^t_0\sigma(X^{\varepsilon}_s,\mathscr{L}_{X^{\varepsilon}_s},Y^{\varepsilon}_s)d W_s,\vspace{2mm}\\
\displaystyle
Y^{\varepsilon}_t=\zeta+\frac{1}{\varepsilon} \int^t_0f(X^{\varepsilon}_s,\mathscr{L}_{X^{\varepsilon}_s},Y^{\varepsilon}_s)ds+\frac{1}{\sqrt{\varepsilon}}\int^t_0 h( X^{\varepsilon}_s, \mathscr{L}_{X^{\varepsilon}_s}, Y^{\varepsilon}_s)ds
+\frac{1}{\sqrt{\varepsilon}} \int^t_0 g(X^{\varepsilon}_s,\mathscr{L}_{X^{\varepsilon}_s}, Y^{\varepsilon}_s)dW_s.
\end{array}\right.
\end{equation*}

(ii)  \eref{Bound Condition} and \eref{A6} in condition $\ref{A3}$ are used to ensure the solution $(X_t^\varepsilon,Y_t^\varepsilon)$ have finite $k$-th moment, for any $k\in\mathbb{N}_{+}$ (see \eref{X1} and \eref{Y0} below). In fact, the $18$-th moment of $Y_t^\varepsilon$ is enough for proving our main result, thus (\ref{A6}) could be replaced by the following condition,
\begin{eqnarray*}
2\langle f(x,\mu, y_1)-f(x, \mu,y_2), y_1-y_2\rangle
+17\|g(x, \mu,y_1)-g(x,\mu, y_2)\|^2\leq -\gamma|y_1-y_2|^2.
\end{eqnarray*}

(iii) Condition $\ref{A2}$ is mainly used to investigate the regularity of solutions of Poisson equation (cf.~\cite[Proposition 4.1]{RSX1}). Condition $\ref{A4}$ is a necessary condition when studying the diffusion approximation for the multi-scale stochastic system  on the whole space (see \cite{PV1, PV2}). We remark that such kind of condition is also natural and similar to the centering condition in the theory of central limit theorem for multi-scale system.

\end{remark}

\subsection{Main results}
For a given vector-valued or matrix-valued map
$$F: \RR^n\times \mathscr{P}_2(\RR^n)\times \RR^m\rightarrow\mathbb{K},$$
we denote its averaged map by
\begin{equation*}\label{14}
\overline{F}(x,\mu)=\int_{\RR^m} F(x,\mu,y)\nu^{x,\mu}(dy),\nonumber
\end{equation*}
where $\nu^{x,\mu}$ is the unique invariant probability measure of transition semigroup to the frozen equation (\ref{FEQ2}) below.

In order to present our main results, we introduce the following Poisson equation with measure dependence,
\begin{equation}
-\mathscr{L}_{2}(x,\mu)\Phi(x,\mu,y)=K(x,\mu,y),\label{PE}
\end{equation}
where $\Phi(x,\mu,y):=(\Phi^1(x,\mu,y),\ldots, \Phi^n(x,\mu,y))$,
$$\mathscr{L}_{2}(x,\mu)\Phi(x,\mu,y):=(\mathscr{L}_{2}(x,\mu)\Phi^1(x,\mu,y),\ldots, \mathscr{L}_{2}(x,\mu)\Phi^n(x,\mu,y))$$
and for any $k=1,\ldots,n,$
\begin{eqnarray*}
\mathscr{L}_{2}(x,\mu)\Phi^k(x,\mu,y):=\langle f(x,\mu,y), \partial_y \Phi^k(x,\mu,y)\rangle+\frac{1}{2}\text{Tr}\left[(g g^{*})(x,\mu,y)\partial^2_{yy} \Phi^k(x,\mu,y)\right].
\end{eqnarray*}
In section \ref{sec3.11}, we will verify that, under the conditions $\ref{A1}$-$\ref{A4}$, Poisson equation (\ref{PE}) admits a  solution $\Phi(x,\mu,y)$ satisfying $\Phi(\cdot,\mu,\cdot)\in C^{3,3}(\RR^n\times\RR^m;\RR^n)$ and $\Phi(x,\cdot,y)\in C^{(2,1)}(\mathscr{P}_2(\RR^n); \RR^n)$.

\vspace{0.2cm}
Moreover, we give the following notations for simplification.
\begin{eqnarray*}
&&\partial_x \Phi_K(x,\mu,y):=(\partial_x \Phi^i_K(x,\mu,y))_{\{1\leq i\leq n\}}:=(\langle\partial_x \Phi^i(x,\mu,y), K(x,\mu,y)\rangle)_{\{1\leq i\leq n\}};\\
&&\partial_y \Phi_h(x,\mu,y):=(\partial_y \Phi^i_h(x,\mu,y))_{\{1\leq i\leq n\}}:=(\langle\partial_y \Phi^i(x,\mu,y), h(x,\mu,y)\rangle)_{\{1\leq i\leq n\}};\\
&&\partial_y \Phi_{g}(x,\mu,y):=\left((\partial_y \Phi_{g})_{ij}(x,\mu,y)\right)_{\{1\leq i\leq n,1\leq j\leq d\}}\\
&&~\quad\quad\quad\quad\quad:=\left(\sum^{m}_{k=1}\partial_{y_k}\Phi^i(x,\mu,y)g_{kj}(x,\mu,y)\right)_{\{1\leq i\leq n,1\leq j\leq d\}}\\
&&\text{Tr}\left[\partial^2_{xy}\Phi_{\sigma g^{*}}\right](x,\mu,y):=\left(\text{Tr}\left[\partial^2_{xy}\Phi^i_{\sigma g^{*}}\right](x,\mu,y)\right)_{\{1\leq i\leq n\}}\\
&&\quad\quad\quad\quad\quad\quad\quad\quad\quad:=\left(\text{Tr}\left[(\sigma g^{*})(x,\mu,y)\partial^2_{xy}\Phi^i(x,\mu,y)\right]\right)_{\{1\leq i\leq n\}};
\end{eqnarray*}
\begin{eqnarray*}
&&(K\otimes\Phi)(x,\mu,y):=((K\otimes\Phi)_{ij}(x,\mu))_{\{1\leq i,j\leq n\}}:=\left(K_i(x,\mu,y)\Phi^j(x,\mu,y)\right)_{\{1\leq i,j\leq n\}};\\
&&\left[(\sigma g^{*})\partial_{y}\Phi\right](x,\mu,y):=\left(\left[(\sigma g^{*})\partial_{y}\Phi\right]_{ij}(x,\mu,y)\right)_{\{1\leq i,j\leq n\}}\\
&&\quad\quad\quad\quad\quad\quad\quad\quad:=\left(\sum^m_{k=1}(\sigma g^{*})_{ik}(x,\mu,y)\partial_{y_k}\Phi^j(x,\mu,y)\right)_{\{1\leq i,j\leq n\}};\\
&&\quad\quad\quad\quad\quad\quad\quad\quad:=\left(\sum^m_{k=1}\sum^d_{l=1}\sigma_{il} g^{*}_{lk}(x,\mu,y)\partial_{y_k}\Phi^j(x,\mu,y)\right)_{\{1\leq i,j\leq n\}};\\
&&(\sigma\sigma^{\ast})(x,\mu,y):=\left((\sigma\sigma^{\ast})_{ij}(x,\mu,y)\right)_{\{1\leq i,j\leq n\}}:=\left(\sum^{d}_{k=1}\sigma_{ik}(x,\mu,y)\sigma^{\ast}_{kj}(x,\mu,y)\right)_{\{1\leq i,j\leq n\}}.
\end{eqnarray*}

The following is our first main result.

\begin{theorem}\label{main result 1}
Suppose that conditions  $\ref{A1}$-$\ref{A4}$ hold.
If the initial values $\xi\in L^{9}(\Omega;\RR^n), \zeta\in L^{18}(\Omega;\RR^m)$.
Then as $\varepsilon\to 0$,  $\{X^{\varepsilon}\}_{\varepsilon>0}$ converges weakly in $C([0,T];\RR^n)$  to the solution of following equation
\begin{eqnarray}\label{eq5}
dX_t= \!\!\!\!\!\!\!\!&& \Theta(X_t,\mathscr{L}_{X_t})dt +\Sigma(X_t,\mathscr{L}_{X_t})d\hat{W}_t,\quad X_0=\xi,~~~
\end{eqnarray}
where $\hat{W}_t$ is a $n$-dimensional standard Brownian motion and
\begin{eqnarray*}
&&\Theta(x,\mu):=\overline{b+\partial_x \Phi_K+\partial_y \Phi_h+\text{Tr}\left[\partial^2_{xy}\Phi_{\sigma g^{*}}\right]}(x,\mu),\\
&&\Sigma(x,\mu):= \Big(\overline{(K\otimes\Phi)+(K\otimes\Phi)^{\ast}+(\sigma g^{*})\partial_{y}\Phi+[(\sigma g^{*})\partial_{y}\Phi]^{\ast}+(\sigma\sigma^{\ast})}\Big)^{\frac{1}{2}}(x,\mu).
\end{eqnarray*}

\end{theorem}
\begin{remark}
As we mentioned before, there is a gap about the diffusion coefficient $\Sigma$. More precisely, it is obviously that the matrix $$\overline{(K\otimes\Phi)+(K\otimes\Phi)^{\ast}+(\sigma g^{*})\partial_{y}\Phi+[(\sigma g^{*})\partial_{y}\Phi]^{\ast}+(\sigma\sigma^{\ast})}$$ is symmetric, then its square root makes sense if it is also positive semi-definite, which seems nontrivial and is necessary to be asserted. From another view, we shall use a different method to describe the weakly convergence process of $\{X^{\varepsilon}\}_{\varepsilon>0}$ in $C([0,T];\RR^n)$, whose diffusion coefficient $\tilde{\Sigma}$ (see Theorem \ref{main result 2} below) is the square root of a positive semi-definite symmetry matrix, thus the gap will be filled.
\end{remark}

\vspace{0.2cm}

Now we state our second main result.
\begin{theorem}\label{main result 2}
Suppose that conditions  $\ref{A1}$-$\ref{A4}$ hold.
If the initial values $\xi\in L^{9}(\Omega;\RR^n), \zeta\in L^{18}(\Omega;\RR^m)$.
Then as $\varepsilon\to 0$,  $\{X^{\varepsilon}\}_{\varepsilon>0}$ converges weakly in $C([0,T];\RR^n)$  to the solution of following equation
\begin{eqnarray}\label{eq62}
dX_t= \!\!\!\!\!\!\!\!&&\Theta(X_t,\mathscr{L}_{X_t})dt +\tilde{\Sigma}(X_t,\mathscr{L}_{X_t})d\tilde{W}_t,\quad X_0=\xi,~~~
\end{eqnarray}
where $\tilde{W}_t$ is a $n$-dimensional standard Brownian motion  and
\begin{eqnarray*}
&&\Theta(x,\mu):=\overline{b+\partial_x \Phi_K+\partial_y \Phi_h+\text{Tr}\left[\partial^2_{xy}\Phi_{\sigma g^{*}}\right]}(x,\mu),\\
&&\tilde{\Sigma}(x,\mu):= \Big(\overline{(\partial_y\Phi_g+\sigma)(\partial_y\Phi_g+\sigma)^*}\Big)^{\frac{1}{2}}(x,\mu).
\end{eqnarray*}
\end{theorem}

\begin{remark}\label{re0}
It is necessary to take up some space to explain the relation between the diffusion coefficients $\Sigma$ and $\tilde{\Sigma}$. In fact $\Sigma\equiv\tilde{\Sigma}$, which is equivalent to
\begin{eqnarray}
\overline{(\partial_y\Phi_g)(\partial_y\Phi_g)^{\ast}}(x,\mu)=\overline{(K\otimes\Phi)+(K\otimes\Phi)^{\ast}}(x,\mu)\label{Assert1}
\end{eqnarray}
and
\begin{eqnarray}
\overline{(\partial_y\Phi_g)\cdot \sigma^{\ast}+\sigma\cdot(\partial_y\Phi_g)^{\ast}}(x,\mu)=\overline{(\sigma g^{*})\partial_{y}\Phi+[(\sigma g^{*})\partial_{y}\Phi]^{\ast}}(x,\mu).\label{Assert2}
\end{eqnarray}
By the definition, it is easy to see that $\sigma\cdot(\partial_y\Phi_g)^{\ast}(x,\mu,y)=(\sigma g^{*})\partial_{y}\Phi(x,\mu,y)$, hence \eref{Assert2} holds.

Next, we show \eref{Assert1} holds. Define $\{Y^{x,\mu,\eta}_t\}_{t\geq 0}$ is the solution of the frozen equation (see \eref{FEQ2} below) with random initial value $Y^{x,\mu,\eta}_0=\eta$, whose distribution is the invariant measure $\nu^{x,\mu}$. Then it is clear that $\{Y^{x,\mu,\eta}_t\}_{t\geq 0}$ is the stationary process. For any $1\leq i,j\leq n$, by It\^{o}'s formula, we have for any $t>0$,
\begin{eqnarray}
\Phi^i(x,\mu, Y^{x,\mu,\eta}_t)\Phi^j(x,\mu, Y^{x,\mu,\eta}_t)=\!\!\!\!\!\!\!\!&&\Phi^i(x,\mu, \eta)\Phi^j(x,\mu, \eta)\nonumber\\
&&+\int^t_0 \mathscr{L}_{2}(x,\mu)\Phi^i(x,\mu,Y^{x,\mu,\eta}_s) \Phi^j(x,\mu,Y^{x,\mu,\eta}_s)ds\nonumber\\
&&+\int^t_0 \mathscr{L}_{2}(x,\mu)\Phi^j(x,\mu,Y^{x,\mu,\eta}_s) \Phi^i(x,\mu,Y^{x,\mu,\eta}_s)ds\nonumber\\
&&+\int^t_0\left[(\partial_y\Phi_g)(\partial_y\Phi_g)^{\ast}\right]_{ij}(x,\mu,Y^{x,\mu,\eta}_s)ds+M_t,\label{remarkF1}
\end{eqnarray}
where $M_t$ is a local martingale define by
\begin{eqnarray*}
M_t=\!\!\!\!\!\!\!\!&&\int^t_0 \langle \partial_y\Phi^i(x,\mu,Y^{x,\mu,\eta}_s)\Phi^j(x,\mu,Y^{x,\mu,\eta}_s)\\
&&\quad\quad\quad+ \Phi^i(x,\mu,Y^{x,\mu,\eta}_s)\partial_y \Phi^j(x,\mu,Y^{x,\mu,\eta}_s), g(x,\mu,Y^{x,\mu,\eta}_s)dW_s\rangle.
\end{eqnarray*}
Note that
\begin{eqnarray*}
&&\mathscr{L}_{2}(x,\mu)\Phi^i(x,\mu,y)=-K^i(x,\mu,y),\\
&&\mathscr{L}_{2}(x,\mu)\Phi^j(x,\mu,y)=-K^j(x,\mu,y).
\end{eqnarray*}
Then by taking expectation on both sides in \eref{remarkF1}, we will obtain that
\begin{eqnarray*}
\left[\overline{(K\otimes\Phi)+(K\otimes\Phi)^{\ast}}\right]_{ij}(x,\mu)=\left[\overline{(\partial_y\Phi_g)(\partial_y\Phi_g)^{\ast}}\right]_{ij}(x,\mu).
\end{eqnarray*}
Hence, \eref{Assert1} holds.

\end{remark}

\section{Poisson equation, a priori estimates and tightness}

In this section, we present some necessary regularities of Poisson equation (\ref{PE}) which play an important role in our analysis firstly. Secondly, the uniform estimates and increment estimates of time for  $X^\varepsilon$ to (\ref{eq4}) are obtained. Finally, we study the tightness of the solution $\{X^\varepsilon\}_{\varepsilon>0}$ to equation (\ref{eq4}) in $C([0,T],\RR^n)$.
\subsection{Regularity of Poisson equation}\label{sec3.11}

Recall the frozen equation for any fixed $x\in\RR^n$ and $\mu\in\mathscr{P}_2(\RR^n)$,
\begin{equation}\left\{\begin{array}{l}\label{FEQ2}
\displaystyle
dY_{t}=f(x, \mu, Y_{t})dt+g(x, \mu, Y_{t})dW_{t},\\
Y_{0}=y\in\RR^m.\\
\end{array}\right.
\end{equation}
In view of \eref{A21} and \eref{sm}, it is clear that
$(\ref{FEQ2})$ admits a unique strong solution $\{Y_{t}^{x,\mu,y}\}_{t\geq 0}$, furthermore, it admits a unique invariant probability measure denoted by $\nu^{x,\mu}$ (cf.~e.g.~\cite[Theorem 4.3.9]{LR1}).


The following regularity of the solution of Poisson equation (\ref{PE}), which plays an important role in the proof of our main results.

\begin{proposition}\label{P3.6}
Suppose that $\ref{A1}$-$\ref{A4}$ hold. Poisson equation \eref{PE} admits a solution
\begin{eqnarray}
\Phi(x,\mu,y):=\int^{\infty}_{0} \EE K(x,\mu,Y^{x,\mu,y}_t)dt,\label{SPE}
\end{eqnarray}
which satisfies that $\Phi(\cdot,\mu,\cdot)\in C^{3,3}(\RR^n\times\RR^m;\RR^n)$ and $\Phi(x,\cdot,y)\in C^{(2,1)}(\mathscr{P}_2(\RR^n); \RR^n)$.  Moreover, there exists $C>0$ such that for any $x,z\in\RR^n,\mu\in\mathscr{P}_2(\RR^n)$ and $y\in\RR^m$,
\begin{eqnarray}
\!\!\!\!\!\!\!\!\!\!\!\!\!\!\!\!&&\max\left\{|\Phi(x,\mu,y)|,\|\partial_x \Phi(x,\mu,y)\|, \|\partial_{\mu}\Phi(x,\mu,y)\|_{L^2(\mu)},\|\partial^2_{xx} \Phi(x,\mu,y)\|,\right.\nonumber\\
\!\!\!\!\!\!\!\!\!\!\!\!\!\!\!\!&& \quad\quad\quad\left. |\partial_{z}\partial_{\mu}\Phi(x,\mu,y)(z)|, \|\partial_{\mu}\partial_{\mu}\Phi(x,\mu,y)(z)\|_{L^2(\mu)}\right \}\leq C\big(1+|y|\big);\label{E1}\\
\!\!\!\!\!\!\!\!\!\!\!\!\!\!\!\!&&\max\Big\{\|\partial_y \Phi(x,\mu,y)\|,\|\partial^2_{yy} \Phi(x,\mu,y)\|,\|\partial^2_{xy} \Phi(x,\mu,y)\|,\|\partial_{\mu}\partial_y\Phi(x,\mu,y)\|_{L^2(\mu)}\Big\}\leq C;\label{E2}\\
\!\!\!\!\!\!\!\!\!\!\!\!\!\!\!\!&&\max\Big\{\|\partial^3_{xyx} \Phi(x,\mu,y)\|,
\|\partial^3_{yyx} \Phi(x,\mu,y)\|,\|\partial_{\mu}\partial^2_{yx} \Phi(x,\mu,y)\|_{L^2(\mu)}\Big\}\leq C.\label{E4}
\end{eqnarray}

\end{proposition}

\begin{proof}
The proof is formulated in subsection \ref{app1.1} in the Appendix.
\end{proof}

 \subsection{Some a priori estimates}
In this subsection, we establish some preliminary results for the moment estimates of solution $(X_{t}^{\varepsilon}, Y_{t}^{\varepsilon})$ to the stochastic system (\ref{eq4}). To do this, we need the following fluctuation estimate.
\begin{lemma}\label{lem1}$\mathbf{(Fluctuation~estimates)}$
Assume that the conditions in Theorem \ref{main result 1} hold. For any $T>0$ and $\EE|\xi|^{3p/2}+\EE|\zeta|^{3p}<\infty$ for some $p\geq 2$, then there exists a constant $C_{p,T}>0$ such that for small enough $\vare>0$,
\begin{equation}\label{sfe}
\EE\left[\sup_{t\in[0,T]}\left|\int_0^tK(X_{s}^{\varepsilon}, \mathscr{L}_{X_{s}^{\varepsilon}},Y_{s}^{\varepsilon})ds\right|^p\right]\leq C_T\varepsilon^{p/2}\left(1+\EE|\xi|^{3p/2}+\EE|\zeta|^{3p}\right).
\end{equation}
\end{lemma}

\begin{proof}
Recall the Poisson equation (\ref{PE}), by It\^{o}'s formula (cf.~\cite[Theorem 7.1]{BLPR}) for the function $\Phi(x,\mu,y)$, it leads to
\allowdisplaybreaks
\begin{eqnarray*}
&&\Phi(X_{t}^{\varepsilon},\mathscr{L}_{X^{\varepsilon}_{t}},Y^{\varepsilon}_{t})\\
=\!\!\!\!\!\!\!\!&&\Phi(\xi,\mathscr{L}_{\xi},\zeta)
+\int^t_0 \EE\left[b(X^{\varepsilon}_s,\mathscr{L}_{ X^{\varepsilon}_{s}}, Y^{\varepsilon}_s)\partial_{\mu}\Phi(x,\mathscr{L}_{X^{\varepsilon}_{s}},y)(X^{\varepsilon}_s)\right]\mid_{\{x=X_{s}^{\varepsilon},y=Y^{\varepsilon}_{s}\}}ds\\
&&+\frac{1}{\sqrt{\varepsilon}}\int^t_0\EE\left[K(X^{\varepsilon}_s,\mathscr{L}_{ X^{\varepsilon}_{s}}, Y^{\varepsilon}_s)\partial_{\mu}\Phi(x,\mathscr{L}_{X^{\varepsilon}_{s}},y)(X^{\varepsilon}_s)\right]\mid_{\{x=X_{s}^{\varepsilon},y=Y^{\varepsilon}_{s}\}}ds\\
&&+\int^t_0 \frac{1}{2}\EE \text{Tr}\left[\sigma\sigma^{*}(X^{\varepsilon}_s,\mathscr{L}_{ X^{\varepsilon}_{s}},Y^{\varepsilon}_s)\partial_z\partial_{\mu}\Phi(x,\mathscr{L}_{X^{\varepsilon}_{s}},y)(X^{\varepsilon}_s)\right]\mid_{\{x=X_{s}^{\varepsilon},y=Y^{\varepsilon}_{s}\}}ds\\
&&+\int^t_0 \mathscr{L}_{1}(\mathscr{L}_{X^{\varepsilon}_{s}},Y^{\varepsilon}_{s})\Phi(X_{s}^{\varepsilon},\mathscr{L}_{X^{\varepsilon}_{s}},Y^{\varepsilon}_{s})ds+\frac{1}{\sqrt{\varepsilon}}\int^t_0 \partial_x \Phi_{K}(X_{s}^{\varepsilon},\mathscr{L}_{X^{\varepsilon}_{s}},Y^{\varepsilon}_{s})ds\\
&&+\frac{2}{\sqrt{\varepsilon}}\int^t_0 \text{Tr}\left[\partial^2_{xy}\Phi_{\sigma g^{*}}\right](X^{\varepsilon}_{s},\mathscr{L}_{X^{\varepsilon}_{s}},Y^{\varepsilon}_{s})ds+\frac{1}{\sqrt{\varepsilon}}\int^t_0  \partial_y\Phi_{h}(X_{s}^{\varepsilon},\mathscr{L}_{X^{\varepsilon}_{s}},Y^{\varepsilon}_{s})ds      \\
&&+\frac{1}{\varepsilon}\int^t_0 \mathscr{L}_{2}(X_{s}^{\varepsilon},\mathscr{L}_{X^{\varepsilon}_{s}})\Phi(X_{s}^{\varepsilon},\mathscr{L}_{X^{\varepsilon}_{s}},Y^{\varepsilon}_{s})ds+M^{1,\varepsilon}_t+\frac{1}{\sqrt{\varepsilon}}M^{2,\varepsilon}_t,
\end{eqnarray*}
where $\mathscr{L}_{1}(\mu,y)\Phi(x,\mu,y):=(\mathscr{L}_{1}(\mu,y)\Phi^{i}(x,\mu,y))_{\{1\leq i\leq n\}}$ with

\begin{eqnarray*}
 \mathscr{L}_{1}(\mu,y)\Phi^i(x,\mu,y):=\!\!\!\!\!\!\!\!&&\left\langle b(x,\mu,y), \partial_x \Phi^{i}(x,\mu,y)\right \rangle\\
 &&+\frac{1}{2}\text{Tr}\left[\sigma\sigma^{*}(x,\mu,y)\partial^2_{xx} \Phi^i(x,\mu,y)\right ],\quad i=1,\ldots, n,
\end{eqnarray*}
and
$M^{1,\varepsilon}_t$ and $M^{2,\varepsilon}_t$ are $\RR^n$-valued local martingales, which are defined by
\begin{eqnarray}
&&M^{1,\varepsilon}_t=\int^t_0 \partial_x \Phi_{\sigma}(X_{s}^{\varepsilon},\mathscr{L}_{X^{\varepsilon}_{s}},Y_{s}^{\varepsilon})dW_s,\label{ma1}\\
&&M^{2,\varepsilon}_t=\int^t_0 \partial_y \Phi_{g}(X_{s}^{\varepsilon},\mathscr{L}_{X^{\varepsilon}_{s}},Y_{s}^{\varepsilon})dW_s.\label{ma2}
\end{eqnarray}
Then it is straightforward that for any $p\geq 2$,
\begin{eqnarray}\label{es1}
\!\!\!\!\!\!\!\!&&\EE\left[\sup_{t\in[0,T]}\left|\int_0^tK(X_{s}^{\varepsilon}, \mathscr{L}_{X_{s}^{\varepsilon}},Y_{s}^{\varepsilon})ds\right|^p\right]\nonumber\\
\leq\!\!\!\!\!\!\!\!&& \varepsilon^p\EE\Bigg\{\sup_{t\in[0,T]}\Big|\Phi(\xi,\mathscr{L}_{\xi},\zeta)-\Phi(X_{t}^{\varepsilon},\mathscr{L}_{X^{\varepsilon}_{t}},Y^{\varepsilon}_{t})\nonumber\\
\!\!\!\!\!\!\!\!&&
+\int^t_0 \EE\left[b(X^{\varepsilon}_s,\mathscr{L}_{ X^{\varepsilon}_{s}}, Y^{\varepsilon}_s)\partial_{\mu}\Phi(x,\mathscr{L}_{X^{\varepsilon}_{s}},y)(X^{\varepsilon}_s)\right]\mid_{\{x=X_{s}^{\varepsilon},y=Y^{\varepsilon}_{s}\}}ds\nonumber\\
\!\!\!\!\!\!\!\!&&+\int^t_0 \frac{1}{2}\EE \text{Tr}\left[\sigma\sigma^{*}(X^{\varepsilon}_s,\mathscr{L}_{ X^{\varepsilon}_{s}},Y^{\varepsilon}_s)\partial_z\partial_{\mu}\Phi(x,\mathscr{L}_{X^{\varepsilon}_{s}},y)(X^{\varepsilon}_s)\right]\mid_{\{x=X_{s}^{\varepsilon},y=Y^{\varepsilon}_{s}\}}ds\nonumber\\
\!\!\!\!\!\!\!\!&&+\int^t_0 \mathscr{L}_{1}(\mathscr{L}_{X^{\varepsilon}_{s}},Y^{\varepsilon}_{s})\Phi(X_{s}^{\varepsilon},\mathscr{L}_{X^{\varepsilon}_{s}},Y^{\varepsilon}_{s})ds\Big|^p\Bigg\} +\varepsilon^p\EE\left\{\sup_{t\in[0,T]}\big| M^{1,\varepsilon}_t\big|^p\right\}\nonumber\\
\!\!\!\!\!\!\!\!&&
+\varepsilon^{p/2}\EE\left\{\sup_{t\in[0,T]}\Big|\int^t_0\EE\left[K(X^{\varepsilon}_s,\mathscr{L}_{ X^{\varepsilon}_{s}}, Y^{\varepsilon}_s)\partial_{\mu}\Phi(x,\mathscr{L}_{X^{\varepsilon}_{s}},y)(X^{\varepsilon}_s)\right]\mid_{\{x=X_{s}^{\varepsilon},y=Y^{\varepsilon}_{s}\}}ds\Big|^p\right\}\nonumber\\
\!\!\!\!\!\!\!\!&&
+\varepsilon^{p/2}\EE\left\{\sup_{t\in[0,T]}\Big|\int^t_0\partial_x \Phi_K(X_{s}^{\varepsilon},\mathscr{L}_{X^{\varepsilon}_{s}},Y^{\varepsilon}_{s})ds\Big|^p\right\}+\varepsilon^{p/2}\EE\left\{\sup_{t\in[0,T]}\Big|\int^t_0\partial_y \Phi_h(X_{s}^{\varepsilon},\mathscr{L}_{X^{\varepsilon}_{s}},Y^{\varepsilon}_{s})ds\Big|^p\right\}\nonumber\\
\!\!\!\!\!\!\!\!&&+C\varepsilon^{p/2}\EE\left\{\sup_{t\in[0,T]}\Big|\int^t_0 \text{Tr}\left[\partial^2_{xy}\Phi_{\sigma g^{*}}\right](X^{\varepsilon}_{s},\mathscr{L}_{X^{\varepsilon}_{s}},Y^{\varepsilon}_{s})ds\Big|^p\right\}+\varepsilon^{p/2}\EE\left\{\sup_{t\in[0,T]}\big|M^{2,\varepsilon}_t\big|^p\right\}
\nonumber\\
=:\!\!\!\!\!\!\!\!&&\sum_{i=1}^{7} \mathscr{R}_i^\varepsilon.
\end{eqnarray}

Regarding the term $b+\frac{1}{\sqrt{\varepsilon}}K$ as the whole drift term, by a small modification in \cite[Lemma 3.1]{RSX1}, under the conditions $\ref{A1}$-$\ref{A2}$, we can easily prove that for any $p\geq 2$, $T>0$, there exists a constant $ C_{p,T}>0$ such that
\begin{eqnarray}
\mathbb{E}\left\{\sup_{t\in [0, T]}|X_{t}^{\varepsilon}|^{p}\right\}\leq \frac{C_{p,T}(1+\EE|\xi|^{p}+\EE|\zeta|^{p})}{\varepsilon^{p/2}}.\label{X1}
\end{eqnarray}
Furthermore, there exists a constant $ C_{p}>0$ such that for any small enough $\ep>0$,
\begin{eqnarray}
&&\sup_{t\geq 0}\mathbb{E}|Y_{t}^{\varepsilon}|^{p}\leq C_p(1+\EE|\zeta|^{p}),\label{Y0}\\
&&\mathbb{E}\left\{\sup_{t\in [0, T]}|Y_{t}^{\varepsilon}|^{p}\right\}\leq \frac{C_p(1+\EE|\zeta|^{p})T}{\ep},\label{Y1}
\end{eqnarray}
whose proof will be given in the subsection \ref{appendix 1} in Appendix.

For the term $\mathscr{R}_1^\varepsilon$. By (\ref{E1}) and (\ref{Y1}), we have
\begin{eqnarray*}
 \EE\left\{\sup_{t\in[0,T]}\left|\Phi(X_{t}^{\varepsilon},\mathscr{L}_{X^{\varepsilon}_{t}},Y^{\varepsilon}_{t})\right|^p\right\}\leq C \EE\left\{\sup_{t\in[0,T]}\left(1+\big|Y^{\varepsilon}_{t}\big|^p\right)\right\}\leq \frac{C_{p,T}\left(1+\EE|\zeta|^{p}\right)}{\varepsilon},
\end{eqnarray*}
and by (\ref{A11}), (\ref{E1}), \eref{X1} and (\ref{Y0}), it follows that
\begin{eqnarray*}
 \!\!\!\!\!\!\!\!&&\EE\left\{\sup_{t\in[0,T]}\Big|\int^t_0 \EE\left[b(X^{\varepsilon}_s,\mathscr{L}_{ X^{\varepsilon}_{s}}, Y^{\varepsilon}_s)\partial_{\mu}\Phi(x,\mathscr{L}_{X^{\varepsilon}_{s}},y)(X^{\varepsilon}_s)\right]\mid_{\{x=X_{s}^{\varepsilon},y=Y^{\varepsilon}_{s}\}}ds\Big|^p\right\}
 \nonumber\\
  \leq\!\!\!\!\!\!\!\!&&\EE\left\{\sup_{t\in[0,T]}\Big|\int^t_0 \left[\EE|b(X^{\varepsilon}_s,\mathscr{L}_{ X^{\varepsilon}_{s}}, Y^{\varepsilon}_s)|^2\right]^{1/2}\left[\EE\|\partial_{\mu}\Phi(x,\mathscr{L}_{X^{\varepsilon}_{s}},y)(X^{\varepsilon}_s)\|^2\right]^{1/2}\mid_{\{x=X_{s}^{\varepsilon},y=Y^{\varepsilon}_{s}\}}ds\Big|^p\right\}
 \nonumber\\
 \leq\!\!\!\!\!\!\!\!&&C_{p,T}\EE\left\{\int_0^T\left[1+\left(\EE|X^{\varepsilon}_s|^2\right)^{p/2}+\left(\EE|Y^{\varepsilon}_s|^2\right)^{p/2}\right](1+|Y^{\varepsilon}_s|^p)ds\right\}
 \nonumber\\
 \leq\!\!\!\!\!\!\!\!&&\vare^{-p/2}C_{p,T}\left[1+\left(\EE|\xi|^{2}\right)^{p/2}+\left(\EE|\zeta|^{2}\right)^{p/2}\right]\EE\int_0^T(1+|Y^{\varepsilon}_s|^p)ds
 \nonumber\\
 \leq\!\!\!\!\!\!\!\!&&\vare^{-p/2}C_{p,T}\left(1+\EE|\xi|^{p}+\EE|\zeta|^{p}\right).
\end{eqnarray*}
Similarly, due to (\ref{Bound Condition}), (\ref{E1}), (\ref{X1}) and (\ref{Y0}), it is easy to see that
\begin{eqnarray*}
\!\!\!\!\!\!\!\!&&\EE\left\{\sup_{t\in[0,T]}\left|\int^t_0 \EE \text{Tr}\left[\sigma\sigma^{*}(X^{\varepsilon}_s,\mathscr{L}_{ X^{\varepsilon}_{s}},Y^{\varepsilon}_s)\partial_z\partial_{\mu}\Phi(x,\mathscr{L}_{X^{\varepsilon}_{s}},y)(X^{\varepsilon}_s)\right]\mid_{\{x=X_{s}^{\varepsilon},y=Y^{\varepsilon}_{s}\}}ds\right|^p\right\}
 \nonumber\\
 \leq\!\!\!\!\!\!\!\!&&C_{p,T}\EE\int_0^T(1+|Y^{\varepsilon}_s|^{p})\big(1+(\EE|Y^{\varepsilon}_s|^4)^{p/2}\big)ds
 \nonumber\\
 \leq\!\!\!\!\!\!\!\!&&C_{p,T}\big(1+\EE|\zeta|^{2p}\big)
\end{eqnarray*}
and
\begin{eqnarray*}
 \!\!\!\!\!\!\!\!&&\EE\left\{\sup_{t\in[0,T]}\left|\int^t_0 \mathscr{L}_{1}(\mathscr{L}_{X^{\varepsilon}_{s}},Y^{\varepsilon}_{s})\Phi(X_{s}^{\varepsilon},\mathscr{L}_{X^{\varepsilon}_{s}},Y^{\varepsilon}_{s})ds\right|^p\right\}
 \nonumber\\
 \leq\!\!\!\!\!\!\!\!&&C_{p,T}\EE\int_0^T\left[1+
 \big(\EE|X^{\varepsilon}_{s}|^2\big)^{p/2}+|X^{\varepsilon}_{s}|^p+|Y^{\varepsilon}_{s}|^p\right](1+|Y^{\varepsilon}_s|^p)ds+C_{p,T}\EE\int_0^T(1+|Y^{\varepsilon}_s|^{3p})ds
 \nonumber\\
 \leq\!\!\!\!\!\!\!\!&&C_{p,T}(1+\EE|\zeta|^{3p})+C_{p,T}\int_0^T\EE|X^{\varepsilon}_s|^p ds+C_{p,T}\left[\EE\int_0^T|X^{\varepsilon}_{s}|^{3p/2}ds\right]^\frac{2}{3}\left[\EE\int_0^T|Y^{\varepsilon}_{s}|^{3p}ds\right]^\frac{1}{3}
\nonumber\\
 \leq\!\!\!\!\!\!\!\!&&\vare^{-p/2}C_{p,T}\left(1+\EE|\xi|^{3p/2}+\EE|\zeta|^{3p}\right).
\end{eqnarray*}
Then it turns out that
\begin{eqnarray}\label{es2}
\mathscr{R}_1^\varepsilon\leq C_{p,T}\varepsilon^{p/2}\left(1+\EE|\xi|^{3p/2}+\EE|\zeta|^{3p}\right).
\end{eqnarray}

For the term $\mathscr{R}_2^\varepsilon$. By the Burkholder-Davis-Gundy's inequality, we get
\begin{eqnarray}\label{es3}
\mathscr{R}_2^\varepsilon  \leq\!\!\!\!\!\!\!\!&& C_{p}\varepsilon^p \EE\left[\int^T_0 \|\partial_x \Phi(X_{s}^{\varepsilon},\mathscr{L}_{X^{\varepsilon}_{s}},Y_{s}^{\varepsilon})\|^2
 \|\sigma(X^{\varepsilon}_s,\mathscr{L}_{ X^{\varepsilon}_{s}},Y_{s}^{\varepsilon})\|^2 ds\right]^{p/2}
\nonumber\\
\leq\!\!\!\!\!\!\!\!&&C_{p,T}\varepsilon^p\int_0^T\EE\big( 1 +|Y^{\varepsilon}_{t}|^{2p}\big)dt
\nonumber\\
\leq\!\!\!\!\!\!\!\!&& C_{p,T}\varepsilon^p(1+\EE|\zeta|^{2p}).
\end{eqnarray}

By a similar argument above, one can easily obtain
\begin{equation}\label{es4}
\sum^7_{i=3}\mathscr{R}_{i}^\varepsilon\leq C_{p,T}\varepsilon^{p/2}(1+\EE|\zeta|^{2p}).
\end{equation}


Consequently, combining (\ref{es1}) and \eref{es2}-(\ref{es4}), we conclude the desired estimate (\ref{sfe}). The proof is complete.
\end{proof}

\vspace{1mm}
Based on the above fluctuation estimates,  the uniform moment estimate for the slow variable $X_{t}^{\varepsilon}$ of (\ref{eq4}) is derived in the following. Furthermore, we also establish the increment estimates of time for $X_{t}^{\varepsilon}$, which is important to prove the strong/weak averaging principle for the terms in equations (\ref{MP1}), (\ref{MP2}) and (\ref{tildeX}) below.
\begin{lemma}
Assume that the conditions in Theorem \ref{main result 1} hold. For any $T>0$ and $\EE|\xi|^{3p/2}+\EE|\zeta|^{3p}<\infty$ for some $p\geq 2$, then there exists a constant $C_{p,T}>0$ such that
\begin{equation}
\sup_{\varepsilon\in(0,1)}\mathbb{E}\left\{\sup_{t\in [0, T]}|X_{t}^{\varepsilon}|^{p}\right\}\leq C_{p,T}\left(1+\EE|\xi|^{3p/2}+\EE|\zeta|^{3p}\right).\label{X2}
\end{equation}
In addition, for any $0\leq t\leq t+h\leq T$,
\begin{equation}
\mathbb{E}|X_{t+h}^{\varepsilon}-X_{t}^{\varepsilon}|^{p}\leq C_{p,T}\left(1+\EE|\xi|^{3p/2}+\EE|\zeta|^{3p}\right)\left(h^{p/2}+\frac{h^p}{\varepsilon^{p/2}}\right).\label{COX}
\end{equation}
\end{lemma}

\begin{proof}
Applying It\^{o}'s formula, Burkholder-Davis-Gundy's inequality and (\ref{sfe}), we easily deduce that
\begin{eqnarray*}
\mathbb{E}\left\{\sup_{t\in [0, T]}|X_{t}^{\varepsilon}|^{p}\right\}\leq\!\!\!\!\!\!\!\!&&C_{p,T}(1+\EE|\xi|^p)+C_{p,T}\int_0^T\mathbb{E}|X_{t}^{\varepsilon}|^{p}dt+C_{p,T}\int_0^T\mathbb{E}|Y_{t}^{\varepsilon}|^{p}dt
\nonumber\\
\!\!\!\!\!\!\!\!&&
+\frac{C_p}{\varepsilon^{p/2}}\mathbb{E}\left\{\sup_{t\in [0, T]}\Big|\int_0^t K(X_{s}^{\varepsilon}, \mathscr{L}_{X_{s}^{\varepsilon}},Y_{s}^{\varepsilon})ds\Big|^p\right\}
\nonumber\\
\leq\!\!\!\!\!\!\!\!&& C_{p,T}\left(1+\EE|\xi|^{3p/2}+\EE|\zeta|^{3p}\right)+C_{p,T}\int_0^T\mathbb{E}|X_{t}^{\varepsilon}|^{p}dt.
\end{eqnarray*}
Then the Gronwall's lemma implies (\ref{X2}) holds.

Analogously,  it is straightforward that for any $0\leq t\leq t+h\leq T$,
\begin{eqnarray*}
\mathbb{E}|X_{t+h}^{\varepsilon}-X_{t}^{\varepsilon}|^{p}\leq\!\!\!\!\!\!\!\!&&C_p\mathbb{E}\left|\int_t^{t+h}b(X_{s}^{\varepsilon},\mathscr{L}_{X^{\varepsilon}_{s}},Y^{\varepsilon}_{s})ds\right|^p
+C_p\mathbb{E}\left|\int_t^{t+h}\sigma(X_{s}^{\varepsilon},\mathscr{L}_{X^{\varepsilon}_{s}},Y^{\varepsilon}_{s})dW_s^1\right|^p
\nonumber\\
\!\!\!\!\!\!\!\!&&
+\frac{C_p}{\varepsilon^{p/2}}\mathbb{E}\left|\int_{t}^{t+h}K(X_{s}^{\varepsilon}, \mathscr{L}_{X_{s}^{\varepsilon}},Y_{s}^{\varepsilon})ds\right|^p
\nonumber\\
\leq\!\!\!\!\!\!\!\!&&C_p h^{p/2}\left(1+ \sup_{t\in[0,T]}\EE|X_{t}^{\varepsilon}|^{p}+\sup_{t\in[0,T]}\EE|Y_{t}^{\varepsilon}|^{p}\right)+\frac{C_p}{\varepsilon^{p/2}}\mathbb{E}\left|\int_{t}^{t+h}(1+|Y_{s}^{\varepsilon}|)ds\right|^p
\nonumber\\
\leq\!\!\!\!\!\!\!\!&& C_{p,T}\left(1+\EE|\xi|^{3p/2}+\EE|\zeta|^{3p}\right)\left(h^{p/2}+\frac{h^p}{\varepsilon^{p/2}}\right).
\end{eqnarray*}
The proof is complete.

\end{proof}

\subsection{Tightness}\label{S4.2}

In this subsection, we intend to prove the tightness of the solution $\{X^\varepsilon\}_{\varepsilon>0}$ to equation (\ref{eq4}) in $C([0,T],\RR^n)$, then it has a weakly convergent subsequence. To this end, we recall the following necessary and sufficient condition of tightness (cf.~\cite{KS1}).
\begin{lemma}
For any $T>0$, the family $\{\Pi^\varepsilon\}_{\varepsilon\in(0,1)}$ is tight in $C([0,T];\RR^n)$ if and only if the following two conditions hold.

(i) There exists  a constant $r>0$ such that
\begin{equation}\label{t1}
\sup_{\varepsilon\in(0,1)}\EE|\Pi^\varepsilon_0|^r<\infty.
\end{equation}

(ii) There exist constants $r,\delta>0$ such that for any $0\leq t_1,t_2\leq T$,
\begin{equation}\label{t2}
\sup_{\varepsilon\in(0,1)}\EE|\Pi^\varepsilon_{t_2}-\Pi^\varepsilon_{t_1}|^r\leq C_T|t_2-t_1|^{1+\delta}.
\end{equation}

\end{lemma}

\vspace{0.3cm}
Recall the slow process in system (\ref{eq4})
\begin{equation}\label{eq6}
X^\varepsilon_t=\xi+I_1^\varepsilon(t)+I_2^\varepsilon(t)+I_3^\varepsilon(t),~t\in[0,T],
\end{equation}
where we denote
\begin{eqnarray*}
I_1^\varepsilon(t)=\!\!\!\!\!\!\!\!&&\int^t_0b(X^{\varepsilon}_s, \mathscr{L}_{X^{\varepsilon}_s}, Y^{\varepsilon}_s)ds,
\nonumber\\
I_2^\varepsilon(t)=\!\!\!\!\!\!\!\!&&\frac{1}{\sqrt{\varepsilon}}\int^t_0 K(X^{\varepsilon}_s, \mathscr{L}_{X^{\varepsilon}_s}, Y^{\varepsilon}_s)ds,
\nonumber\\
I_3^\varepsilon(t)=\!\!\!\!\!\!\!\!&&\int^t_0\sigma(X^{\varepsilon}_s,\mathscr{L}_{X^{\varepsilon}_s},Y^{\varepsilon}_s)d W_s.
\end{eqnarray*}

\begin{proposition}\label{th1}
Assume that the conditions in Theorem \ref{main result 1} hold. Then  $\{\mathscr{L}_{\Pi^\varepsilon}\}_{\varepsilon>0}$ is tight in $C([0,T];\RR^{4n})$, where $\Pi^\varepsilon:=(X^\varepsilon,I_1^\varepsilon,I_2^\varepsilon,I_3^\varepsilon)$.
\end{proposition}
\begin{proof}
 For proving the tightness of $\{\mathscr{L}_{X^\varepsilon}\}_{\varepsilon>0}$, it is sufficient to show that  $\{I_1^\varepsilon\}_{\varepsilon>0},\{I_2^\varepsilon\}_{\varepsilon>0}$ and $\{I_2^\varepsilon\}_{\varepsilon>0}$ satisfy the criterions (\ref{t1}) and (\ref{t2}), respectively. Note that (\ref{t1}) is a direct consequence of the a priori estimates (\ref{Y0}), (\ref{sfe}) and (\ref{X2}), we only need to verify the criterion (\ref{t2}).

By (\ref{Y0}), (\ref{X2}) and the Burkholder-Davis-Gundy's inequality,
\begin{eqnarray*}
\EE|I_1^\varepsilon(t_2)-I_1^\varepsilon(t_1)|^6\leq\!\!\!\!\!\!\!\!&&|t_2-t_1|^5\EE\left[\int^{t_2}_{t_1}\big|b(X^{\varepsilon}_s, \mathscr{L}_{X^{\varepsilon}_s}, Y^{\varepsilon}_s)\big|^6ds\right]
\nonumber\\
\leq\!\!\!\!\!\!\!\!&&|t_2-t_1|^5\sup_{s\in[0,T]} \EE\big|b(X^{\varepsilon}_s, \mathscr{L}_{X^{\varepsilon}_s}, Y^{\varepsilon}_s)\big|^6
\nonumber\\
\leq\!\!\!\!\!\!\!\!&&C_{T}|t_2-t_1|^6(1+\EE|\xi|^{9}+\EE|\zeta|^{18})
\end{eqnarray*}
and
\begin{eqnarray*}
\EE|I_3^\varepsilon(t_2)-I_3^\varepsilon(t_1)|^6\leq\!\!\!\!\!\!\!\!&&|t_2-t_1|^2\EE\left[\int^{t_2}_{t_1}\big\|\sigma(X^{\varepsilon}_s, \mathscr{L}_{X^{\varepsilon}_s}, Y^{\varepsilon}_s)\big\|^6ds\right]
\nonumber\\
\leq\!\!\!\!\!\!\!\!&&C_{T}|t_2-t_1|^3(1+\EE|\zeta|^{6}),
\end{eqnarray*}
thus it is easy to see that (\ref{t2}) holds for $\{I_1^\varepsilon\}_{\varepsilon>0}$ and $\{I_3^\varepsilon\}_{\varepsilon>0}$.

As for $\{I_2^\varepsilon\}_{\varepsilon>0}$, we recall that
\begin{eqnarray}
I_2^{\varepsilon}(t)=\!\!\!\!\!\!\!\!&& \sqrt{\varepsilon}\Bigg\{\Phi(\xi,\mathscr{L}_{\xi},\zeta)-\Phi(X_{t}^{\varepsilon},\mathscr{L}_{X^{\varepsilon}_{t}},Y^{\varepsilon}_{t})\nonumber\\
\!\!\!\!\!\!\!\!&&
+\int^t_0 \EE\left[b(X^{\varepsilon}_s,\mathscr{L}_{ X^{\varepsilon}_{s}}, Y^{\varepsilon}_s)\partial_{\mu}\Phi(x,\mathscr{L}_{X^{\varepsilon}_{s}},y)(X^{\varepsilon}_s)\right]\mid_{\{x=X_{s}^{\varepsilon},y=Y^{\varepsilon}_{s}\}}ds\nonumber\\
\!\!\!\!\!\!\!\!&&+\frac{1}{\sqrt{\varepsilon}}\int^t_0\EE\left[K(X^{\varepsilon}_s,\mathscr{L}_{ X^{\varepsilon}_{s}}, Y^{\varepsilon}_s)\partial_{\mu}\Phi(x,\mathscr{L}_{X^{\varepsilon}_{s}},y)(X^{\varepsilon}_s)\right]\mid_{\{x=X_{s}^{\varepsilon},y=Y^{\varepsilon}_{s}\}}ds\nonumber\\
\!\!\!\!\!\!\!\!&&+\int^t_0 \frac{1}{2}\EE \text{Tr}\left[\sigma\sigma^{*}(X^{\varepsilon}_s,\mathscr{L}_{ X^{\varepsilon}_{s}},Y^{\varepsilon}_s)\partial_z\partial_{\mu}\Phi(x,\mathscr{L}_{X^{\varepsilon}_{s}},y)(X^{\varepsilon}_s)\right]\mid_{\{x=X_{s}^{\varepsilon},y=Y^{\varepsilon}_{s}\}}ds\nonumber\\
\!\!\!\!\!\!\!\!&&+\int^t_0 \mathscr{L}_{1}(\mathscr{L}_{X^{\varepsilon}_{s}},Y^{\varepsilon}_{s})\Phi(X_{s}^{\varepsilon},\mathscr{L}_{X^{\varepsilon}_{s}},Y^{\varepsilon}_{s})ds+ M^{1,\varepsilon}_t\Bigg\}\nonumber\\
\!\!\!\!\!\!\!\!&&
+\int^t_0  \partial_x \Phi_K(X_{s}^{\varepsilon},\mathscr{L}_{X^{\varepsilon}_{s}},Y^{\varepsilon}_{s}) ds+\int^t_0  \partial_y \Phi_h(X_{s}^{\varepsilon},\mathscr{L}_{X^{\varepsilon}_{s}},Y^{\varepsilon}_{s}) ds\nonumber\\
\!\!\!\!\!\!\!\!&&+\int^t_0\text{Tr}\left[\partial^2_{xy}\Phi_{\sigma g^{*}}\right](X^{\varepsilon}_{s},\mathscr{L}_{X^{\varepsilon}_{s}},Y^{\varepsilon}_{s})ds +M^{2,\varepsilon}_t
\nonumber\\
=:\!\!\!\!\!\!\!\!&&I_{21}^{\varepsilon}(t)+I_{22}^{\varepsilon}(t)+I_{23}^{\varepsilon}(t)+I_{24}^{\varepsilon}(t)+M^{2,\varepsilon}_t,\label{es7}
\end{eqnarray}
where $M^{i,\varepsilon}_t$, $i=1,2$, are defined by (\ref{ma1}) and (\ref{ma2}), respectively.

 Firstly,  we deal with the term
$$
J^{\varepsilon}_{21}(t):=\int^t_0\EE\left[K(X^{\varepsilon}_s,\mathscr{L}_{ X^{\varepsilon}_{s}}, Y^{\varepsilon}_s)\partial_{\mu}\Phi(x,\mathscr{L}_{X^{\varepsilon}_{s}},y)(X^{\varepsilon}_s)\right]\mid_{\{x=X_{s}^{\varepsilon},y=Y^{\varepsilon}_{s}\}}ds
$$
in $I^{\varepsilon}_{21}$. To do this, we construct a copy of $(X^{\varepsilon}, Y^{\varepsilon})$ by $(\breve{X}^{\varepsilon}, \breve{Y}^{\varepsilon})$ which is define on another probability space. Since $K$ satisfies the centering condition \ref{A4}, using \eref{E1} and by a minor revision in Lemma \ref{lem3}, it easy to check that for any small enough $\Delta>0$,
\begin{eqnarray*}
&&\EE\left\{\sup_{t\in[0,T]}\left|\int^t_0\left[K(X^{\varepsilon}_s,\mathscr{L}_{ X^{\varepsilon}_{s}}, Y^{\varepsilon}_s)\partial_{\mu}\Phi(\breve X_{s}^{\varepsilon},\mathscr{L}_{X^{\varepsilon}_{s}},\breve Y^{\varepsilon}_s)(X^{\varepsilon}_s)\right]ds\right|\right\}\\
\leq\!\!\!\!\!\!\!\!&&  C_T(1+{\EE|\xi|^{9}}+\EE|\zeta|^{18})\Bigg\{\frac{\Delta}{\sqrt{\varepsilon}}+\sqrt{\Delta}\\
&&+\frac{\varepsilon}{\Delta}\max_{0\leq j\leq[T/\Delta]-1}\left[\int_{0} ^{\frac{\Delta}{\varepsilon}} \int_{r} ^{\frac{\Delta}{\varepsilon}}(1+|\breve Y^{\varepsilon}_{s\Delta+j\Delta}|+|\breve Y^{\varepsilon}_{r\Delta+j\Delta}|)\left(e^{-\frac{(s-r)\beta}{2}}+\sqrt{\varepsilon}\right)dsdr \right]^{1/2}\Bigg\},
\end{eqnarray*}
which implies that
\begin{eqnarray*}
\!\!\!\!\!\!\!\!&&\EE\left\{\sup_{t\in[0,T]}\left|J^{\varepsilon}_{21}(t)\right|\right\}
\\
=\!\!\!\!\!\!\!\!&& \breve{\EE}\left\{\sup_{t\in[0,T]}\left|\EE\int^t_0\left[K(X^{\varepsilon}_s,\mathscr{L}_{ X^{\varepsilon}_{s}}, Y^{\varepsilon}_s)\partial_{\mu}\Phi(\breve X_{s}^{\varepsilon},\mathscr{L}_{X^{\varepsilon}_{s}},\breve Y^{\varepsilon}_s)(X^{\varepsilon}_s)\right]ds\right|\right\}\\
\leq\!\!\!\!\!\!\!\!&&  \breve{\EE}\EE\left\{\sup_{t\in[0,T]}\left|\int^t_0\left[K(X^{\varepsilon}_s,\mathscr{L}_{ X^{\varepsilon}_{s}}, Y^{\varepsilon}_s)\partial_{\mu}\Phi(\breve X_{s}^{\varepsilon},\mathscr{L}_{X^{\varepsilon}_{s}},\breve Y^{\varepsilon}_s)(X^{\varepsilon}_s)\right]ds\right|\right\}\\
\leq\!\!\!\!\!\!\!\!&&  C_T(1+{\EE|\xi|^{9}}+\EE|\zeta|^{18})\left(\frac{\Delta}{\sqrt{\varepsilon}}+\sqrt{\Delta}+\sqrt{\varepsilon}+\frac{\sqrt{\varepsilon}}{\sqrt{\Delta}}\right).
\end{eqnarray*}
Then taking $\Delta=\varepsilon^{\frac{2}{3}}$, it follows that
\begin{eqnarray}
\lim_{\varepsilon\rightarrow 0}\EE\left\{\sup_{t\in[0,T]}\left|J^{\varepsilon}_{21}(t)\right|\right\}=0.\label{CI22}
\end{eqnarray}
In addition, by the conditions $\ref{A1}$, $\ref{A3}$, Proposition \ref{P3.6}, \eref{X1} and (\ref{Y0}), we have
\begin{equation}\label{es23}
\EE\left\{\sup_{t\in[0,T]}|I_{21}^{\varepsilon}(t)-J_{21}^{\varepsilon}(t)|^6\right\}\leq C_{T}\varepsilon^3(1+\EE|\xi|^{9}+\EE|\zeta|^{18}).
\end{equation}
Combining (\ref{CI22}) and (\ref{es23}), it immediately obtain that $I_{21}^{\varepsilon}$ converges to $0$ in probability in $C([0,T];\RR^{n})$, as $\varepsilon\to 0$.

Hence, it suffices to verify that the remaining terms in (\ref{es7}) also satisfy  the criterion (\ref{t2}).
Indeed, according to the condition \eref{Bound Condition} and (\ref{Y0}),  we get
\begin{eqnarray*}
\EE|I_{22}^\varepsilon(t_2)-I_{22}^\varepsilon(t_1)|^6\leq\!\!\!\!\!\!\!\!&&|t_2-t_1|^5\int^{t_2}_{t_1}\EE\big(1+|Y^{\varepsilon}_{s}|^{12}\big)ds
\nonumber\\
\leq\!\!\!\!\!\!\!\!&&C_T|t_2-t_1|^6\big(1+\EE|\zeta|^{12}\big).
\end{eqnarray*}
Then  it is easy see that $\{I_{22}^\varepsilon\}_{\varepsilon>0}$ fulfills the criterion (\ref{t2}). The terms $I_{23}^{\varepsilon}$, $I_{24}^{\varepsilon}$ and $M^{2,\varepsilon}_t$ satisfy  the criterion (\ref{t2}), whose proofs are omitted since they can be handled similarly.  The proof is complete.
\end{proof}

\section{Proof of main results}\label{sec5}
In this section, we shall give the detailed proofs of our main results, i.e., Theorems \ref{main result 1} and \ref{main result 2}. By the discussion in the introduction, {\it martingale problem approach} (see subsection 4.1) and {\it martingale characteriziation} (see subsection 4.2) are used to characterize the limiting process, respectively.

\subsection{Proof of Theorem \ref{main result 1}}
 It is sufficient to prove that  any sequence $\{\varepsilon_k\}_{k\geq1}$ has a  subsequence which we keep denoting by $\{\varepsilon_{k}\}_{k\geq1}$ such that $X^{\varepsilon_{k}}$ converges weakly to the solution $X$ of (\ref{eq5}) in $C([0,T];\RR^n)$, as $k\to\infty$. If it is proved that the limit is unique in the distribution sense, then it will follow that the whole family $X^{\varepsilon}$ converges weakly in $C([0,T];\RR^n)$ as $\vare\rightarrow 0$. We will divide the proof into three steps.

\textbf{Step 1}: We shall apply the martingale problem approach to characterize the limiting process $X$. Let $\Psi_{t_0}(\cdot)$
be a bounded continuous function on $C([0,T],\RR^{n})$ which is measurable with respect to the sigma-field $\sigma(\varphi_t, \varphi\in C([0,T],\RR^{n}), t\leq t_0 )$. We intend to show that for any $t_0 \geq 0$, any $\Psi_{t_0}(\cdot)$ and $U\in C^3_b(\RR^n)$, the following assertion holds:
\begin{eqnarray}
\EE\left[\left(U(X_t)-U(X_{t_0})-\int^t_{t_0}L_{\mu_s}U(X_s)ds\right)\Psi_{t_0}(X)\right]=0,\quad t\geq t_0,\label{Mar11}
\end{eqnarray}
where $\mu_s:=\mathscr{L}_{X_s}$ and
$$L_{\mu}U(x):=\sum^n_{i=1}\partial_{x_i}U(x)\Theta_i(x,\mu)+\frac{1}{2}\sum^{n}_{i=1}\sum^n_{j=1}\partial_{x_i}\partial_{x_j} U(x)(\Sigma\Sigma^{\ast})_{ij}(x,\mu).$$
Then by the uniqueness of the martingale problem to stochastic system \eref{eq5} (cf.~e.g.~\cite[Corollary 4.1]{LM}),  the limiting process $X$ satisfies equation \eref{eq5}.

Note that the Skorohod representation theorem yields that it is possible to construct a probability space $(\hat{\Omega},\hat{\mathscr{F}},\hat{\PP})$, and there exists a sequence $\{\hat{X}^{\varepsilon_{k}}, \hat W^{k}\}_{k\geq 1}$ and $\hat{X}$ on this space such that
$$
(X^{\varepsilon_{k}}, W)\sim (\hat{X}^{\varepsilon_{k}}, \hat W^{k}), \quad  X\sim\hat{X},
$$
where $\sim$ means they have the same distribution in $C([0,T];\RR^{n})$. Moreover, we have $\hat{\PP}\text{-a.s.}$, as $k\to \infty$
\begin{eqnarray}
&&\hat{X}^{\varepsilon_{k}}\rightarrow \hat{X}\quad \text{in}\quad C([0,T];\RR^{n}),\label{hatXC}\\
&&\hat W^{k}\rightarrow\hat{W}, \quad \text{in}\quad C([0,T];\RR^{n}).\nonumber
\end{eqnarray}
Meanwhile, $(\hat{X}^{\varepsilon_{n_k}}, \hat W^{k})$ solve the following stochastic systems
\begin{equation}\left\{\begin{array}{l}\label{ChangeEquation}
\displaystyle
d \hat{X}^{\varepsilon_k}_t = b(\hat{X}^{\varepsilon_k}_t, \mathscr{L}_{\hat{X}^{\varepsilon_k}_t}, \hat{Y}^{\varepsilon_k}_t)dt+\frac{1}{\sqrt{\varepsilon_k}}K(\hat{X}^{\varepsilon_{k}}_t, \mathscr{L}_{\hat{X}^{\varepsilon_{k}}_t}, \hat{Y}^{\varepsilon_k}_t)dt+\sigma(\hat{X}^{\varepsilon_k}_t, \mathscr{L}_{\hat{X}^{\varepsilon_k}_t},\hat{Y}^{\varepsilon_k}_t)d \hat W^{k}_t,
\\
\displaystyle d \hat{Y}^{\varepsilon_k}_t =\frac{1}{\varepsilon_k}f(\hat Y^{\varepsilon_k}_t, \mathscr{L}_{\hat{X}^{\varepsilon_k}_t}, \hat{Y}^{\varepsilon_k}_t)dt+\frac{1}{\sqrt{\varepsilon_k}}h(\hat Y^{\varepsilon_k}_t, \mathscr{L}_{\hat{X}^{\varepsilon_k}_t}, \hat{Y}^{\varepsilon_k}_t)dt+\frac{1}{\sqrt{\varepsilon_k}}g( \hat{X}^{\varepsilon_k}_t, \mathscr{L}_{\hat{X}^{\varepsilon_k}_t}, \hat{Y}^{\varepsilon_k}_t)d \hat W^{k}_t,\\
\displaystyle \hat{X}^{\varepsilon_k}_0=\hat{\xi},~\hat{Y}^{\varepsilon_k}_0=\hat{\zeta},
\end{array}\right.
\end{equation}
where $\hat{\xi}$ and $\hat{\zeta}$ are two random variables on $(\hat{\Omega},\hat{\mathscr{F}},\hat{\PP})$ which satisfy $\hat{\xi}\sim \xi$ and $\hat{\zeta}\sim \zeta$.

\vspace{1mm}
It is easy to check the a priori estimates \eref{Y0} and (\ref{X2}) hold for $\hat{X}^{\varepsilon_{k}}$ and $\hat{Y}^{\varepsilon_{k}}$ respectively, i.e., for any $T>0$ and $\hat \EE|\hat \xi|^{3p/2}+\hat \EE|\hat \zeta|^{3p}<\infty$ for some $p\geq 2$, then there exists a constant $C_{p,T}>0$ such that
\begin{eqnarray}
&&\sup_{k\in\mathbb{N}_{+}}\hat{\mathbb{E}}\left\{\sup_{t\in [0, T]}|\hat{X}_{t}^{\varepsilon_{k}}|^{p}\right\}\leq C_{p,T}(1+\hat{\EE}|\xi|^{3p/2}+\hat{\EE}|\hat \zeta|^{3p}),\label{X22}\\
&&\sup_{k\in\mathbb{N}_{+}}\sup_{t\geq 0}\hat{\mathbb{E}}|\hat Y_{t}^{\varepsilon_k}|^{p}\leq C_{p}\left(1+\hat{\EE}|\hat\zeta|^{p}\right),\label{hatY0}
\end{eqnarray}
where $\hat{E}$ is the expectation on $(\hat{\Omega},\hat{\mathscr{F}},\hat{\PP})$. Then by \eref{hatXC} and \eref{X22}, the Vitali's convergence theorem (cf.~\cite[Theorem 4.5.4]{B1}) implies that for any $p'<p$,
\begin{equation}\label{vita1}
\hat{\EE}\left\{\sup_{t\in[0,T]}|\hat{X}^{\varepsilon_{k}}_t-\hat{X}_t|^{p'}\right\}\to 0,~\text{as}~k\to\infty.
\end{equation}

Now, in order to prove (\ref{Mar11}), it is sufficient to prove that
\begin{eqnarray}
\hat{\EE}\left[\left(U(\hat{X}_t)-U(\hat{X}_s)-\int^t_{t_0}L_{\hat{\mu}_s}U(\hat{X}_s)ds\right)\Psi_{t_0}(\hat{X})\right]=0,\quad t\geq t_0,\label{Martingle method}
\end{eqnarray}
where $\hat{\mu}_s:=\mathscr{L}_{\hat{X}_s}$.

\vspace{1mm}
\textbf{Step 2}: Recall the solution $\Phi$ of Poisson equation (\ref{PE}), applying It\^{o}'s formula we obtain that for any $t\geq t_0$,
\begin{eqnarray*}
&&\langle\Phi(\hat{X}_{t}^{\varepsilon_k},\mathscr{L}_{\hat{X}_{t}^{\varepsilon_k}},\hat{Y}_{t}^{\varepsilon_k}), \nabla U(\hat{X}_{t}^{\varepsilon_k})\rangle
=\langle\Phi(\hat{X}_{t_0}^{\varepsilon_k},\mathscr{L}_{\hat{X}^{\varepsilon_k}_{t_0}},\hat{Y}^{\varepsilon_k}_{t_0}), \nabla U(\hat{X}^{\vare_k}_{t_0})\rangle\\
&&+\sum^n_{i=1}\int^t_{t_0} \hat{\EE}\left[b(\hat{X}^{\varepsilon_k}_s,\mathscr{L}_{ \hat{X}^{\varepsilon_k}_{s}}, \hat{Y}^{\varepsilon_k}_s)\partial_{\mu}\Phi^i(x,\mathscr{L}_{\hat{X}^{\varepsilon_k}_{s}},y)(\hat{X}^{\varepsilon_k}_s)\right]\partial _{x_i}U(x)\mid_{\{x=\hat{X}_{s}^{\varepsilon_k},y=\hat{Y}^{\varepsilon_k}_{s}\}}ds\\
&&+\sum^n_{i=1}\int^t_{t_0} \hat{\EE}\left[\frac{1}{\sqrt{\varepsilon_k}}K(\hat{X}^{\varepsilon_k}_s,\mathscr{L}_{ \hat{X}^{\varepsilon_k}_{s}}, \hat{Y}^{\varepsilon_k}_s)\partial_{\mu}\Phi^i(x,\mathscr{L}_{\hat{X}^{\varepsilon_k}_{s}},y)(\hat{X}^{\varepsilon_k}_s)\right]\partial _{x_i}U(x)\mid_{\{x=\hat{X}_{s}^{\varepsilon_k},y=\hat{Y}^{\varepsilon_k}_{s}\}}ds\\
&&+\sum^n_{i=1}\int^t_{t_0} \frac{1}{2}\hat{\EE} \text{Tr}\left[(\sigma\sigma^{*})(\hat{X}^{\varepsilon_k}_s,\mathscr{L}_{ \hat{X}^{\varepsilon_k}_{s}}, \hat{Y}^{\varepsilon_k}_s)\partial_z\partial_{\mu}\Phi^i(x,\mathscr{L}_{\hat{X}^{\varepsilon_k}_{s}},y)(\hat{X}^{\varepsilon_k}_s)\right]\partial _{x_i}U(x)\mid_{\{x=\hat{X}_{s}^{\varepsilon_k},y=\hat{Y}^{\varepsilon_k}_{s}\}}ds\\
&&+\int^t_{t_0} \mathscr{L}_{1}(\mathscr{L}_{\hat{X}^{\varepsilon_k}_{s}},\hat{Y}^{\varepsilon_k}_{s})(\Phi\nabla U)(\hat{X}_{s}^{\varepsilon_k},\mathscr{L}_{\hat{X}^{\varepsilon_k}_{s}},\hat{Y}^{\varepsilon_k}_{s})ds\\
&&+\frac{1}{\sqrt{\varepsilon_k}}\sum^n_{i=1}\int^t_{t_0}\partial_{x_i} U(\hat{X}_{s}^{\varepsilon_k})\left\{\partial_x \Phi^i_K(\hat{X}_{s}^{\varepsilon_k},\mathscr{L}_{\hat{X}^{\varepsilon_k}_{s}},\hat{Y}^{\varepsilon_k}_{s})+\partial_y \Phi^i_h(\hat{X}_{s}^{\varepsilon_k},\mathscr{L}_{\hat{X}^{\varepsilon_k}_{s}},\hat{Y}^{\varepsilon_k}_{s}) \right.\\
&&\quad\quad\quad\quad\quad\quad\quad\quad\quad\quad\quad\quad+\left.\text{Tr}\left[\partial^2_{xy}\Phi^i_{\sigma g^{*}}(\hat{X}^{\varepsilon_k}_{s},\mathscr{L}_{\hat{X}^{\varepsilon_k}_{s}},\hat{Y}^{\varepsilon_k}_{s})\right]\right\}ds          \\
&&+\frac{1}{\sqrt{\varepsilon_k}}\sum^n_{i=1}\sum^n_{j=1}\int^t_{t_0} \partial _{x_i}\partial_{x_j} U(\hat{X}^{\vare_k}_s) (K\otimes\Phi)_{ij}(\hat{X}_{s}^{\varepsilon_k},\mathscr{L}_{\hat{X}^{\varepsilon_k}_{s}},\hat{Y}^{\varepsilon_k}_{s})ds\\
&&+\frac{1}{\sqrt{\varepsilon_k}}\sum^n_{i=1}\sum^n_{j=1}\int^t_{t_0} \partial _{x_i}\partial_{x_j}U(\hat{X}_{s}^{\varepsilon_k})\left[(\sigma g^{*})\partial_{y}\Phi\right]_{ij}(\hat{X}^{\varepsilon_k}_{s},\mathscr{L}_{\hat{X}^{\varepsilon_k}_{s}},\hat{Y}^{\varepsilon_k}_{s})ds          \\
&&+\frac{1}{\varepsilon_k}\int^t_{t_0} \langle \mathscr{L}_{2}(\hat{X}_{s}^{\varepsilon_k},\mathscr{L}_{\hat{X}^{\varepsilon_k}_{s}})\Phi(\hat{X}_{s}^{\varepsilon_k},\mathscr{L}_{\hat{X}^{\varepsilon_k}_{s}},\hat{Y}^{\varepsilon_k}_{s}), \nabla U(\hat{X}^{\vare_k}_s)\rangle ds\\
&&+\int^t_{t_0}\langle\partial_x(\Phi\nabla U)(\hat{X}_{s}^{\varepsilon_k},\mathscr{L}_{\hat{X}^{\varepsilon_k}_{s}},\hat{Y}^{\varepsilon_k}_{s}), \sigma(\hat{X}_{s}^{\varepsilon_k},\mathscr{L}_{\hat{X}^{\varepsilon_k}_{s}},\hat{Y}^{\varepsilon_k}_{s})d\hat W^{k}_s\rangle\\
&&+\frac{1}{\sqrt{\vare_{k}}}\int^t_{t_0}\langle\partial_y(\Phi\nabla U)(\hat{X}_{s}^{\varepsilon_k},\mathscr{L}_{\hat{X}^{\varepsilon_k}_{s}},\hat{Y}^{\varepsilon_k}_{s}), g(\hat{X}_{s}^{\varepsilon_k},\mathscr{L}_{\hat{X}^{\varepsilon_k}_{s}},\hat{Y}^{\varepsilon_k}_{s})d\hat W^{k}_s\rangle,
\end{eqnarray*}
where $(\Phi\nabla U)(x,\mu,y):=\sum^n_{i=1}\Phi^i(x,\mu,y)\partial_{x_i}U(x)$. Define
\begin{eqnarray*}
R^{\vare_k}_{U}=\!\!\!\!\!\!\!\!&&\langle\Phi(\hat{X}_{t_0}^{\varepsilon_k},\mathscr{L}_{\hat{X}^{\varepsilon_k}_{t_0}},\hat{Y}^{\varepsilon_k}_{t_0}), \nabla U(\hat{X}^{\vare_k}_{t_0})\rangle-\langle\Phi(\hat{X}_{t}^{\varepsilon_k},\mathscr{L}_{\hat{X}_{t}^{\varepsilon_k}},\hat{Y}_{t}^{\varepsilon_k}), \nabla U(\hat{X}_{t}^{\varepsilon_k})\rangle\\
&&+\sum^n_{i=1}\int^t_{t_0} \hat{\EE}\left[b(\hat{X}^{\varepsilon_k}_s,\mathscr{L}_{ \hat{X}^{\varepsilon_k}_{s}}, \hat{Y}^{\varepsilon_k}_s)\partial_{\mu}\Phi^i(x,\mathscr{L}_{\hat{X}^{\varepsilon_k}_{s}},y)(\hat{X}^{\varepsilon_k}_s)\right]\partial _{x_i}U(x)\mid_{\{x=\hat{X}_{s}^{\varepsilon_k},y=\hat{Y}^{\varepsilon_k}_{s}\}}ds\\
&&+\sum^n_{i=1}\int^t_{t_0} \hat{\EE}\left[\frac{1}{\sqrt{\varepsilon_k}}K(\hat{X}^{\varepsilon_k}_s,\mathscr{L}_{ \hat{X}^{\varepsilon_k}_{s}}, \hat{Y}^{\varepsilon_k}_s)\partial_{\mu}\Phi^i(x,\mathscr{L}_{\hat{X}^{\varepsilon_k}_{s}},y)(\hat{X}^{\varepsilon_k}_s)\right]\partial _{x_i}U(x)\mid_{\{x=\hat{X}_{s}^{\varepsilon_k},y=\hat{Y}^{\varepsilon_k}_{s}\}}ds\\
&&+\sum^n_{i=1}\int^t_{t_0} \frac{1}{2}\hat{\EE} \text{Tr}\left[(\sigma\sigma^{*})(\hat{X}^{\varepsilon_k}_s,\mathscr{L}_{ \hat{X}^{\varepsilon_k}_{s}}, \hat{Y}^{\varepsilon_k}_s)\partial_z\partial_{\mu}\Phi^i(x,\mathscr{L}_{\hat{X}^{\varepsilon_k}_{s}},y)(\hat{X}^{\varepsilon_k}_s)\right]\partial _{x_i}U(x)\mid_{\{x=\hat{X}_{s}^{\varepsilon_k},y=\hat{Y}^{\varepsilon_k}_{s}\}}ds\\
&&+\int^t_{t_0} \mathscr{L}_{1}(\mathscr{L}_{\hat{X}^{\varepsilon_k}_{s}},\hat{Y}^{\varepsilon_k}_{s})(\Phi\nabla U)(\hat{X}_{s}^{\varepsilon_k},\mathscr{L}_{\hat{X}^{\varepsilon_k}_{s}},\hat{Y}^{\varepsilon_k}_{s})ds\\
&&+\int^t_{t_0}\langle\partial_x(\Phi\nabla U)(\hat X_{s}^{\varepsilon_k},\mathscr{L}_{\hat X^{\varepsilon_k}_{s}},\hat Y^{\varepsilon_k}_{s}), \sigma(\hat X_{s}^{\varepsilon_k},\mathscr{L}_{\hat X^{\varepsilon_k}_{s}},\hat Y^{\varepsilon_k}_{s})d\hat W^{k}_s\rangle.
\end{eqnarray*}
Notice that
\begin{eqnarray*}
&&2\sum^n_{i=1}\sum^n_{j=1}\frac{1}{\sqrt{\varepsilon_k}}\int^t_{t_0} \partial _{x_i}\partial_{x_j} U(\hat{X}^{\vare_k}_s) (K\otimes\Phi)_{ij}(\hat{X}_{s}^{\varepsilon_k},\mathscr{L}_{\hat{X}^{\varepsilon_k}_{s}},\hat{Y}^{\varepsilon_k}_{s})ds\\
=\!\!\!\!\!\!\!\!&&\sum^n_{i=1}\sum^n_{j=1}\frac{1}{\sqrt{\varepsilon_k}}\int^t_{t_0} \partial _{x_i}\partial_{x_j} U(\hat{X}^{\vare_k}_s) \left[(K\otimes\Phi)+(K\otimes\Phi)^{\ast}\right]_{ij}(\hat{X}_{s}^{\varepsilon_k},\mathscr{L}_{\hat{X}^{\varepsilon_k}_{s}},\hat{Y}^{\varepsilon_k}_{s})ds
\end{eqnarray*}
and
\begin{eqnarray*}
&&2\sum^n_{i=1}\sum^n_{j=1}\int^t_{t_0} \partial _{x_i}\partial_{x_j}U(\hat{X}_{s}^{\varepsilon_k})\left[(\sigma g^{*})\partial_{y}\Phi\right]_{ij}(\hat{X}^{\varepsilon_k}_{s},\mathscr{L}_{\hat{X}^{\varepsilon_k}_{s}},\hat{Y}^{\varepsilon_k}_{s})ds\\
=\!\!\!\!\!\!\!\!&&\sum^n_{i=1}\sum^n_{j=1}\int^t_{t_0} \partial _{x_i}\partial_{x_j}U(\hat{X}_{s}^{\varepsilon_k})\left(\left[(\sigma g^{*})\partial_{y}\Phi\right]+\left[(\sigma g^{*})\partial_{y}\Phi\right]^{\ast}\right)_{ij}(\hat{X}^{\varepsilon_k}_{s},\mathscr{L}_{\hat{X}^{\varepsilon_k}_{s}},\hat{Y}^{\varepsilon_k}_{s})ds.
\end{eqnarray*}
Thus, it follows that
\begin{eqnarray*}
&&\frac{1}{\sqrt{\varepsilon_k}}\int^t_{t_0} \langle\nabla U(\hat X^{\vare_k}_s),K(\hat{X}^{\varepsilon_k}_s,\mathscr{L}_{ \hat X^{\varepsilon_k}_{s}},\hat Y^{\varepsilon_k}_s)\rangle ds\\
=\!\!\!\!\!\!\!\!&&-\frac{1}{\sqrt{\varepsilon_k}}\int^t_{t_0} \langle\nabla U(\hat X^{\vare_k}_s),\mathscr{L}_{2}(\hat X_{s}^{\varepsilon_k},\mathscr{L}_{\hat X^{\varepsilon_k}_{s}})\Phi(\hat X_{s}^{\varepsilon_k},\mathscr{L}_{\hat X^{\varepsilon_k}_{s}},\hat Y^{\varepsilon_k}_{s})\rangle ds\\
=\!\!\!\!\!\!\!\!&&\sqrt{\vare_k}R^{\vare_k}_{U}+\sum^n_{i=1}\int^t_{t_0} \partial_{x_i} U(\hat{X}_{s}^{\varepsilon_k})\left\{\partial_x \Phi^i_K(\hat{X}_{s}^{\varepsilon_k},\mathscr{L}_{\hat{X}^{\varepsilon_k}_{s}},\hat{Y}^{\varepsilon_k}_{s})+\partial_y \Phi^i_h(\hat{X}_{s}^{\varepsilon_k},\mathscr{L}_{\hat{X}^{\varepsilon_k}_{s}},\hat{Y}^{\varepsilon_k}_{s}) \right.\\
&&\quad\quad\quad\quad\quad\quad\quad\quad\quad\quad\quad\quad+\left.\text{Tr}\left[\partial^2_{xy}\Phi^i_{\sigma g^{*}}(\hat{X}^{\varepsilon_k}_{s},\mathscr{L}_{\hat{X}^{\varepsilon_k}_{s}},\hat{Y}^{\varepsilon_k}_{s})\right]\right\}ds          \\
&&+\frac{1}{2}\sum^n_{i=1}\sum^n_{j=1}\int^t_{t_0} \partial_{x_i}\partial_{x_j} U(\hat{X}^{\vare_k}_s) \left[(K\otimes\Phi)+(K\otimes\Phi)^{\ast}\right]_{ij}(\hat{X}_{s}^{\varepsilon_k},\mathscr{L}_{\hat{X}^{\varepsilon_k}_{s}},\hat{Y}^{\varepsilon_k}_{s})ds\\
&&+\frac{1}{2}\sum^n_{i=1}\sum^n_{j=1}\int^t_{t_0} \partial_{x_i}\partial_{x_j}U(\hat{X}_{s}^{\varepsilon_k})\left(\left[(\sigma g^{*})\partial_{y}\Phi\right]+\left[(\sigma g^{*})\partial_{y}\Phi\right]^{\ast}\right)_{ij}(\hat{X}^{\varepsilon_k}_{s},\mathscr{L}_{\hat{X}^{\varepsilon_k}_{s}},\hat{Y}^{\varepsilon_k}_{s})ds          \\
&&+\int^t_{t_0}\langle\partial_y(\Phi\nabla U)(\hat X_{s}^{\varepsilon_k},\mathscr{L}_{\hat X^{\varepsilon_k}_{s}},\hat Y^{\varepsilon_k}_{s}), g(\hat X_{s}^{\varepsilon_k},\mathscr{L}_{\hat X^{\varepsilon_k}_{s}},\hat Y^{\varepsilon_k}_{s})d\hat W^{k}_s\rangle.
\end{eqnarray*}
By the formulation above and applying It\^{o}'s formula again, we have
\begin{eqnarray*}
U(\hat X^{\vare_k}_t)=\!\!\!\!\!\!\!\!&& U(\hat X^{\vare_k}_{t_0})\!+\!\sum^n_{i=1}\int^t_{t_0}\!\!\partial_{x_i}U(\hat X^{\vare_k}_s)b_i(\hat X^{\vare_k}_s,\mathscr{L}_{ \hat X^{\vare_k}_{s}}, \hat Y^{\varepsilon_k}_s) ds\\
&&+\int^t_{t_0} \langle\nabla U(\hat X^{\vare_k}_s),\frac{1}{\sqrt{\varepsilon_k}}K(\hat X^{\vare_k}_s,\mathscr{L}_{\hat X^{\vare_k}_s}, \hat Y^{\varepsilon_k}_s)\rangle ds\nonumber\\
&&+\frac{1}{2}\sum^n_{i=1}\sum^n_{j=1} \int^t_{t_0}\left[\partial_{x_i}\partial_{x_j}U(\hat X^{\vare_k}_s)(\sigma\sigma^{\ast})_{ij}(\hat X^{\varepsilon_k}_s,\mathscr{L}_{\hat X^{\varepsilon_k}_{s}}, \hat Y^{\varepsilon_k}_s)\right] ds\nonumber\\
\!\!\!\!\!\!\!\!&&+ \int^t_{t_0}\langle\nabla U(\hat X^{\vare_k}_s),\sigma(\hat X^{\varepsilon_k}_s,\mathscr{L}_{\hat X^{\varepsilon_k}_{s}}, \hat Y^{\varepsilon_k}_s)d \hat W^{k}_s\rangle\\
=\!\!\!\!\!\!\!\!&& U(\hat X^{\vare_k}_{t_0})+\sum^n_{i=1}\int^t_{t_0}\partial_{x_i}U(\hat X^{\vare_k}_s)\left\{b_i(\hat X^{\vare_k}_s,\mathscr{L}_{ \hat X^{\vare_k}_{s}}, \hat Y^{\varepsilon_k}_s)+\partial_x\Phi^i_K(\hat{X}_{s}^{\varepsilon_k},\mathscr{L}_{\hat{X}^{\varepsilon_k}_{s}},\hat{Y}^{\varepsilon_k}_{s})\right.\\
&&\left.\quad\quad+\partial_y \Phi^i_h(\hat{X}_{s}^{\varepsilon_k},\mathscr{L}_{\hat{X}^{\varepsilon_k}_{s}},\hat{Y}^{\varepsilon_k}_{s})+\text{Tr}\left[\partial^2_{xy}\Phi^i_{\sigma g^{*}}\right](\hat{X}^{\varepsilon_k}_{s},\mathscr{L}_{\hat{X}^{\varepsilon_k}_{s}},\hat{Y}^{\varepsilon_k}_{s})\right\} ds          \\
&&+\frac{1}{2}\sum^n_{i=1}\sum^n_{j=1}\int^t_{t_0} \partial_{x_i}\partial_{x_j} U(\hat X^{\vare}_s)\left\{\left[(K\otimes\Phi)+(K\otimes\Phi)^{\ast}\right]_{ij}(\hat{X}_{s}^{\varepsilon_k},\mathscr{L}_{\hat{X}^{\varepsilon_k}_{s}},\hat{Y}^{\varepsilon_k}_{s})  \right.  \\
&&\left.+\Big(\left[(\sigma g^{*})\partial_{y}\Phi\right]+\left[(\sigma g^{*})\partial_{y}\Phi\right]^{\ast}\Big)_{ij}(\hat{X}^{\varepsilon_k}_{s},\mathscr{L}_{\hat{X}^{\varepsilon_k}_{s}},\hat{Y}^{\varepsilon_k}_{s})+(\sigma\sigma^{\ast})_{ij}(\hat X^{\varepsilon_k}_s,\mathscr{L}_{\hat X^{\varepsilon_k}_{s}},\hat Y^{\varepsilon_k}_s)\right\}ds\\
&&+\int^t_{t_0}\langle\partial_y(\Phi\nabla U)(\hat X_{s}^{\varepsilon_k},\mathscr{L}_{\hat X^{\varepsilon_k}_{s}},\hat Y^{\varepsilon_k}_{s}), g(\hat X_{s}^{\varepsilon_k},\mathscr{L}_{\hat X^{\varepsilon_k}_{s}},\hat Y^{\varepsilon_k}_{s})d\hat W^{k}_s\rangle\nonumber\\
\!\!\!\!\!\!\!\!&&+ \int^t_{t_0}\langle\nabla U(\hat X^{\vare_k}_s),\sigma(\hat X^{\varepsilon_k}_s,\mathscr{L}_{\hat X^{\varepsilon_k}_{s}}, \hat Y^{\varepsilon_k}_s)d \hat W^{k}_s\rangle+ \sqrt{\vare_k}R^{\vare_k}_{U}.
\end{eqnarray*}
Obviously, $\lim_{k\rightarrow \infty}\hat \EE \left[(\Gamma^{\vare_k}- \sqrt{\vare_k}R^{\vare_k}_{U})\Psi_{t_0}(\hat X^{\vare_k})\right]=0$, where
\begin{eqnarray*}
\Gamma^{\vare_k}:=\!\!\!\!\!\!\!\!&&U(\hat X^{\vare_k}_t)-U(\hat X^{\vare_k}_{t_0})-\sum^n_{i=1}\int^t_{t_0}\partial_{x_i}U(\hat X^{\vare_k}_s)\left\{b_i(\hat X^{\vare_k}_s,\mathscr{L}_{ \hat X^{\vare_k}_{s}}, \hat Y^{\varepsilon_k}_s)+\partial_x \Phi^i_K(\hat{X}_{s}^{\varepsilon_k},\mathscr{L}_{\hat{X}^{\varepsilon_k}_{s}},\hat{Y}^{\varepsilon_k}_{s})\right.\\
&&\left.\quad\quad\quad\quad+\partial_y \Phi^i_h(\hat{X}_{s}^{\varepsilon_k},\mathscr{L}_{\hat{X}^{\varepsilon_k}_{s}},\hat{Y}^{\varepsilon_k}_{s})+\text{Tr}\left[\partial^2_{xy}\Phi^i_{\sigma g^{*}}\right](\hat{X}^{\varepsilon_k}_{s},\mathscr{L}_{\hat{X}^{\varepsilon_k}_{s}},\hat{Y}^{\varepsilon_k}_{s})\right\} ds          \\
&&-\frac{1}{2}\sum^n_{i=1}\sum^n_{j=1}\int^t_{t_0} \partial_{x_i}\partial_{x_j} U(\hat X^{\vare_k}_s)\left\{\left[(K\otimes\Phi)+(K\otimes\Phi)^{\ast}\right]_{ij}(\hat{X}_{s}^{\varepsilon_k},\mathscr{L}_{\hat{X}^{\varepsilon_k}_{s}},\hat{Y}^{\varepsilon_k}_{s})  \right.  \\
&&\left.\quad+\Big(\left[(\sigma g^{*})\partial_{y}\Phi\right]+\left[(\sigma g^{*})\partial_{y}\Phi\right]^{\ast}\Big)_{ij}(\hat{X}^{\varepsilon_k}_{s},\mathscr{L}_{\hat{X}^{\varepsilon_k}_{s}},\hat{Y}^{\varepsilon_k}_{s})+(\sigma\sigma^{\ast})_{ij}(\hat X^{\varepsilon_k}_s,\mathscr{L}_{\hat X^{\varepsilon_k}_{s}},\hat Y^{\varepsilon_k}_s)\right\}ds.
\end{eqnarray*}
By \eref{X22}, \eref{hatY0}, \eref{CI22} and Proposition \ref{P3.6}, it is easy to see that
\begin{eqnarray*}
\lim_{k\rightarrow \infty}\sqrt{\vare_k}\hat \EE\left[R^{\vare_k}_{U}\Psi_{t_0}(\hat X^{\vare_k})\right]=0.
\end{eqnarray*}

Thus it remains to prove that
$$\lim_{k\rightarrow \infty}\hat{\EE}\left[\Gamma^{\vare_k}\Psi_{t_0}(\hat X^{\vare_k})\right]=\hat \EE \left[\Gamma\Psi_{t_0}(\hat X)\right],$$
where $\Gamma:=U(\hat X_t)-U(\hat X_{t_0})-\int^t_{t_0}L_{\hat{\mu}_s}U(\hat X_s)ds $, then \eref{Martingle method} holds obviously.

\textbf{Step 3}: Note that $\hat{X}^{\varepsilon_k}$ converges to $\hat{X}$ in $C([0,T];\RR^n)$ $\PP$-a.s., it is sufficient to prove that for any $1\leq i,j\leq n$,
\begin{eqnarray}
\!\!\!\!\!\!\!\!&&\hat \EE\Bigg|\int^t_{t_0}\!\partial_{x_i}U(\hat X^{\vare_k}_s)\Big\{b_i(\hat X^{\vare_k}_s,\mathscr{L}_{ \hat X^{\vare_k}_{s}}, \hat Y^{\varepsilon_k}_s)\!+\!\partial_x \Phi^i_K(\hat{X}_{s}^{\varepsilon_k},\mathscr{L}_{\hat{X}^{\varepsilon_k}_{s}},\hat{Y}^{\varepsilon_k}_{s})\nonumber\\
&&\quad\quad+\partial_y \Phi^i_h(\hat{X}_{s}^{\varepsilon_k},\mathscr{L}_{\hat{X}^{\varepsilon_k}_{s}},\hat{Y}^{\varepsilon_k}_{s})+\text{Tr}\left[\partial^2_{xy}\Phi^i_{\sigma g^{*}}\right](\hat{X}^{\varepsilon_k}_{s},\mathscr{L}_{\hat{X}^{\varepsilon_k}_{s}},\hat{Y}^{\varepsilon_k}_{s})\Big\}ds\nonumber\\
&&\quad-\int^t_{t_0} \partial_{x_i}U(\hat X^{\vare_k}_s)\left\{ \bar{b}_i(\hat{X}_s,\mathscr{L}_{\hat{X}_{s}})+\overline{\partial_{x}\Phi^i_{K}}(\hat{X}_{s},\mathscr{L}_{\hat{X}_{s}})\right.\nonumber\\
&&\quad\quad+\overline{\partial_{y}\Phi^i_{h}}(\hat{X}_{s},\mathscr{L}_{\hat{X}_{s}})+\left.\overline{\text{Tr}\left[\partial^2_{xy}\Phi^i_{\sigma g^{*}}\right]}(\hat{X}_s,\mathscr{L}_{ \hat{X}_{s}})\right\} ds \Bigg|=0\label{MP1}
\end{eqnarray}
and
\begin{eqnarray}
&&\hat \EE\Bigg|\int^t_{t_0} \partial_{x_i}\partial_{x_j} U(\hat{X}^{\vare_k}_s)\Big\{\left[(K\otimes\Phi)+(K\otimes\Phi)^{\ast}\right]_{ij}(\hat{X}_{s}^{\varepsilon_k},\mathscr{L}_{\hat{X}^{\varepsilon_k}_{s}},\hat{Y}^{\varepsilon_k}_{s})\nonumber\\
&&\quad\quad+\Big(\left[(\sigma g^{*})\partial_{y}\Phi\right]+\left[(\sigma g^{*})\partial_{y}\Phi\right]^{\ast}\Big)_{ij}(\hat{X}^{\varepsilon_k}_s,\mathscr{L}_{\hat{X}^{\varepsilon_k}_{s}},\hat{Y}^{\varepsilon_k}_s) +(\sigma\sigma^{\ast})_{ij}(\hat{X}^{\varepsilon_k}_s,\mathscr{L}_{\hat{X}^{\varepsilon_k}_{s}}, \hat{Y}^{\varepsilon_k}_s)\Big\}ds\nonumber\\
&&\quad-\int^t_{t_0} \partial_{x_i}\partial_{x_j} U(\hat{X}_s)\Big\{\overline{(K\otimes\Phi)+(K\otimes\Phi)^{\ast}}_{ij}(\hat{X}_{s},\mathscr{L}_{\hat{X}_{s}})\nonumber\\
&&\quad\quad+\Big(\overline{\left[(\sigma g^{*})\partial_{y}\Phi\right]+\left[(\sigma g^{*})\partial_{y}\Phi\right]^{\ast}}\Big)_{ij}(\hat{X}_s,\mathscr{L}_{\hat{X}_{s}})+(\overline{\sigma\sigma^{\ast}})_{ij}(\hat{X}_s,\mathscr{L}_{\hat{X}_{s}})\Big\}ds\Bigg|=0.\label{MP2}
\end{eqnarray}

In order to prove \eref{MP1} and \eref{MP2}, we shall use the strategy of Khasminskii's time discretization. Firstly, we define for any $1\leq i,j\leq n$,
\begin{eqnarray*}
&&G_i(x,\mu,y)=\partial_{x_i} U(x)\left\{b_i(x,\mu,y)+\partial_x \Phi^i_K(x,\mu,y)+\partial_y \Phi^i_h(x,\mu,y)+\text{Tr}\left[\partial^2_{xy}\Phi^{i}_{\sigma g^{*}}\right](x,\mu,y)\right\},\\
&&H_{ij}(x,\mu,y)=\partial_{x_i}\partial_{x_j}U(x)\Big\{\left[(K\otimes\Phi)+(K\otimes\Phi)^{\ast}\right]_{ij}(x,\mu,y)\\
&&\quad\quad\quad\quad\quad\quad+\Big(\left[(\sigma g^{*})\partial_{y}\Phi\right]+\left[(\sigma g^{*})\partial_{y}\Phi\right]^{\ast}\Big)_{ij}(x,\mu,y)+(\sigma\sigma^{\ast})_{ij}(x,\mu,y)\Big\}.
\end{eqnarray*}
Then by \eref{E1}-\eref{E4}, it is easy check that
\begin{eqnarray*}
&&|G_i(x_1,\mu_1,y_1)-G_i(x_2,\mu_2,y_2)|\leq C\left[1+|x_1|+\mu_1(|\cdot|^2)^{1/2}+|y_1|^2+|y_2|^2\right]|x_1-x_2|\\
&&\quad\quad\quad\quad\quad\quad\quad\quad\quad\quad\quad\quad\quad\quad+C(1+|y_1|^2+|y_2|^2)(\mathbb{W}_2(\mu_1,\mu_2)+|y_1-y_2|),\\
&&|H_{ij}(x_1,\mu_1,y_1)-H_{ij}(x_2,\mu_2,y_2)|\leq C(1+|y_1|^2+|y_2|^2)\left(|x_1-x_2|+\mathbb{W}_2(\mu_1,\mu_2)+|y_1-y_2|\right).
\end{eqnarray*}
Finally, note that
\begin{eqnarray*}
&&\overline{G}_i(x,\mu)=\partial_{x_i} U(x)\left\{\bar{b}_i(x,\mu)+\overline{\partial_x \Phi^i_K}(x,\mu)+\overline{\partial_y \Phi^i_h}(x,\mu)+\overline{\text{Tr}\left[\partial^2_{xy}\Phi_{\sigma g^{*}}\right]}(x,\mu)\right\},\\
&&\overline{H_{ij}}(x,\mu)=\partial_{x_i}\partial_{x_j} U(x)\Big\{\overline{\left[(K\otimes\Phi)+(K\otimes\Phi)^{\ast}\right]}_{ij}\\
&&\quad\quad\quad\quad\quad+\Big(\overline{\left[(\sigma g^{*})\partial_{y}\Phi\right]+\left[(\sigma g^{*})\partial_{y}\Phi\right]^{\ast}}\Big)_{ij}(x,\mu)+\overline{(\sigma\sigma^{\ast})}_{ij}(x,\mu)\Big\}.
\end{eqnarray*}
As a consequence, it immediately obtain that \eref{MP1} and \eref{MP2} holds by \eref{Key3} in the appendix. The proof is complete. \hspace{\fill}$\Box$

\subsection{Proof of Theorem \ref{main result 2}}
 We also divide the proof into three steps.

\vspace{1mm}
\textbf{Step 1}: Recall the formulas \eref{eq6} and \eref{es7}, then we know that
\begin{equation*}\label{eq17}
X^{\varepsilon_{k}}_t=\xi+I_1^{\varepsilon_{k}}(t)+I_{21}^{\varepsilon_{k}}(t)+I_{22}^{\varepsilon_{k}}(t)+I_{23}^{\varepsilon_{k}}(t)+I_{24}^{\varepsilon_{k}}(t)+ M^{\varepsilon_{k}}_t,
\end{equation*}
where $M^{\varepsilon_{k}}_t$ is a local martingale which is defined as follows
\begin{eqnarray*}
&&M^{\varepsilon_{k}}_t=M^{2,\varepsilon_{k}}_t+I_{3}^{\varepsilon_{k}}(t)=\int^t_0 \big[\partial_y \Phi_g+\sigma\big](X_{s}^{\varepsilon_{k}},\mathscr{L}_{X^{\varepsilon_{k}}_{s}},Y_{s}^{\varepsilon_{k}}) dW_s.
\end{eqnarray*}
It is worth noting that by (\ref{CI22}) and (\ref{es23}), $I_{21}^{\varepsilon_{k}}$ converges to $0$ in probability in $C([0,T];\RR^{n})$, as $k\to\infty$.

In the sequel, we shall apply the martingale representation theorem to characterize the limiting process $X$. By Proposition \ref{th1}, for a subsequence $\{\varepsilon_{k}\}_{k\geq1}$, we suppose that the family
$$\{(X^{\varepsilon_{k}},I_1^{\varepsilon_{k}},I_{21}^{\varepsilon_{k}},I_{22}^{\varepsilon_{k}},I_{23}^{\varepsilon_{k}},I_{24}^{\varepsilon_{k}},M^{\varepsilon_{k}},W) \}_{k\geq1}$$
are weakly convergent in $C([0,T],\RR^{7n+d})$.

By the Skorohod representation theorem, it is possible to construct on a complete probability space $(\tilde{\Omega},\tilde{\mathscr{F}},\tilde{\PP})$, and there exists a sequence $\{(\tilde{X}^{\varepsilon_{k}},\tilde{I}_1^{\varepsilon_{k}},\tilde{I}_{21}^{\varepsilon_{k}},\tilde{I}_{22}^{\varepsilon_{k}},\tilde{I}_{23}^{\varepsilon_{k}},\tilde{I}_{24}^{\varepsilon_{k}},\tilde{M}^{\varepsilon_{k}},\tilde{W}^{k}) \}_{k\geq1}$ and $(\tilde{X},\tilde{I}_1,\tilde{I}_{21},\tilde{I}_{22},\tilde{I}_{23},\tilde{I}_{24},\tilde{M},\tilde{W}) $ on this space such that
$$
(X^{\varepsilon_{k}},I_1^{\varepsilon_{k}},I_{21}^{\varepsilon_{k}},I_{22}^{\varepsilon_{k}},I_{23}^{\varepsilon_{k}},I_{24}^{\varepsilon_{k}},M^{\varepsilon_{k}},W)\sim (\tilde{X}^{\varepsilon_{k}},\tilde{I}_1^{\varepsilon_{k}},\tilde{I}_{21}^{\varepsilon_{k}},\tilde{I}_{22}^{\varepsilon_{k}},\tilde{I}_{23}^{\varepsilon_{k}},,\tilde{I}_{23}^{\varepsilon_{k}},\tilde{M}^{\varepsilon_{k}},\tilde{W}^{k})
$$
and
$$(X,I_1,I_{21},I_{22},I_{23},I_{24},M,W)\sim (\tilde{X},\tilde{I}_1,\tilde{I}_{21},\tilde{I}_{22},\tilde{I}_{23},\tilde{I}_{24},\tilde{M},\tilde{W}).$$
Moreover, we have $\tilde{\PP}\text{-a.s.}$, as $k\to \infty$,
\begin{eqnarray}
&&\tilde{X}^{\varepsilon_{k}}\rightarrow \tilde{X}\quad \text{in}\quad C([0,T];\RR^{n});\label{Xasconvergence}\\
&&\tilde{I}_1^{\varepsilon_{k}}\rightarrow \tilde{I}_1\quad \text{in}\quad C([0,T];\RR^{n});\nonumber\label{I1asconvergence}\\
&&\tilde{I}_{21}^{\varepsilon_{k}}\rightarrow \tilde{I}_{21}\quad \text{in}\quad C([0,T];\RR^{n});\nonumber\label{I21asconvergence}\\
&&\tilde{I}_{22}^{\varepsilon_{k}}\rightarrow \tilde{I}_{22}\quad \text{in}\quad C([0,T];\RR^{n});\nonumber\label{I22asconvergence}\\
&&\tilde{I}_{23}^{\varepsilon_{k}}\rightarrow \tilde{I}_{23}\quad \text{in}\quad C([0,T];\RR^{n});\nonumber\label{I23asconvergence}\\
&&\tilde{I}_{24}^{\varepsilon_{k}}\rightarrow \tilde{I}_{24}\quad \text{in}\quad C([0,T];\RR^{n});\nonumber\label{I24asconvergence}\\
&&\tilde{W}^{k}\rightarrow\tilde{W} \quad \text{in}\quad C([0,T];\RR^{n}).\nonumber\label{Masconvergence}
\end{eqnarray}
Meanwhile, $(\tilde{X}^{\varepsilon_{k}}, \tilde{W}^{k})$ solve the following SDEs
\begin{equation}\left\{\begin{array}{l}\label{ChangeEquation2}
\displaystyle
d \tilde{X}^{\varepsilon_k}_t = b(\tilde{X}^{\varepsilon_k}_t, \mathscr{L}_{\tilde{X}^{\varepsilon_k}_t}, \tilde{Y}^{\varepsilon_k}_t)dt+\frac{1}{\sqrt{\varepsilon_k}}K(\tilde{X}^{\varepsilon_{k}}_t, \mathscr{L}_{\tilde{X}^{\varepsilon_{k}}_t}, \tilde{Y}^{\varepsilon_k}_t)dt+\sigma(\tilde{X}^{\varepsilon_k}_t, \mathscr{L}_{\tilde{X}^{\varepsilon_k}_t},\tilde{Y}^{\varepsilon_k}_t)d \tilde{W}^{k}_t,
\\
\displaystyle d \tilde{Y}^{\varepsilon_k}_t =\frac{1}{\varepsilon_k}f(\tilde Y^{\varepsilon_k}_t, \mathscr{L}_{\tilde{X}^{\varepsilon_k}_t}, \tilde{Y}^{\varepsilon_k}_t)dt+\frac{1}{\sqrt{\varepsilon_k}}h(\tilde Y^{\varepsilon_k}_t, \mathscr{L}_{\tilde{X}^{\varepsilon_k}_t}, \tilde{Y}^{\varepsilon_k}_t)dt+\frac{1}{\sqrt{\varepsilon_k}}g( \tilde{X}^{\varepsilon_k}_t, \mathscr{L}_{\tilde{X}^{\varepsilon_k}_t}, \tilde{Y}^{\varepsilon_k}_t)d \tilde W^{k}_t\\
\displaystyle \tilde{X}^{\varepsilon_k}_0=\tilde{\xi},~\tilde{Y}^{\varepsilon_k}_0=\tilde{\zeta},
\end{array}\right.
\end{equation}
and
\begin{eqnarray}
\tilde{X}^{\varepsilon_{k}}_t=\tilde{\xi}+\tilde{I}_1^{\varepsilon_{k}}(t)+\tilde{I}_{21}^{\varepsilon_{k}}(t)+\tilde{I}_{22}^{\varepsilon_{k}}(t)+\tilde{I}_{23}^{\varepsilon_{k}}(t)+\tilde{I}_{24}^{\varepsilon_{k}}(t)+ \tilde{M}^{\varepsilon_{k}}_t,\label{tildeX}
\end{eqnarray}
with $\tilde{I}_1^{\varepsilon_{k}},\tilde{I}_{21}^{\varepsilon_{k}},\tilde{I}_{22}^{\varepsilon_{k}},\tilde{I}_{23}^{\varepsilon_{k}},\tilde{I}_{24}^{\varepsilon_{k}}$ and $\tilde{M}^{\varepsilon_{k}}$ satisfying
\begin{eqnarray*}
&&\tilde{I}_1^{\varepsilon_k}(t)=\int^t_0b(\tilde{X}^{\varepsilon_k}_s, \mathscr{L}_{\tilde{X}^{\varepsilon_k}_s},\tilde{Y}^{\varepsilon_k}_s)ds,
\nonumber\\
&&\tilde{I}_{21}^{\varepsilon_k}(t)=\sqrt{\varepsilon_k}\Big\{\Phi(\tilde{\xi},\mathscr{L}_{\tilde{\xi}},\tilde{\zeta})-\Phi(\tilde{X}_{t}^{\varepsilon_k},\mathscr{L}_{\tilde{X}^{\varepsilon_k}_{t}},\tilde{Y}^{\varepsilon_k}_{t})\nonumber\\
\!\!\!\!\!\!\!\!&&
\quad\quad\quad\quad\quad\quad+\int^t_0 \tilde{\EE}\left[b(\tilde{X}^{\varepsilon_k}_s,\mathscr{L}_{ \tilde{X}^{\varepsilon_k}_{s}}, \tilde{Y}^{\varepsilon_k}_s)\partial_{\mu}\Phi(x,\mathscr{L}_{\tilde{X}^{\varepsilon_k}_{s}},y)(\tilde{X}^{\varepsilon_k}_s)\right]\mid_{\{x=\tilde{X}_{s}^{\varepsilon_k},y=\tilde{Y}^{\varepsilon_k}_{s}\}}ds\nonumber\\
\!\!\!\!\!\!\!\!&&
\quad\quad\quad\quad\quad\quad+\frac{1}{\sqrt{\varepsilon_k}}\int^t_0 \tilde{\EE}\left[K(\tilde{X}^{\varepsilon_k}_s,\mathscr{L}_{ \tilde{X}^{\varepsilon_k}_{s}}, \tilde{Y}^{\varepsilon_k}_s)\partial_{\mu}\Phi(x,\mathscr{L}_{\tilde{X}^{\varepsilon_k}_{s}},y)(\tilde{X}^{\varepsilon_k}_s)\right]\mid_{\{x=\tilde{X}_{s}^{\varepsilon_k},y=\tilde{Y}^{\varepsilon_k}_{s}\}}ds,\nonumber\\
\!\!\!\!\!\!\!\!&&\quad\quad\quad\quad\quad\quad+\int^t_0 \frac{1}{2}\tilde{\EE} \text{Tr}\left[\sigma\sigma^{*}(\tilde{X}^{\varepsilon_k}_s,\mathscr{L}_{ \tilde{X}^{\varepsilon_k}_{s}},\tilde{Y}^{\varepsilon_k}_s)\partial_z\partial_{\mu}\Phi(x,\mathscr{L}_{\tilde{X}^{\varepsilon_k}_{s}},y)(\tilde{X}^{\varepsilon_k}_s)\right]\mid_{\{x=\tilde{X}_{s}^{\varepsilon_k},y=\tilde{Y}^{\varepsilon_k}_{s}\}}ds\nonumber\\
\!\!\!\!\!\!\!\!&&\quad\quad\quad\quad\quad\quad+\int^t_0 \mathscr{L}_{1}(\mathscr{L}_{\tilde{X}^{\varepsilon_k}_{s}},\tilde{Y}^{\varepsilon_k}_{s})\Phi(\tilde{}_{s}^{\varepsilon_k},\mathscr{L}_{\tilde{X}^{\varepsilon_k}_{s}},\tilde{Y}^{\varepsilon_k}_{s})ds+\int^t_0 \partial_x \Phi_{\sigma}(\tilde{X}_{s}^{\varepsilon_k},\mathscr{L}_{\tilde{X}^{\varepsilon_k}_{s}},\tilde{Y}_{s}^{\varepsilon_k})d\tilde{W}^{k}_s\Big\},\nonumber\\
&&\tilde{I}_{22}^{\varepsilon_k}(t)=\int^t_0  \partial_x \Phi_K(\tilde{X}_{s}^{\varepsilon_k},\mathscr{L}_{\tilde{X}^{\varepsilon_k}_{s}},\tilde{Y}^{\varepsilon_k}_{s})ds,\\
&&\tilde{I}_{23}^{\varepsilon_k}(t)=\int^t_0  \partial_y \Phi_h(\tilde{X}_{s}^{\varepsilon_k},\mathscr{L}_{\tilde{X}^{\varepsilon_k}_{s}},\tilde{Y}^{\varepsilon_k}_{s})ds,\\
&&\tilde{I}_{24}^{\varepsilon_k}(t)=\int^t_0\text{Tr}\left[\partial^2_{xy}\Phi_{\sigma g^{*}}\right](\tilde{X}^{\varepsilon_k}_{s},\mathscr{L}_{\tilde{X}^{\varepsilon_k}_{s}},\tilde{Y}^{\varepsilon_k}_{s})ds,\\
&&\tilde{M}^{\varepsilon_k}_t=\int^t_0 (\partial_y \Phi_{g}+\sigma)(\tilde X_{s}^{\varepsilon},\mathscr{L}_{\tilde X^{\varepsilon_k}_{s}},\tilde Y_{s}^{\varepsilon_k})d\tilde{W}^{k}_s,
\end{eqnarray*}
where $\tilde{\xi}$ and $\tilde{\zeta}$ are two random variables satisfying $\tilde{\xi}\sim \xi$ and $\tilde{\zeta}\sim \zeta$, and $\tilde{\EE}$ is the expectation on $(\tilde{\Omega},\tilde{\mathscr{F}},\tilde{\PP})$.

It is easy to check that the  a priori estimates \eref{Y0} and (\ref{X2}) also hold for $\tilde{X}^{\varepsilon_{k}}$ and $\tilde{Y}^{\varepsilon_{k}}$, i.e.,
\begin{eqnarray}
&&\sup_{k\in\mathbb{N}_{+}}\tilde{\mathbb{E}}\left\{\sup_{t\in [0, T]}|\tilde{X}_{t}^{\varepsilon_{k}}|^{p}\right\}\leq C_{p,T}\left(1+\tilde{\EE}|\tilde \xi|^{3p/2}+\tilde{\EE}|\tilde \zeta|^{3p}\right),\label{tildeX1}\\
&&\sup_{k\in\mathbb{N}_{+}}\sup_{t\geq 0}\tilde{\mathbb{E}}|\tilde Y_{t}^{\varepsilon_k}|^{p}\leq C_{p}\left(1+\tilde{\EE}|\tilde\zeta|^{p}\right).\nonumber\label{tildeY0}
\end{eqnarray}
Then by \eref{Xasconvergence} and \eref{tildeX1}, applying the Vitali's convergence theorem (cf.~\cite[Theorem 4.5.4]{B1}), we deduce that for any $p'<p$,
\begin{equation}\label{vita2}
\tilde{\EE}\left\{\sup_{t\in[0,T]}|\tilde{X}^{\varepsilon_{k}}_t-\tilde{X}_t|^{p'}\right\}\to 0,~\text{as}~k\to\infty.\nonumber
\end{equation}

\vspace{0.2cm}
\textbf{Step 2}: In this step, we will identify the limiting processes $\tilde{I_1},\tilde{I}_{21}, \tilde{I}_{22},\tilde{I}_{23}$ respectively. On the one hand, it is easy to see that $\lim_{k\rightarrow \infty}\tilde{I}_{21}^{\varepsilon_k}(t)=0$. Furthermore, by Lemma \ref{lem3} in the appendix, it is easy to check that $\tilde \PP$-a.s.,
\begin{eqnarray*}
&&\lim_{k\rightarrow \infty}\tilde{I}_1^{\varepsilon_k}=\int_0^{\cdot}\bar{b}(\tilde{X}_s,\mathscr{L}_{\tilde{X}_s}) ds \quad \text{in}\quad C([0,T];\RR^{n}), \label{eq7}
\\
&&\lim_{k\rightarrow \infty}\tilde{I}_{22}^{\varepsilon_k}=\int_0^{\cdot}\overline{\partial_{x}\Phi_{K}}(\tilde{X}_s,\mathscr{L}_{\tilde{X}_s})ds \quad \text{in}\quad C([0,T];\RR^{n}),\label{eq8}
\\
&&\lim_{k\rightarrow \infty}\tilde{I}_{23}^{\varepsilon_k}=\int_0^{\cdot}\overline{\partial_{y}\Phi_{h}}(\tilde{X}_s,\mathscr{L}_{\tilde{X}_s})ds \quad \text{in}\quad C([0,T];\RR^{n}),\label{eq9}
\\
&&\lim_{k\rightarrow \infty}\tilde{I}_{24}^{\varepsilon_k}=\int_0^{\cdot}\overline{  \text{Tr}\left[\partial^2_{xy}\Phi_{\sigma g^{\ast}} \right] }(\tilde{X}_s,\mathscr{L}_{\tilde{X}_s})  ds \quad \text{in}\quad C([0,T];\RR^{n}).\label{eq10}
\end{eqnarray*}
As a consequence, it follows that
\begin{eqnarray}
\tilde{I}_1(t)+\sum^4_{i=1}\tilde{I}_{2i}(t)
=\!\!\!\!\!\!\!\!&&\int_0^{t}\overline{b+\partial_x \Phi_K+\partial_y \Phi_h+\text{Tr}\left[\partial^2_{xy}\Phi_{\sigma g^{*}}\right]}(\tilde{X}_s,\mathscr{L}_{\tilde{X}_s})  ds. \label{PtildeX}
\end{eqnarray}
If we have the following result:
\begin{eqnarray}
\tilde{M}_t=\int^t_0 \Big(\overline{(\partial_y\Phi_g+\sigma)(\partial_y\Phi_g+\sigma)^*(\tilde{X}_s,\mathscr{L}_{\tilde{X}_s})
}\Big)^{\frac{1}{2}}d\tilde{W}'_s,\label{Martingale identity}
\end{eqnarray}
whose detailed proof  is left in the \textbf{Step 3} below, where $\{\tilde{W}'_t\}_{t\geq 0}$ is a standard $n$-dimensional Brownian motion on probability space $(\tilde{\Omega},\tilde{\mathscr{F}},\tilde{\PP})$, then by taking limit $k\rightarrow \infty$ in both sides of \eref{tildeX}, using \eref{PtildeX} and \eref{Martingale identity}, we immediately obtain that $\tilde{X}$ is a weak solution of (\ref{eq62}).

Consequently, by the weak uniqueness of equation (\ref{eq62}), $\tilde{X}$  equals in distribution to the solution $X$ of equation (\ref{eq62}).

\vspace{1mm}
\textbf{Step 3}: In this step, we intend to prove \eref{Martingale identity}. Note that $\hat{M}^{\varepsilon_{k}}_t$ is the continuous local martingales on $(\tilde{\Omega},\tilde{\mathscr{F}},\tilde{\PP})$ with respect to filtration $\tilde{\mathscr{F}}^k_t:=\sigma\{\tilde{W}^{k}_s,\tilde{\xi},\tilde{\zeta}, s\leq t\}$, whose quadratic variational process is
\begin{eqnarray*}
\langle \tilde{M}^{\varepsilon_{k}}\rangle_t=\int_0^{t}\big[(\partial_y\Phi_g+\sigma)(\partial_y\Phi_g+\sigma)^*(\tilde{X}^{\vare_k}_s,\mathscr{L}_{\tilde{X}_s},\tilde{Y}^{\vare_k}_s)
\big]ds.
\end{eqnarray*}
Since $\tilde{M}^{\varepsilon_{k}}$ converges to $\tilde{M}$ $\PP$-a.s., in $C([0,T];\RR^n)$, as $k\to\infty$. According to the Vitali's convergence theorem, it is easy to deduce that  $\tilde{M}$ is the continuous local martingale on $(\tilde{\Omega},\tilde{\mathscr{F}},\hat{\PP})$ with respect to filtration $\tilde{\mathscr{F}}_t:=\sigma\{\tilde{W}_s,\tilde{\xi},\tilde{\zeta}, s\leq t\}$.

Now, we define
$$
R(x,\mu,y)=(\partial_y\Phi_g+\sigma)(\partial_y\Phi_g+\sigma)^*(x,\mu,y).
$$
Then it is easy check that
$$
\|R(x_1,\mu_1,y_1)-R(x_2,\mu_2,y_2)\|\leq C(1+|y_1|^2+|y_2|^2)\left(|x_1-x_2|+\mathbb{W}_2(\mu_1,\mu_2)+|y_1-y_2|\right).
$$
By \eref{Key2} in Lemma \ref{lem3}, it follows for any $t\in [0,T]$,
\begin{eqnarray*}
\lim_{k\rightarrow \infty}\tilde{\EE}\left\|\int_0^t R(\tilde{X}_{s}^{\varepsilon_k},\mathscr{L}_{\tilde{X}^{\varepsilon_k}_{s}},\tilde{Y}_{s}^{\varepsilon_k})ds-\int_0^{t}\bar{R}(\tilde{X}_s,\mathscr{L}_{\tilde{X}_s})  ds\right\|^2=0.\label{Key}
\end{eqnarray*}
Thus there exists a subsequence of $\{\varepsilon_k\}_{k\geq 1}$ which we keep denoting by $\{\varepsilon_{k}\}_{k\geq1}$ tending to $0$ such that
$$
\int_0^t R(\tilde{X}_{s}^{\varepsilon_k},\mathscr{L}_{\tilde{X}^{\varepsilon_k}_{s}},\tilde{Y}_{s}^{\varepsilon_k})ds\rightarrow\int_0^{t}\bar{R}(\tilde{X}_s,\mathscr{L}_{\tilde{X}_s})   ds,~\tilde{\PP}\text{-a.s.},~\text{as}~k\to\infty.
$$
Note that $\tilde{M}^{\varepsilon_k}_t\otimes \tilde{M}^{\varepsilon_k}_t-\int_0^{t}R(\tilde{X}_{r}^{\varepsilon_k},\mathscr{L}_{\tilde{X}^{\varepsilon_k}_{r}},\tilde{Y}_{r}^{\varepsilon_k})dr$ is a matrix-valued martingale with respect to filtration $\tilde{\mathscr{F}}^k_t$,  by the Vitali's convergence theorem, we have for any $0\leq s\leq t\leq T$,
\begin{eqnarray*}
\!\!\!\!\!\!\!\!&&\EE\left[\left(\tilde{M}_t\otimes \tilde{M}_t-\int_0^{t}\bar{R}(\tilde{X}_r,\mathscr{L}_{\tilde{X}_r})dr\right)-\left(\tilde{M}_s\otimes \tilde{M}_s-\int_0^{s}\bar{R}(\tilde{X}_r,\mathscr{L}_{\tilde{X}_r})dr\right) \big|\tilde{\mathscr{F}}_s\right]
\nonumber \\
=\!\!\!\!\!\!\!\!&&\lim_{k\to\infty}\EE\left[\left(\tilde{M}^{\varepsilon_k}_t\otimes \tilde{M}^{\varepsilon_k}_t-\int_0^{t}R(\tilde{X}_{r}^{\varepsilon_k},\mathscr{L}_{\tilde{X}^{\varepsilon_k}_{r}},\tilde{Y}_{r}^{\varepsilon_k})dr\right)\right.\\
&&\quad\quad\quad-\left.\left(\tilde{M}^{\varepsilon_k}_s\otimes  \tilde{M}^{2,\varepsilon_k}_s-\int_0^{s}R(\tilde{X}_{r}^{\varepsilon_k},\mathscr{L}_{\tilde{X}^{\varepsilon_k}_{r}},\tilde{Y}_{r}^{\varepsilon_k})dr\right)\big|\tilde{\mathscr{F}}^{k}_s\right]
\nonumber \\
=\!\!\!\!\!\!\!\!&&0,~\tilde \PP\text{-a.s.}.
\end{eqnarray*}
Therefore, we conclude the quadratic variational process of $\tilde{M}$ is
\begin{eqnarray*}
\langle \tilde{M}\rangle_t=\int_0^{t}\overline{(\partial_y\Phi_g+\sigma)(\partial_y\Phi_g+\sigma)^*(\tilde{X}_s,\mathscr{L}_{\tilde{X}_s})}ds,~ t\in[0,T].\label{quadratic3}
\end{eqnarray*}

Finally,  according to the  martingale representation theorem (cf.~\cite{SV}), there exists a  $n$-dimensional standard Brownian motions $\tilde{W}'$ under the probability measure $\tilde{\PP}$ such that \eref{Martingale identity} holds. The proof is complete.
 \hspace{\fill}$\Box$

\allowdisplaybreaks
\section{Appendix}

\subsection{Proof of Proposition \ref{P3.6}} \label{app1.1}

Since $K(x,\mu,y)$ satisfies the centering condition $\ref{A4}$, it is well-known that (\ref{SPE}) is a solution of Poisson equation (\ref{PE}), which satisfies $\Phi(\cdot,\mu,\cdot)\in C^{3,3}(\RR^n\times\RR^m;\RR^n)$ and $\Phi(x,\cdot,y)\in C^{(1,1)}(\mathscr{P}_2(\RR^n); \RR^n)$ under the conditions $\ref{A1}$-$\ref{A2}$. The proof of (\ref{E2}) and  (\ref{E1})  is similar to that of the recent works \cite[Proposition 4.1]{RSX1} and \cite[Proposition 3.1]{HLLS1}, thus we omit the detailed proof. We now mainly focus on the proof of (\ref{E4}).

 Firstly, in view of (\ref{SPE}), we deduce that for any $h\in \RR^n,k\in\RR^m$,
$$\partial_x\Phi(x,\mu,y)\cdot h=\int_0^{\infty}\EE\left[\partial_xK(x,\mu,Y^{x,\mu,y}_t)\cdot h+\partial_yK(x,\mu,Y^{x,\mu,y}_t)\cdot\big(\partial_xY^{x,\mu,y}_t\cdot h\big)\right]dt,$$
which implies that
\begin{eqnarray*}
\!\!\!\!\!\!\!\!&&\partial^2_{yx}\Phi(x,\mu,y)\cdot(h,k)
\nonumber\\
=\!\!\!\!\!\!\!\!&&\int_0^{\infty}\EE\big[\partial^2_{yx}K(x,\mu,Y^{x,\mu,y}_t)\cdot\big(h,(\partial_yY^{x,\mu,y}_t\cdot k)\big) +\partial^2_{yy}K(x,\mu,Y^{x,\mu,y}_t)\cdot\big(\partial_x Y^{x,\mu,y}_t\cdot h,\partial_y Y^{x,\mu,y}_t\cdot k\big)\nonumber\\
\!\!\!\!\!\!\!\!&&~~~~~~~~+\partial_y K(x,\mu,Y^{x,\mu,y}_t)\cdot\big(\partial^2_{yx}Y^{x,\mu,y}_t\cdot (h,k)\big)\big]dt,
\end{eqnarray*}
where $\partial_xY^{x,\mu,y}_t\cdot h$, $\partial_yY^{x,\mu,y}_t\cdot k$ and $\partial^2_{yx}Y^{x,\mu,y}_t\cdot(h,k)$ fulfill the following equations respectively,
\begin{equation*}
\left\{ \begin{aligned}
d\partial_{x}Y^{x,\mu,y}_t\cdot h=&~\big[\partial_{x}f(x,\mu,Y^{x,\mu,y}_t)\cdot h+\partial_{y}f(x,\mu,Y^{x,\mu,y}_t)\cdot \left(\partial_xY^{x,\mu,y}_t\cdot h\right)\big]dt\\
&+\left[\partial_{x}g(x,\mu,Y^{x,\mu,y}_t)\cdot h+\partial_yg(x,\mu,Y^{x,\mu,y}_t)\cdot\left(\partial_{x}Y^{x,\mu,y}_t \cdot h\right) \right]dW_t,\\
\partial_{x}Y^{x,\mu,y}_0\cdot h=&~0,
\end{aligned} \right.
\end{equation*}
and
\begin{equation*}
\left\{ \begin{aligned}
d\partial_{y}Y^{x,\mu,y}_t\cdot k=&~\big[\partial_{y}f(x,\mu,Y^{x,\mu,y}_t)\cdot \left(\partial_yY^{x,\mu,y}_t\cdot k\right)\big]dt+\left[\partial_yg(x,\mu,Y^{x,\mu,y}_t)\cdot\left(\partial_{y}Y^{x,\mu,y}_t \cdot k\right) \right]dW_t,\\
\partial_{y}Y^{x,\mu,y}_0\cdot k=&~k,
\end{aligned} \right.
\end{equation*}
and
\begin{equation*}
\left\{ \begin{aligned}
d\partial^2_{yx}Y^{x,\mu,y}_t\cdot(h,k)=&~\Big[\partial^2_{yx}f(x,\mu,Y^{x,\mu,y}_t)\cdot (h,(\partial_yY^{x,\mu,y}_t\cdot k))\\
&+\partial^2_{yy}f(x,\mu,Y^{x,\mu,y}_t)\cdot (\partial_yY^{x,\mu,y}_t\cdot k,\partial_xY^{x,\mu,y}_t\cdot h)\\
&+ \partial_yf(x,\mu,Y^{x,\mu,y}_t)\cdot\left(\partial^2_{yx}Y^{x,\mu,y}_t \cdot(h,k)\right) \Big]dt\\
&+\Big[\partial^2_{yx}g(x,\mu,Y^{x,\mu,y}_t)\cdot\big(h,(\partial_yY^{x,\mu,y}_t\cdot k)\big)\\
&+\partial^2_{yy}g(x,\mu,Y^{x,\mu,y}_t)\cdot\left(\partial_{y}Y^{x,\mu,y}_t \cdot k,\partial_{x}Y^{x,\mu,y}_t \cdot h\right)\\
&+\partial_{y}g(x,\mu,Y^{x,\mu,y}_t)\cdot\left(\partial^2_{yx}Y^{x,\mu,y}_t \cdot (h,k)\right)
 \Big]dW_t,\\
\partial^2_{yx}Y^{x,\mu,y}_0\cdot(h,k)=&~0.
\end{aligned} \right.
\end{equation*}
Under the conditions $\ref{A1}$-$\ref{A2}$, by a straightforward computation, it is easy to prove that
\begin{eqnarray}
&&\sup_{t\geq 0,x\in\RR^n,\mu\in\mathscr{P}_2(\RR^n),y\in\RR^m}\EE|\partial_xY^{x,\mu,y}_t\cdot h|^2\leq C|h|^2,\label{es19}\\
&&\sup_{x\in\RR^n,\mu\in\mathscr{P}_2(\RR^n),y\in\RR^m}\EE|\partial_yY^{x,\mu,y}_t\cdot k|^4\leq Ce^{-2\beta t}|k|^4,\label{9}\\
&&\sup_{x\in\RR^n,\mu\in\mathscr{P}_2(\RR^n),y\in\RR^m}\EE|\partial^2_{yx}Y^{x,\mu,y}_t\cdot(h,k)|^2\leq Ce^{-\beta t}|h|^2|k|^2,\label{10}
\end{eqnarray}
which together with the boundedness of $\|\partial_yK\|$, $\|\partial^2_{yx}K\|$ and $\|\partial^2_{yy}K\|$, it follows that
$$\sup_{x\in\RR^n,\mu\in\mathscr{P}_2(\RR^n),y\in\RR^m}|\partial_{yx}^2 \Phi(x,\mu,y)\cdot(h,k)|\leq C|h||k|.$$

Then the proof of (\ref{E4}) will be divided by the following three steps.\\

\textbf{Step 1:} For any $h,l\in \RR^n,k\in\RR^m$, we know that
\begin{eqnarray*}
\partial^2_{xyx}\Phi(x,\mu,y)\cdot(h,k,l)
=\!\!\!\!\!\!\!\!&&\int_0^{\infty}\EE\Big[\partial^3_{xyx}K(x,\mu,Y^{x,\mu,y}_t)\cdot\big(h,\partial_yY^{x,\mu,y}_t\cdot k,l\big)\nonumber\\
\!\!\!\!\!\!\!\!&&+\partial^3_{yyx}K(x,\mu,Y^{x,\mu,y}_t)\cdot\big(h,(\partial_yY^{x,\mu,y}_t\cdot k),(\partial_yY^{x,\mu,y}_t\cdot k)\big)
\nonumber\\
\!\!\!\!\!\!\!\!&&+\partial^2_{yx}K(x,\mu,Y^{x,\mu,y}_t)\cdot\big(h,\partial^2_{xy}Y^{x,\mu,y}_t\cdot (k,l)\big)
\nonumber\\
\!\!\!\!\!\!\!\!&&
+\partial^3_{xyy}K(x,\mu,Y^{x,\mu,y}_t)\cdot\big(\partial_yY^{x,\mu,y}_t\cdot k,\partial_yY^{x,\mu,y}_t\cdot k,h\big)
\nonumber\\
\!\!\!\!\!\!\!\!&&
+\partial^3_{yyy}K(x,\mu,Y^{x,\mu,y}_t)\cdot\big(\partial_xY^{x,\mu,y}_t\cdot h,\partial_yY^{x,\mu,y}_t\cdot k,\partial_yY^{x,\mu,y}_t\cdot k\big)
\nonumber\\
\!\!\!\!\!\!\!\!&&
+2\partial^2_{yy}K(x,\mu,Y^{x,\mu,y}_t)\cdot\big(\partial^2_{xy}Y^{x,\mu,y}_t\cdot(k,h),\partial_yY^{x,\mu,y}_t\cdot k\big)
\nonumber\\
\!\!\!\!\!\!\!\!&&+\partial^2_{xy}K(x,\mu,Y^{x,\mu,y}_t)\cdot\big(\partial^2_{yx}Y^{x,\mu,y}_t\cdot(h,k),l\big)
\nonumber\\
\!\!\!\!\!\!\!\!&&+\partial^2_{yy}K(x,\mu,Y^{x,\mu,y}_t)\cdot\big(\partial_xY^{x,\mu,y}_t\cdot h,\partial^2_{yx}Y^{x,\mu,y}_t\cdot(l,k)\big)
\nonumber\\
\!\!\!\!\!\!\!\!&&+\partial_{y}K(x,\mu,Y^{x,\mu,y}_t)\cdot\big(\partial^3_{xyx}Y^{x,\mu,y}_t\cdot (h,k,l)\big)
\Big]dt.
\end{eqnarray*}
Similar to (\ref{10}), it is easy to deduce that
\begin{eqnarray*}
\sup_{x\in\RR^n,\mu\in\mathscr{P}_2(\RR^n),y\in\RR^m}\EE|\partial^2_{xy}Y^{x,\mu,y}_t\cdot(k,h)|^2\leq Ce^{-\beta t}|h|^2|k|^2.
\end{eqnarray*}
We recall that $\partial^3_{xyx}Y^{x,\mu,y}_t\cdot(h,k,l)$ satisfies
\begin{equation*}
\left\{ \begin{aligned}
d\partial^2_{xyx}Y^{x,\mu,y}_t\cdot(h,k,l)=&~\Big[\partial^3_{xyx}f(x,\mu,Y^{x,\mu,y}_t)\cdot (h,\partial_yY^{x,\mu,y}_t\cdot k,l)\\
&+\partial^3_{yyx}f(x,\mu,Y^{x,\mu,y}_t)\cdot (h,\partial_xY^{x,\mu,y}_t\cdot l,\partial_yY^{x,\mu,y}_t\cdot k)\\
&+\partial^2_{yx}f(x,\mu,Y^{x,\mu,y}_t)\cdot (h,\partial^2_{xy}Y^{x,\mu,y}_t\cdot (k,l))\\
&+ \partial^3_{xyy}f(x,\mu,Y^{x,\mu,y}_t)\cdot(\partial_{y}Y^{x,\mu,y}_t\cdot k,\partial_{x}Y^{x,\mu,y}_t\cdot h,l) \\
&+\partial^3_{yyy}f(x,\mu,Y^{x,\mu,y}_t)\cdot(\partial_{x}Y^{x,\mu,y}_t\cdot h,\partial_{y}Y^{x,\mu,y}_t\cdot k,\partial_{x}Y^{x,\mu,y}_t\cdot l)\\
&+\partial^2_{yy}f(x,\mu,Y^{x,\mu,y}_t)\cdot (\partial^2_{xy}Y^{x,\mu,y}_t\cdot (k,h),\partial_{x}Y^{x,\mu,y}_t\cdot l)\\
&+\partial^2_{yy}f(x,\mu,Y^{x,\mu,y}_t)\cdot (\partial_{y}Y^{x,\mu,y}_t\cdot k,\partial^2_{xx}Y^{x,\mu,y}_t\cdot (h,l))\\
&+\partial^2_{xy}f(x,\mu,Y^{x,\mu,y}_t)\cdot (\partial^2_{yx}Y^{x,\mu,y}_t\cdot (h,k),l)\\
&+\partial_{y}f(x,\mu,Y^{x,\mu,y}_t)\cdot (\partial^3_{xyx}Y^{x,\mu,y}_t\cdot (h,k,l))\Big]dt+d\tilde{M}_t,\\
\partial^3_{xyx}Y^{x,\mu,y}_0\cdot(h,k,l)=&~0,
\end{aligned} \right.
\end{equation*}
where
\begin{eqnarray*}
d\tilde{M}_t:=\!\!\!\!\!\!\!\!&&\Big[\partial^3_{xyx}g(x,\mu,Y^{x,\mu,y}_t)\cdot (h,\partial_yY^{x,\mu,y}_t\cdot k,l)\\
\!\!\!\!\!\!\!\!&&+\partial^3_{yyx}g(x,\mu,Y^{x,\mu,y}_t)\cdot (h,\partial_xY^{x,\mu,y}_t\cdot l,\partial_yY^{x,\mu,y}_t\cdot k)\\
\!\!\!\!\!\!\!\!&&+\partial^2_{yx}g(x,\mu,Y^{x,\mu,y}_t)\cdot (h,\partial^2_{xy}Y^{x,\mu,y}_t\cdot (k,l))\\
\!\!\!\!\!\!\!\!&&+ \partial^3_{xyy}g(x,\mu,Y^{x,\mu,y}_t)\cdot(\partial_{y}Y^{x,\mu,y}_t\cdot k,\partial_{x}Y^{x,\mu,y}_t\cdot h,l) \\
\!\!\!\!\!\!\!\!&&+\partial^3_{yyy}g(x,\mu,Y^{x,\mu,y}_t)\cdot(\partial_{x}Y^{x,\mu,y}_t\cdot h,\partial_{y}Y^{x,\mu,y}_t\cdot k,\partial_{x}Y^{x,\mu,y}_t\cdot l)\\
\!\!\!\!\!\!\!\!&&+\partial^2_{yy}g(x,\mu,Y^{x,\mu,y}_t)\cdot (\partial^2_{xy}Y^{x,\mu,y}_t\cdot (k,h),\partial_{x}Y^{x,\mu,y}_t\cdot l)\\
\!\!\!\!\!\!\!\!&&+\partial^2_{yy}g(x,\mu,Y^{x,\mu,y}_t)\cdot (\partial_{y}Y^{x,\mu,y}_t\cdot k,\partial^2_{xx}Y^{x,\mu,y}_t\cdot (h,l))\\
\!\!\!\!\!\!\!\!&&+\partial^2_{xy}g(x,\mu,Y^{x,\mu,y}_t)\cdot (\partial^2_{yx}Y^{x,\mu,y}_t\cdot (h,k),l)\\
\!\!\!\!\!\!\!\!&&+\partial_{y}g(x,\mu,Y^{x,\mu,y}_t)\cdot (\partial^3_{xyx}Y^{x,\mu,y}_t\cdot (h,k,l))\Big]d W_t.
\end{eqnarray*}

Under the conditions $\ref{A1}$-$\ref{A2}$, by a straightforward computation, it is easy to prove that
\begin{eqnarray}\label{es20}
\sup_{x\in\RR^n,\mu\in\mathscr{P}_2(\RR^n),y\in\RR^m}\EE|\partial^3_{xyx}Y^{x,\mu,y}_t\cdot(h,k,l)|^2\leq Ce^{-\beta t}|h|^2|l|^2|k|^2,
\end{eqnarray}
then by (\ref{es19})-(\ref{es20}), we can infer that
 $$\sup_{x\in\RR^n,\mu\in\mathscr{P}_2(\RR^n),y\in\RR^m}|\partial_{xyx}^3 \Phi(x,\mu,y)\cdot(h,k,l)|\leq C|h||l||k|.$$

\textbf{Step 2:} Similarly, for any $h\in\RR^n,k,l\in\RR^m$, $\partial^3_{yyx}\Phi(x,\mu,y)\cdot(h,k,l)$ could be represented by
\begin{eqnarray*}
\partial^3_{yyx}\Phi(x,\mu,y)\cdot(h,k,l)
=\!\!\!\!\!\!\!\!&&\int_0^{\infty}\tilde \EE\Big[\partial^3_{yyx}K(x,\mu,Y^{x,\mu,y}_t)\cdot\big(h,\partial_yY^{x,\mu,y}_t\cdot k,\partial_yY^{x,\mu,y}_t\cdot l\big)
\nonumber\\
\!\!\!\!\!\!\!\!&&
+\partial^2_{yx}K(x,\mu,Y^{x,\mu,y}_t)\cdot\big(h,\partial^2_{yy}Y^{x,\mu,y}_t\cdot (k,l)\big)
\nonumber\\
\!\!\!\!\!\!\!\!&&
+\partial^3_{yyy}K(x,\mu,Y^{x,\mu,y}_t)\cdot\big(\partial_yY^{x,\mu,y}_t\cdot k,\partial_yY^{x,\mu,y}_t\cdot l,\partial_yY^{x,\mu,y}_t\cdot l\big)
\nonumber\\
\!\!\!\!\!\!\!\!&&+2\partial^2_{yy}K(x,\mu,Y^{x,\mu,y}_t)\cdot\big(\partial_yY^{x,\mu,y}_t\cdot k,\partial^2_{yy}Y^{x,\mu,y}_t\cdot (l,l)\big)
\nonumber\\
\!\!\!\!\!\!\!\!&&+\partial^2_{yy}K(x,\mu,Y^{x,\mu,y}_t)\cdot\big(\partial_yY^{x,\mu,y}_t\cdot k,\partial^2_{yx}Y^{x,\mu,y}_t\cdot (h,l)\big)
\nonumber\\
\!\!\!\!\!\!\!\!&&+\partial_{y}K(x,\mu,Y^{x,\mu,y}_t)\cdot\big(\partial^3_{yyx}Y^{x,\mu,y}_t\cdot (h,k,l)\big)
\Big]dt.
\end{eqnarray*}
Similar to (\ref{10}), it is easy to deduce that
\begin{eqnarray*}
\sup_{x\in\RR^n,\mu\in\mathscr{P}_2(\RR^n),y\in\RR^m}\EE|\partial^2_{yy}Y^{x,\mu,y}_t\cdot(k,l)|^2\leq Ce^{-\beta t}|k|^2|l|^2.
\end{eqnarray*}

We also recall that $\partial^3_{yyx}Y^{x,\mu,y}_t\cdot(h,k,l)$ fulfills the following equation
\begin{equation*}
\left\{ \begin{aligned}
d\partial^2_{yyx}Y^{x,\mu,y}_t\cdot(h,k,l)=&~\Big[\partial^3_{yyx}f(x,\mu,Y^{x,\mu,y}_t)\cdot (h,\partial_yY^{x,\mu,y}_t\cdot k,\partial_yY^{x,\mu,y}_t\cdot l)\\
&+\partial^2_{yx}f(x,\mu,Y^{x,\mu,y}_t)\cdot (h,\partial^2_{yy}Y^{x,\mu,y}_t\cdot (k,l))\\
&+\partial^3_{yyy}f(x,\mu,Y^{x,\mu,y}_t)\cdot (\partial_{y}Y^{x,\mu,y}_t\cdot k,\partial_{y}Y^{x,\mu,y}_t\cdot l,\partial_{x}Y^{x,\mu,y}_t\cdot h)\\
&+ \partial^2_{yy}f(x,\mu,Y^{x,\mu,y}_t)\cdot(\partial^2_{yy}Y^{x,\mu,y}_t\cdot (k,l),\partial_{x}Y^{x,\mu,y}_t\cdot h) \\
&+2\partial^2_{yy}f(x,\mu,Y^{x,\mu,y}_t)\cdot(\partial_{y}Y^{x,\mu,y}_t\cdot k,\partial^2_{yx}Y^{x,\mu,y}_t\cdot (h,l))\\
&+\partial_{y}f(x,\mu,Y^{x,\mu,y}_t)\cdot (\partial^3_{yyx}Y^{x,\mu,y}_t\cdot (h,k,l))\Big]dt+d\tilde{N}_t\\
\partial^3_{yyx}Y^{x,\mu,y}_0\cdot(h,k,l)=&~0,
\end{aligned} \right.
\end{equation*}
where
\begin{eqnarray*}
d\tilde{N}_t:=\!\!\!\!\!\!\!\!&&\Big[\partial^3_{yyx}g(x,\mu,Y^{x,\mu,y}_t)\cdot (h,\partial_yY^{x,\mu,y}_t\cdot k,\partial_yY^{x,\mu,y}_t\cdot l)\\
\!\!\!\!\!\!\!\!&&+\partial^2_{yx}g(x,\mu,Y^{x,\mu,y}_t)\cdot (h,\partial^2_{yy}Y^{x,\mu,y}_t\cdot (k,l))\\
\!\!\!\!\!\!\!\!&&+\partial^3_{yyy}g(x,\mu,Y^{x,\mu,y}_t)\cdot (\partial_{y}Y^{x,\mu,y}_t\cdot k,\partial_{y}Y^{x,\mu,y}_t\cdot l,\partial_{x}Y^{x,\mu,y}_t\cdot h)\\
\!\!\!\!\!\!\!\!&&+ \partial^2_{yy}g(x,\mu,Y^{x,\mu,y}_t)\cdot(\partial^2_{yy}Y^{x,\mu,y}_t\cdot (k,l),\partial_{x}Y^{x,\mu,y}_t\cdot h)  \\
\!\!\!\!\!\!\!\!&&+2\partial^2_{yy}g(x,\mu,Y^{x,\mu,y}_t)\cdot(\partial_{y}Y^{x,\mu,y}_t\cdot k,\partial^2_{yx}Y^{x,\mu,y}_t\cdot (h,l))\\
\!\!\!\!\!\!\!\!&&+\partial_{y}g(x,\mu,Y^{x,\mu,y}_t)\cdot (\partial^3_{yyx}Y^{x,\mu,y}_t\cdot (h,k,l))\Big]dW_t
\end{eqnarray*}

Under the conditions $\ref{A1}$-$\ref{A2}$, one can prove that
\begin{eqnarray}\label{es21}
\sup_{x\in\RR^n,\mu\in\mathscr{P}_2(\RR^n),y\in\RR^m}\tilde\EE|\partial^3_{yyx}Y^{x,\mu,y}_t\cdot(h,k,l)|^2\leq Ce^{-\beta t}|h|^2|k|^2|l|^2,
\end{eqnarray}
then by (\ref{es19})-(\ref{es21}), it turns out that
 $$\sup_{x\in\RR^n,\mu\in\mathscr{P}_2(\RR^n),y\in\RR^m}|\partial_{yyx}^3 \Phi(x,\mu,y)\cdot(h,k,l)|\leq C|h||k||l|.$$

\textbf{Step 3:} For the term $\partial_{\mu}\partial^2_{yx}\Phi(x,\mu,y)$, under the conditions $\ref{A1}$-$\ref{A2}$, one can easily obtain that for any $\mu_1,\mu_2\in \mathscr{P}_2(\RR^n)$,
\begin{eqnarray*}
&&\sup_{x\in\RR^n,y\in\RR^m}\EE|Y^{x,\mu_1,y}_t-Y^{x,\mu_2,y}_t|^2\leq C\mathbb{W}_{2}(\mu_1,\mu_2)^2,\\
&&\sup_{x\in\RR^n,y\in\RR^m}\EE\|\partial_{y}Y^{x,\mu_1,y}_t-\partial_{y}Y^{x,\mu_2,y}_t\|^2\leq Ce^{-\beta t}\mathbb{W}_{2}(\mu_1,\mu_2)^2,\\
&&\sup_{x\in\RR^n,y\in\RR^m}\EE\|\partial^2_{yx}Y^{x,\mu_1,y}_t-\partial^2_{yx}Y^{x,\mu_2,y}_t\|^2\leq Ce^{-\beta t}\mathbb{W}_{2}(\mu_1,\mu_2)^2,\label{es22}
\end{eqnarray*}
which together with the condition $\ref{A2}$ and estimates (\ref{9})-(\ref{10}) also implies that
\begin{eqnarray}\label{es33}
\!\!\!\!\!\!\!\!&&\|\partial^2_{yx}\Phi(x,\mu_1,y)-\partial^2_{yx}\Phi(x,\mu_2,y)\|
\nonumber\\
\leq \!\!\!\!\!\!\!\!&&\int_0^{\infty}\EE\Big[\|\partial^2_{yx}K(x,\mu_1,Y^{x,\mu_1,y}_t)-\partial^2_{yx}K(x,\mu_2,Y^{x,\mu_1,y}_t)\|\cdot\|\partial_yY^{x,\mu_1,y}_t\|\nonumber\\
\!\!\!\!\!\!\!\!&&+\|\partial^2_{yx}K(x,\mu_2,Y^{x,\mu_1,y}_t)-\partial^2_{yx}K(x,\mu_2,Y^{x,\mu_2,y}_t)\|\cdot\|\partial_yY^{x,\mu_1,y}_t\|\nonumber\\
\!\!\!\!\!\!\!\!&&+\|\partial^2_{yx}K(x,\mu_2,Y^{x,\mu_2,y}_t)\|\cdot\|\partial_yY^{x,\mu_1,y}_t-\partial_yY^{x,\mu_2,y}_t\|\nonumber\\
\!\!\!\!\!\!\!\!&&+\|\partial^2_{yy}K(x,\mu_1,Y^{x,\mu_1,y}_t)-\partial^2_{yy}K(x,\mu_2,Y^{x,\mu_1,y}_t)\|\cdot\|\partial_yY^{x,\mu_1,y}_t\|^2\nonumber\\
\!\!\!\!\!\!\!\!&&+\|\partial^2_{yy}K(x,\mu_2,Y^{x,\mu_1,y}_t)-\partial^2_{yy}K(x,\mu_2,Y^{x,\mu_2,y}_t)\|\cdot\|\partial_yY^{x,\mu_1,y}_t\|^2\nonumber\\
\!\!\!\!\!\!\!\!&&+2\|\partial^2_{yy}K(x,\mu_2,Y^{x,\mu_2,y}_t)\|\cdot\|\partial_yY^{x,\mu_1,y}_t\|\cdot\|\partial_yY^{x,\mu_1,y}_t-\partial_yY^{x,\mu_2,y}_t\|\nonumber\\
\!\!\!\!\!\!\!\!&&+\|\partial_yK(x,\mu_1,Y^{x,\mu_1,y}_t)-\partial_yK(x,\mu_2,Y^{x,\mu_1,y}_t)\|\cdot\|\partial^2_{yx}Y^{x,\mu_1,y}_t\|\nonumber\\
\!\!\!\!\!\!\!\!&&+\|\partial_yK(x,\mu_2,Y^{x,\mu_1,y}_t)-\partial_yK(x,\mu_2,Y^{x,\mu_2,y}_t)\|\cdot\|\partial^2_{yx}Y^{x,\mu_1,y}_t\|\nonumber\\
\!\!\!\!\!\!\!\!&&+\|\partial_yK(x,\mu_2,Y^{x,\mu_2,y}_t)\|\cdot\|\partial^2_{yx}Y^{x,\mu_1,y}_t-\partial^2_{yx}Y^{x,\mu_2,y}_t\|\Big]dt\nonumber\\
\leq \!\!\!\!\!\!\!\!&& C\mathbb{W}_{2}(\mu_1,\mu_2).
\end{eqnarray}
As a direct consequence of (\ref{es33}), we know
$$\sup_{x\in\RR^n,\mu\in\mathscr{P}_2,y\in\RR^m}\|\partial_{\mu}\partial^2_{yx} \Phi(x,\mu,y)\|_{L^2(\mu)}\leq C.$$

 The proof is complete.    \hspace{\fill}$\Box$

\subsection{Proof of \eref{Y0}-\eref{Y1}}\label{appendix 1}

We consider process $Z^\varepsilon_t:=Y^{\varepsilon}_{\varepsilon t}$ that solves the following equation
$$ Z^\varepsilon_t=f(X_{t\varepsilon}^\varepsilon,\mathscr{L}_{X_{t\varepsilon}^\varepsilon},Z^\varepsilon_t)dt+\sqrt{\varepsilon}h(X_{t\varepsilon}^\varepsilon,\mathscr{L}_{X_{t\varepsilon}^\varepsilon},Z^\varepsilon_t)dt+
g(X_{t\varepsilon}^\varepsilon,\mathscr{L}_{X_{t\varepsilon}^\varepsilon},Z^\varepsilon_t)d\tilde{W}_t,~\tilde{Y}^\varepsilon_0=\zeta,$$
where $\tilde{W}_t:=\frac{1}{\sqrt{\varepsilon}}W_{t\varepsilon}$ that coincides in law with $W_t$.

By It\^{o}'s formula, for any $p\geq 4$, we have
\begin{eqnarray}
|Z_{t}^{\varepsilon}|^{p}=\!\!\!\!\!\!\!\!&&|\zeta|^{p}+p\int_{0} ^{t}|Z_{s}^{\varepsilon}|^{p-2}\langle f(X_{s\ep}^{\ep},\mathscr{L}_{X_{s\ep}^{\ep}},Z_{s}^{\ep}),Z_{s}^{\ep}\rangle ds+p\int_{0} ^{t}|Z_{s}^{\varepsilon}|^{2p-2}\langle Z_{s}^{\ep}, g(X_{s\ep}^{\ep},\mathscr{L}_{X_{s\ep}^{\ep}},Z_{s}^{\ep})d\tilde{W}_s\rangle  \nonumber\\
&&+\frac{p}{2}\int_{0} ^{t}|Z_{s}^{\ep}|^{p-2}\|g(X_{s\ep}^{\ep}, \mathscr{L}_{X_{s\ep}^{\ep}}, Z_{s}^{\ep})\|^2ds+\frac{p(p-2)}{2}\int_{0} ^{t}|g^{*}(X^{\varepsilon}_{s\varepsilon},\mathscr{L}_{X^{\varepsilon}_{s\varepsilon}}, Z^{\varepsilon}_{s})Z^{\varepsilon}_{s}|^2 |Z^{\varepsilon}_{s}|^{p-4}ds\nonumber\\
&&+\sqrt{\ep}p\int_{0} ^{t}|Z_{s}^{\varepsilon}|^{p-2}\langle h(X_{s\ep}^{\ep},\mathscr{L}_{X_{s\ep}^{\ep}},Z_{s}^{\ep}),Z_{s}^{\ep}\rangle ds.\label{ITO}
\end{eqnarray}

In view of the conditions \eref{sm} and \eref{A6}, it is easy to deduce that there exist $C_p>0$ and $\beta\in(0,\gamma)$ such that for small enough $\ep>0$,
\begin{eqnarray}
 2\langle f(x,\mu,y), y\rangle+(p-1)\|g(x,\mu, y)\|^2+2\sqrt{\ep}\langle h(x,\mu,y), y\rangle\leq -\beta |y|^2+C_p,\label{RE3}
\end{eqnarray}
which implies that
\begin{eqnarray*}
\frac{d}{dt}\mathbb{E}|Z_{t}^{\varepsilon}|^{p} \leq-\frac{\beta p}{2}\mathbb{E}|Z_{t}^{\varepsilon}|^{p}+C_{p}.
\end{eqnarray*}
Applying the comparison theorem, we obtain
\begin{eqnarray*}
\mathbb{E}|Z_{t}^{\varepsilon}|^{p}\leq \EE|\zeta|^{p}e^{-\frac{\beta p t}{2}}+C_{\beta}\int^t_0 e^{-\frac{\beta p(t-s)}{2}}ds.
\end{eqnarray*}
Hence we get for any small enough $\ep>0$,
\begin{eqnarray}
\sup_{t\geq 0}\EE|Z^{\varepsilon}_{t}|^{p}\leq C_{p}(1+\EE|\zeta|^{p}),\label{Prior E}
\end{eqnarray}
which also implies the estimate (\ref{Y0}) holds.

As for (\ref{Y1}), noting that \eref{ITO} and \eref{RE3} imply that
\begin{eqnarray*}
|Z_{t}^{\varepsilon}|^{p}\leq\!\!\!\!\!\!\!\!&& |\zeta|^{p}+C_p t+C_p\left|\int_{0} ^{t}|Z_{s}^{\varepsilon}|^{p-2}\langle Z_{s}^{\ep}, g(X_{s\ep}^{\ep},\mathscr{L}_{X_{s\ep}^{\ep}},Z_{s}^{\ep})d\tilde{W}_s\rangle\right|.
\end{eqnarray*}
By Burkholder-Davis-Gundy's inequality and Young's inequality we get
\begin{eqnarray*}
\mathbb{E}\left[\sup_{t\in[0,T]}|Z_{t}^{\varepsilon}|^{p}\right]\leq\!\!\!\!\!\!\!\!&&\EE|\zeta|^{p}+C_p T +C_p\EE\left[\int_0^T|Z^\varepsilon_s|^{2p-2}(|Z^\varepsilon_s|^{2}+1)ds\right]^{\frac{1}{2}}
\nonumber\\
\leq\!\!\!\!\!\!\!\!&&\EE|\zeta|^{p}+C_p T+\frac{1}{2}\mathbb{E}\left[\sup_{t\in[0,T]}|Z_{t}^{\varepsilon}|^{p}\right]+C_p\int_0^T\left(\EE|Z^\varepsilon_s|^{p}+1\right)ds.
\end{eqnarray*}
Then by \eref{Prior E}, we obtain that for any $T\geq 1$,
\begin{eqnarray*}
\mathbb{E}\left[\sup_{t\in[0,T]}|Z_{t}^{\varepsilon}|^{p}\right]\leq C_p(1+\EE|\zeta|^{p})T.
\end{eqnarray*}
Hence, it follows that for any $p\in \mathbb{Z}_{+}, T>0$ and $\varepsilon$ small enough,
\begin{eqnarray*}
\mathbb{E}\left[\sup_{t\in[0,T]}|Y_{t}^{\varepsilon}|^{p}\right]=\mathbb{E}\left[\sup_{t\in\left[0,\frac{T}{\varepsilon}\right]}|Z_{t}^{\varepsilon}|^{p}\right]\leq\frac{C_p(1+\EE|\zeta|^{p})T}{\ep},
\end{eqnarray*}
which yields (\ref{Y1}) holds.  Hence the proof is complete.       \hspace{\fill}$\Box$

\subsection{Averaging Principle}
We consider a function $F:\RR^n\times\mathscr{P}_2(\RR^n)\times\RR^m\to\RR$ satisfying that for any $x_1,x_2\in\RR^n$, $\mu_1,\mu_2\in\mathscr{P}_2(\RR^n)$ and $y_1,y_2\in\RR^m$,
\begin{eqnarray}\label{es8}
|F(x_1,\mu_1,y_1)-F(x_2,\mu_2,y_2)|\leq\!\!\!\!\!\!\!\!&& C\left(1+|x_1|+\left[\mu_1(|\cdot|^2)\right]^{1/2}+|y_1|^2+|y_2|^2\right)|x_1-x_2|\nonumber\\
&&+C(1+|y_1|^2+|y_2|^2)\big(\mathbb{W}_{2}(\mu_1,\mu_2)+|y_1-y_2|\big).
\end{eqnarray}

\begin{lemma}\label{lem3}
Suppose that the assumptions $\ref{A1}$-$\ref{A4}$ hold and $F$ satisfies \eref{es8}. Then for any $T>0$, we have
\begin{eqnarray}
\lim_{k\to \infty}\hat\EE\left\{\sup_{t\in[0,T]}\left|\int_0^t F(\hat X_{s}^{\varepsilon_k},\mathscr{L}_{\hat X^{\varepsilon_k}_{s}},\hat Y_{s}^{\varepsilon_k})ds-\int_0^{t}\bar{F}(\hat{X}_s,\mathscr{L}_{\hat{X}_s})  ds\right|^2\right\}=0\label{Key3}
\end{eqnarray}
and
\begin{eqnarray}
\lim_{k\to \infty}\tilde\EE\left\{\sup_{t\in[0,T]}\left|\int_0^t F(\tilde X_{s}^{\varepsilon_k},\mathscr{L}_{\tilde X^{\varepsilon_k}_{s}},\tilde Y_{s}^{\varepsilon_k})ds-\int_0^{t}\bar{F}(\tilde{X}_s,\mathscr{L}_{\tilde{X}_s})  ds\right|^2\right\}=0,\label{Key2}
\end{eqnarray}
where $(\hat{X}^{\varepsilon_k},\hat{Y}^{\varepsilon_k})$ and $(\tilde{X}^{\varepsilon_k},\tilde{Y}^{\varepsilon_k})$ are the solutions of equations \eref{ChangeEquation} and \eref{ChangeEquation2} respectively.
\end{lemma}

\begin{proof} Since the proof of \eref{Key3} and \eref{Key2} follow the same steps, we only prove \eref{Key3} here.

Under the conditions \eref{sm} and $\ref{A3}$, one can easily  prove that for any $y\in \RR^m$,
\begin{eqnarray}
\sup_{x\in \RR^n, \mu\in\mathscr{P}_2}\EE|Y_{t}^{x,\mu,y}|^6\leq e^{-\beta t }|y|^6+C,\quad \sup_{x\in \RR^n, \mu\in\mathscr{P}_2}\int_{\RR^m}|y|^6\nu^{x,\mu}(dy)<\infty,\label{F4.18}
\end{eqnarray}
where $Y_{t}^{x,\mu,y}$ is the solution of frozen equation (\ref{FEQ2}). Moreover, there exist $\beta\in (0,\gamma)$ and for any $x_1,x_2\in\RR^n,\mu_1,\mu_2\in\mathscr{P}_2(\RR^n)$ such that
\begin{eqnarray}
\EE\left |Y^{ x_1,\mu_1,y_1}_t-Y^{x_2,\mu_2,y_2}_t\right |^2\leq e^{-\beta t}|y_1-y_2|^2+C\left[|x_1-x_2|^2+\mathbb{W}_{2}(\mu_1,\mu_2)^2\right].\label{F4.19}
\end{eqnarray}
Using \eref{F4.18}, \eref{F4.19} and the definition of invariant measure,  for any $t\geq 0$,
\begin{eqnarray}
\!\!\!\!\!\!\!\!&&\left|\EE F(x,\mu,Y_t^{x,\mu,y})-\bar{F}(x,\mu)\right|\nonumber\\
=\!\!\!\!\!\!\!\!&&\left|\EE F( x, \mu, Y^{x,\mu,y}_t)-\int_{\RR^m}F(x,\mu, z)\nu^{x,\mu}(dz)\right|\nonumber\\
=\!\!\!\!\!\!\!\!&& \left|\int_{\RR^m}\left[\EE F(x, \mu,Y^{x,\mu,y}_t)-\tilde \EE F(x, \mu,Y^{x,\mu,z}_t)\right]\nu^{x,\mu}(dz)\right|\nonumber\\
\leq\!\!\!\!\!\!\!\!&& C\int_{\RR^m}\!\!\!\!\EE\left[(1+| Y^{x,\mu,y}_t|^2+|Y^{x,\mu,z}_t|^2)\left| Y^{x,\mu,y}_t-Y^{x,\mu,z}_t\right|\right]\nu^{x,\mu}(dz)\nonumber\\
\leq\!\!\!\!\!\!\!\!&& C e^{-\frac{\beta t}{2}}\int_{\RR^m}(1+|y|^2+|z|^2)|y-z|\nu^{x,\mu}(dz)\nonumber\\
\leq\!\!\!\!\!\!\!\!&&C e^{-\frac{\beta t}{2}}\left(1+|y|^3\right).\label{Ergodicity}
\end{eqnarray}
By \eref{Ergodicity}, it turns out  that for any $t>0$,
\begin{eqnarray*}
|\bar{F}(x_1,\mu_1)-\bar{F}(x_2,\mu_2)|=\!\!\!\!\!\!\!\!&&\left|\bar{F}(x_1,\mu_1)-\tilde{\EE}F(x_1,\mu_1,Y_t^{x_1,\mu_1,0})\right|\\
\!\!\!\!\!\!\!\!&&+\left|\tilde{\EE}F(x_2,\mu_2,Y_t^{x_2,\mu_2,0})-\bar{F}(x_2,\mu_2)\right|\\
\!\!\!\!\!\!\!\!&&+\left|\tilde{\EE}F(x_1,\mu_1,Y_t^{x_1,\mu_1,0})-\tilde{\EE}F(x_2,\mu_2,Y_t^{x_2,\mu_2,0})\right|\\
\leq\!\!\!\!\!\!\!\!&& Ce^{-\frac{\beta t}{2} }+C(1+|x_1|+\mu_1(|\cdot|^2)^{1/2})|x_1-x_2|+C\mathbb{W}_{2}(\mu_1,\mu_2).
\end{eqnarray*}
Consequently, letting $t\rightarrow \infty$,  $\bar{F}(x,\mu)$ satisfies
\begin{eqnarray*}
|\bar{F}(x_1,\mu_1)-\bar{F}(x_2,\mu_2)|\leq C(1+|x_1|+\mu_1(|\cdot|^2)^{1/2})|x_1-x_2|+C\mathbb{W}_{2}(\mu_1,\mu_2).\label{LF}
\end{eqnarray*}

The proof of Lemma \ref{lem3}  will be separated by the following two steps.

\vspace{1mm}
\textbf{Step 1:} In this step, we intend to prove \eref{Key3}.
Firstly, we consider an auxiliary process $\bar{Y}^{\varepsilon_k}_t$, which is defined by
\begin{equation*}\left\{\begin{array}{l}\label{30}
\displaystyle
d\bar{Y}_{t}^{\varepsilon_k}=\frac{1}{\varepsilon_k} f(\hat X_{t(\Delta)}^{\varepsilon_k},\mathscr{L}_{\hat X_{t(\Delta)}^{\varepsilon_k}},\bar{Y}_{t}^{\varepsilon_k})dt+\frac{1}{\sqrt{\varepsilon_k}} h(\hat X_{t(\Delta)}^{\varepsilon_k},\mathscr{L}_{\hat X_{t(\Delta)}^{\varepsilon_k}},\bar{Y}_{t}^{\varepsilon_k})dt\\
\quad\quad\quad\quad+\frac{1}{\sqrt{\varepsilon_k}}g( \hat X_{t(\Delta)}^{\varepsilon_k},\mathscr{L}_{\hat X_{t(\Delta)}^{\varepsilon_k}},\bar{Y}_{t}^{\varepsilon_k})d\hat W^{k}_t,\\
\bar{Y}_{0}^{\varepsilon_k}=\hat\zeta,
\end{array}\right.
\end{equation*}
where $t(\Delta):=[\frac{t}{\Delta}]\Delta$ and $[s]$ denotes the integer part of $s$.

By a straightforward computation, it is easy to check that for any $T>0$, there exists  a constant $C_{T}>0$ such that for any $k$ large enough,
\begin{eqnarray}
&&\sup_{t\geq 0}\hat\EE|\bar Y_{t}^{\varepsilon_k}|^{6}\leq C(1+\hat\EE|\hat \zeta|^{6});\label{F4.21}\\
&&\sup_{t\in[0,T]}\!\left(\hat \EE|\hat X_{t}^{\varepsilon_k}-\hat X_{t(\Delta)}^{\varepsilon_k}|^{6}+\mathbb{E}|\hat{Y}_{t}^{\varepsilon_k}-\bar{Y}_{t}^{\varepsilon_k}|^{6}\right)\leq \!C_{T}(1+\hat\EE|\hat\xi|^9+\hat\EE|\hat\zeta|^{18})\!\!\left(\Delta^{3}+\frac{\Delta^{6}}{\vare^3_k}\right).~\label{F4.22}
\end{eqnarray}

Note that we have the composition
\begin{eqnarray}\label{15}
\!\!\!\!\!\!\!\!&&\int_0^t F\left(\hat X_{s}^{\varepsilon_k},\mathscr{L}_{\hat X^{\varepsilon_k}_{s}},\hat Y_{s}^{\varepsilon_k}\right)ds-\int_0^{t}\bar{F}\left(\hat{X}_s,\mathscr{L}_{\hat{X}_s}\right)  ds
\nonumber \\
=\!\!\!\!\!\!\!\!&&\int_{0} ^{t}F\left(\hat X^{\varepsilon_k}_s,\mathscr{L}_{\hat X^{\varepsilon_k}_s},\hat Y^{\varepsilon_k}_s\right)
-F\left(\hat X^{\varepsilon_k}_{s(\Delta)},\mathscr{L}_{\hat X^{\varepsilon_k}_{s(\Delta)}},\bar{Y}^{\varepsilon_k}_s\right) ds
\nonumber \\
 \!\!\!\!\!\!\!\!&& + \int_{0} ^{t} \bar{F}\left(\hat X^{\varepsilon_k}_{s(\Delta)},\mathscr{L}_{\hat X^{\varepsilon_k}_{s(\Delta)}}\right)-
 \bar{F}\left(\hat X^{\varepsilon_k}_{s},\mathscr{L}_{\hat X^{\varepsilon_k}_s}\right) ds\nonumber \\
 \!\!\!\!\!\!\!\!&& + \int_{0} ^{t}
 \bar{F}\left(\hat X^{\varepsilon_k}_{s},\mathscr{L}_{\hat X_s^{\varepsilon_k}}\right)
 -\bar{F}\left(\hat{X}_{s},\mathscr{L}_{\hat{X}_{s}}\right) ds\nonumber \\
  \!\!\!\!\!\!\!\!&&+ \int_{0} ^{t} F\left(\hat X^{\varepsilon_k}_{s(\Delta)},\mathscr{L}_{\hat X^{\varepsilon_k}_{s(\Delta)}},\bar{Y}^{\varepsilon_k}_s\right)-\bar{F}\left(\hat X^{\varepsilon_k}_{s(\Delta)},\mathscr{L}_{\hat X^{\varepsilon_k}_{s(\Delta)}}\right) ds\nonumber \\
=:\!\!\!\!\!\!\!\!&&\mathscr{I}_1(t)+\mathscr{I}_2(t)+\mathscr{I}_3(t)+\mathscr{I}_4(t).
\end{eqnarray}
In what follows, we will estimate the terms $\mathscr{I}_i(t)$, $i=1,2,3,4$, respectively.

Due to (\ref{Y0}), (\ref{COX}), \eref{F4.21} and \eref{F4.22}, we have
\begin{eqnarray}
\!\!\!\!\!\!\!\!&&\hat\EE\left\{\sup_{t\in[0,T]}|\mathscr{I}_{1}(t)+\mathscr{I}_{2}(t)|^2\right\}\nonumber \\
\leq \!\!\!\!\!\!\!\!&&
C_T\hat\EE\left|\int_0^T\left[1+|\hat X_{s}^{\varepsilon_k}|+\left(\hat\EE|\hat X_{s}^{\varepsilon_k}|^2\right)^{1/2}+|\hat Y_{s}^{\varepsilon_k}|^2+|\bar{Y}_{s}^{\varepsilon_k}|^2\right]|\hat X_{s}^{\varepsilon_k}-\hat X_{s(\Delta)} ^{\varepsilon_k}|\right.\nonumber\\
&&\quad\quad\quad\left.+\left(1+|\hat Y_{s}^{\varepsilon_k}|^2+|\bar{Y}_{s}^{\varepsilon_k}|^2\right)\left(\mathbb{W}_{2}(\mathscr{L}_{\hat X^{\vare_k}_{s}},\mathscr{L}_{\hat X_{s(\Delta)} ^{\varepsilon_k}})+|\hat Y_{s}^{\varepsilon_k}-\bar{Y}_{s} ^{\varepsilon_k}|\right)ds\right|^2
\nonumber\\\leq\!\!\!\!\!\!\!\!&&
C_T\left[\hat\EE\int_0^T\!\!\!\!\left(1+|\hat X_{s}^{\varepsilon_k}|^3+|\hat Y_{s}^{\varepsilon_k}|^6+|\bar{Y}_{s}^{\varepsilon_k}|^6\right)ds\right]^{\frac{2}{3}}\nonumber\\
&&\quad\quad\cdot\left[\hat \EE\int_0^T|\hat X_{s}^{\varepsilon_k}-\hat X^{\varepsilon_k}_{s(\Delta)}|^6ds
+C_T\EE\int_0^T|\hat Y_{s}^{\varepsilon_k}-\bar{Y}_{s} ^{\varepsilon_k}|^6 ds\right]^{\frac{1}{3}}
\nonumber\\\leq\!\!\!\!\!\!\!\!&&
C_T\left(\Delta+\frac{\Delta^2}{\varepsilon_k}\right)\left(1+\hat\EE|\hat \xi|^{9}+\hat\EE|\hat\zeta|^{18}\right).
\label{p6}
\end{eqnarray}

For the term $\mathscr{I}_{3}(t)$. It follows that
\begin{eqnarray}
\hat \EE\left\{\sup_{t\in [0,T]}|\mathscr{I}_{3}(t)|^2\right\}
\leq\!\!\!\!\!\!\!\!&&C\hat{\EE}\Big|\int^T_0\left[1+|\hat{X}_{s}^{\varepsilon_k}|+\left(\hat{\EE}|\hat{X}_{s}^{\varepsilon_k}|^2\right)^{1/2}\right]\big[|\hat{X}_{s}^{\varepsilon_k}-\hat{X}_s|+\hat{\EE}|\hat{X}_{s}^{\varepsilon_k}-\hat{X}_s|^2\big]ds\Big|^2\nonumber\\
\leq\!\!\!\!\!\!\!\!&&C_T\left[\hat{\EE}\left(\sup_{s\in[0,T]}|\hat{X}_{s}^{\varepsilon_k}-\hat{X}_s|^4\right)\right]^{1/2}.
\end{eqnarray}

It remains to consider the term $\mathscr{I}_4(t)$. Note that

\begin{eqnarray}
 |\mathscr{I}_4(t)|^2=\!\!\!\!\!\!\!\!&&\left|\sum_{j=0}^{[t/\Delta]-1}\int_{j\Delta} ^{(j+1)\Delta} F\left(\hat X^{\varepsilon_k}_{s(\Delta)},\mathscr{L}_{\hat X^{\varepsilon_k}_{s(\Delta)}},\bar{Y}^{\varepsilon_k}_s\right)-\bar{F}\left(\hat X^{\varepsilon_k}_{s(\Delta)},\mathscr{L}_{\hat X^{\varepsilon_k}_{s(\Delta)}}\right) ds\right.\nonumber \\
 \!\!\!\!\!\!\!\!&& +\left.\int_{t(\Delta)} ^{t} F\left(\hat X^{\varepsilon_k}_{s(\Delta)},\mathscr{L}_{\hat X^{\varepsilon_k}_{s(\Delta)}},\bar{Y}^{\varepsilon_k}_s\right)-\bar{F}\left(\hat X^{\varepsilon_k}_{s(\Delta)},\mathscr{L}_{\hat X^{\varepsilon_k}_{s(\Delta)}}\right) ds\right|^2\nonumber \\
 \leq\!\!\!\!\!\!\!\!&&\frac{C_T}{\Delta}\sum_{j=0}^{[t/\Delta]-1}\left|\int_{j\Delta} ^{(j+1)\Delta} F\left(\hat X^{\varepsilon_k}_{s(\Delta)},\mathscr{L}_{\hat X^{\varepsilon_k}_{s(\Delta)}},\bar{Y}^{\varepsilon_k}_s\right)-\bar{F}\left(\hat X^{\varepsilon_k}_{s(\Delta)},\mathscr{L}_{\hat X^{\varepsilon_k}_{s(\Delta)}}\right)ds\right|^2\nonumber \\
 \!\!\!\!\!\!\!\!&& +2\left|\int_{t(\Delta)} ^{t} F\left(\hat X^{\varepsilon_k}_{s(\Delta)},\mathscr{L}_{\hat X^{\varepsilon_k}_{s(\Delta)}},\bar{Y}^{\varepsilon_k}_s\right)-\bar{F}\left(\hat X^{\varepsilon_k}_{s(\Delta)},\mathscr{L}_{\hat X^{\varepsilon_k}_{s(\Delta)}}\right) ds\right|^2\nonumber \\
=:\!\!\!\!\!\!\!\!&&\mathscr{V}_{1}(t)+\mathscr{V}_{2}(t).  \label{p12}
\end{eqnarray}
For the term $\mathscr{V}_{2}(t)$, owing to the estimates (\ref{X2}) and (\ref{F4.21}),
\begin{eqnarray}
\hat\EE\left\{\sup_{t\in [0,T]}\mathscr{V}_{2}(t)\right\}\leq \!\!\!\!\!\!\!\!&&
C_T\Delta\hat\EE\int_{0} ^{T} \big[1+|\hat X^{\varepsilon_k}_{s(\Delta)}|^6+\mathscr{L}_{\hat X^{\varepsilon_k}_{s(\Delta)}}(|\cdot|^6)+|\bar{Y}^{\varepsilon_k}_s|^6\big]ds.
\nonumber \\
 \leq\!\!\!\!\!\!\!\!&& C_T\Delta(1+{\hat\EE|\hat\xi|^{9}}+{\hat \EE|\hat\zeta|^{18}}).\label{p8}
\end{eqnarray}

Once the following estimate holds
\begin{eqnarray}
\hat \EE\left\{\sup_{t\in [0,T]}\mathscr{V}_{1}(t)\right\}
\leq C_T(1+{\hat\EE|\hat\xi|^{9}}+{\hat\EE|\hat\zeta|^{18}})\left(\varepsilon_k+\frac{\varepsilon^2_k}{\Delta^2}+\frac{\varepsilon_k}{\Delta}+\Delta\sqrt{\varepsilon_k}\right),\label{w3}
\end{eqnarray}
then collecting estimates (\ref{15})-(\ref{w3}) yields that
\begin{eqnarray*}
&&\hat\EE\left|\int_0^{t}F\left(\hat X_{s}^{\varepsilon_k},\mathscr{L}_{\hat X^{\varepsilon_k}_{s}},\hat Y_{s}^{\varepsilon_k}\right)ds-\int_0^{t}\bar{F}\left(\hat{X}_s,\mathscr{L}_{\hat{X}_s}\right)  ds\right|^2
\nonumber \\
\leq\!\!\!\!\!\!\!\!&&
 C_T(1+{\hat\EE|\hat\xi|^{9}}+{\hat\EE|\hat\zeta|^{18}})\left(\frac{\Delta^2}{\varepsilon_k}+\Delta+\varepsilon_k+\frac{\varepsilon_k}{\Delta}+\frac{\varepsilon_k^2}{\Delta^2}\right)
 +C_T\left[\hat{\EE}\left(\sup_{t\in[0,T]}|\hat X_{t}^{\varepsilon_k}-\hat{X}_t|^4\right)\right]^{1/2}.
\end{eqnarray*}
Taking $\Delta=\varepsilon_k^{\frac{2}{3}}$, it is easy to see \eref{Key3} holds by \eref{vita1}.

\vspace{3mm}
\textbf{Step 2:} In this step, we intend to prove \eref{w3}. Note that
\begin{eqnarray}
\hat\EE\left\{\sup_{t\in [0,T]}\mathscr{V}_{1}(t)\right\}\leq\!\!\!\!\!\!\!\!&&\frac{C_T}{\Delta}\hat\EE\sum_{j=0}^{[T/\Delta]-1}\left|\int_{j\Delta} ^{(j+1)\Delta} F\left(\hat X^{\varepsilon_k}_{j\Delta},\mathscr{L}_{\hat X^{\varepsilon_k}_{j\Delta}},\bar{Y}^{\varepsilon_k}_s\right)-\bar{F}\left(\hat X^{\varepsilon_k}_{j\Delta},\mathscr{L}_{\hat X^{\varepsilon_k}_{j\Delta}}\right) ds\right|^2\nonumber \\
\leq\!\!\!\!\!\!\!\!&&\frac{C_T}{\Delta^2}\max_{0\leq j\leq[T/\Delta]-1}\EE\left|\int_{j\Delta} ^{(j+1)\Delta} F\left(\hat X^{\varepsilon_k}_{j\Delta},\mathscr{L}_{\hat X^{\varepsilon_k}_{j\Delta}},\bar{Y}^{\varepsilon_k}_s\right)-\bar{F}\left(\hat X^{\varepsilon_k}_{j\Delta},\mathscr{L}_{\hat X^{\varepsilon_k}_{j\Delta}}\right)ds\right|^2\nonumber \\
\leq\!\!\!\!\!\!\!\!&&\frac{C_T\varepsilon^2_k}{\Delta^2}\max_{0\leq j\leq[T/\Delta]-1}\hat \EE\left|\int_{0} ^{\frac{\Delta}{\varepsilon_k}} F\left(\hat X^{\varepsilon_k}_{j\Delta},\mathscr{L}_{\hat X^{\varepsilon_k}_{j\Delta}},\bar{Y}^{\varepsilon_k}_{s\varepsilon_k+j\Delta}\right)-\bar{F}\left(\hat X^{\varepsilon_k}_{j\Delta},\mathscr{L}_{\hat X^{\varepsilon_k}_{j\Delta}}\right)ds\right|^2
\nonumber \\
\leq\!\!\!\!\!\!\!\!&&
\frac{C_T\varepsilon^2_k}{\Delta^2}\max_{0\leq j\leq[T/\Delta]-1}\left[\int_{0} ^{\frac{\Delta}{\varepsilon_k}} \int_{r} ^{\frac{\Delta}{\varepsilon_k}}\Psi_j(s,r)dsdr \right],\label{w1}
\end{eqnarray}
where for any $0\leq r\leq s\leq \frac{\Delta}{\varepsilon}$,
\begin{eqnarray*}
\Psi_j(s,r):=\!\!\!\!\!\!\!\!&&\hat \EE\left[\left\langle F\left(\hat X^{\varepsilon_k}_{j\Delta},\mathscr{L}_{\hat X^{\varepsilon_k}_{j\Delta}},\bar{Y}^{\varepsilon_k}_{s\varepsilon_k+j\Delta}\right)-\bar{F}\left(\hat X^{\varepsilon_k}_{j\Delta},\mathscr{L}_{\hat X^{\varepsilon_k}_{j\Delta}}\right),
\right.\right.\nonumber \\
\!\!\!\!\!\!\!\!&&\left.\left.~~~~~~F\left(\hat X^{\varepsilon_k}_{j\Delta},\mathscr{L}_{\hat X^{\varepsilon_k}_{j\Delta}},\bar{Y}^{\varepsilon_k}_{r\varepsilon_k+j\Delta}\right)-\bar{F}\left(\hat X^{\varepsilon_k}_{j\Delta},\mathscr{L}_{\hat X^{\varepsilon_k}_{j\Delta}}\right)\right\rangle\right].
\end{eqnarray*}

For any $s\geq 0$, $\mu\in\mathscr{P}_2$, and any $\hat{\mathscr{F}}_s$-measurable $\RR^n$-valued random variable $X$ and $\RR^m$-valued random variable $Y$, we consider the following equation
\begin{eqnarray}\label{p14}
\left\{ \begin{aligned}
&d\tilde{Y}_{t}=\frac{1}{\varepsilon_k}f(X,\mu,\tilde{Y}_{t})dt+\frac{1}{\sqrt{\varepsilon_k}}h(X,\mu,\tilde{Y}_{t})dt+\frac{1}{\sqrt{\varepsilon_k}}g(X,\mu,\tilde{Y}_t)d\hat {W}_{t}^{k},\\
&\tilde{Y}_s=Y.
\end{aligned} \right.
\end{eqnarray}
Note that equation \eref{p14} has a unique solution denoted by $\tilde{Y}_t^{\varepsilon_k,s,X,\mu,Y}$.
By the definition of $\bar{Y}_t^{\varepsilon_k}$, for any $j\in \mathbb{N}$, we have
 $$\bar{Y}_t^{\varepsilon_k}=\tilde{Y}_t^{\varepsilon_k,j\Delta,\hat X^{\varepsilon_k}_{j\Delta},
 \mathscr{L}_{\hat X^{\varepsilon_k}_{j\Delta}},\bar{Y}^{\varepsilon_k}_{j\Delta}},~~
 t\in[j\Delta,(j+1)\Delta].$$
Then it leads to
\begin{eqnarray*}
\Psi_j(s,r)=\!\!\!\!\!\!\!\!&&\hat\EE\left[\left\langle F\left(\hat X^{\varepsilon_k}_{j\Delta},\mathscr{L}_{\hat X^{\varepsilon_k}_{j\Delta}},\tilde{Y}^{\varepsilon_k,j\Delta,\hat X_{j\Delta}^{\varepsilon_k},\mathscr{L}_{\hat X^{\varepsilon_k}_{j\Delta}},\bar{Y}^{\varepsilon_k}_{j\Delta}}_{s\varepsilon_k+j\Delta}\right)
-\bar{F}\left(\hat X^{\varepsilon_k}_{j\Delta},\mathscr{L}_{\hat X^{\varepsilon_k}_{j\Delta}}\right),
\right.\right.\nonumber \\
\!\!\!\!\!\!\!\!&&\left.\left.~~~~~~F\left(\hat X^{\varepsilon_k}_{j\Delta},\mathscr{L}_{\hat X^{\varepsilon_k}_{j\Delta}},\tilde{Y}^{\varepsilon_k,j\Delta, \hat X_{j\Delta}^{\varepsilon_k},\mathscr{L}_{\hat X^{\varepsilon_k}_{j\Delta}},\bar{Y}^{\varepsilon_k}_{j\Delta}}_{r\varepsilon_k+j\Delta}\right)-\bar{F}\left(\hat X^{\varepsilon_k}_{j\Delta},\mathscr{L}_{\hat X^{\varepsilon_k}_{j\Delta}}\right)\right\rangle\right].
\end{eqnarray*}
We point out  that for any fixed $x\in \RR^n$ and
$y\in \RR^m$, $\tilde{Y}_{s\varepsilon+j\Delta}^{\varepsilon,j\Delta,x,\mu,y}$ is independent of $\hat{\mathscr{F}}_{j\Delta}$, and $\hat X_{j\Delta}^{\varepsilon_{k}}$, $\bar{Y}_{j\Delta}^{\varepsilon_k}$ are $\hat{\mathscr{F}}_{j\Delta}$-measurable, therefore
\begin{eqnarray*}
\Psi_j(s,r)=\!\!\!\!\!\!\!\!&&\hat\EE\Big\{\hat \EE\Big[\big\langle F\big(\hat X^{\varepsilon_k}_{j\Delta},\mathscr{L}_{\hat X^{\varepsilon_k}_{j\Delta}},\tilde{Y}^{\varepsilon_k,j\Delta,X_{j\Delta}^{\varepsilon_k},\mathscr{L}_{\hat X^{\varepsilon_k}_{j\Delta}},\bar{Y}^{\varepsilon_k}_{j\Delta}}_{s\varepsilon_k+j\Delta}\big)-\bar{F}\big(\hat X^{\varepsilon_k}_{j\Delta},\mathscr{L}_{\hat X^{\varepsilon_k}_{j\Delta}}\big),
\nonumber\\
\!\!\!\!\!\!\!\!&&~~~~~~F\big(\hat X^{\varepsilon_k}_{j\Delta},\mathscr{L}_{\hat X^{\varepsilon_k}_{j\Delta}},\tilde{Y}^{\varepsilon_k,j\Delta,\hat X_{j\Delta}^{\varepsilon_k},\mathscr{L}_{\hat X^{\varepsilon_k}_{j\Delta}},\bar{Y}^{\varepsilon_k}_{j\Delta}}_{r\varepsilon_k+j\Delta}\big)-\bar{F}\big(\hat X^{\varepsilon_k}_{j\Delta},
\mathscr{L}_{\hat X^{\varepsilon_k}_{j\Delta}}\big)\big\rangle\big|\hat{\mathscr{F}}_{j\Delta}\Big] \Big\}
\nonumber\\
=\!\!\!\!\!\!\!\!&&
\hat\EE\Big\{\hat\EE\Big[\big\langle F\big(x,\mathscr{L}_{\hat X^{\varepsilon_k}_{j\Delta}},\tilde{Y}^{\varepsilon_k,j\Delta,x,\mathscr{L}_{\hat X^{\varepsilon_k}_{j\Delta}},y}_{s\varepsilon_k+j\Delta}\big)-\bar{F}\big(x,\mathscr{L}_{\hat X^{\varepsilon_k}_{j\Delta}}\big),
\nonumber\\
\!\!\!\!\!\!\!\!&&~~~~~~F\big(x,\mathscr{L}_{\hat X^{\varepsilon_k}_{j\Delta}},\tilde{Y}^{\varepsilon_k,j\Delta,x,\mathscr{L}_{\hat X^{\varepsilon_k}_{j\Delta}},y}_{r\varepsilon_k+j\Delta}\big)-\bar{F}\big(x,
\mathscr{L}_{X^{\varepsilon_k}_{j\Delta}}\big)\big\rangle\Big]\Big|_{\{(x,y)=(\hat X^{\varepsilon_k}_{j\Delta},\bar{Y}^{\varepsilon_k}_{j\Delta})\}}\Big\}.
\end{eqnarray*}
Recall the definition of $\{\tilde{Y}_{s\varepsilon_k+j\Delta}^{\varepsilon_k,j\Delta,x,\mu,y}\}_{s\geq0}$, we can deduce that
\begin{eqnarray}
\tilde{Y}_{s\varepsilon_k+j\Delta}^{\varepsilon_k,j\Delta,x,\mu,y}
=\!\!\!\!\!\!\!\!&&y+\frac{1}{\varepsilon_k}\int_{j\Delta}^{s\varepsilon_k+j\Delta}f(x,\mu,\tilde{Y}_{r}^{\varepsilon_k,j\Delta,x,\mu,y})dr
+\frac{1}{\sqrt{\varepsilon_k}}\int_{j\Delta}^{s\varepsilon_k+j\Delta}h(x,\mu,\tilde{Y}_{r}^{\varepsilon_k,j\Delta,x,\mu,y})dr\nonumber\\
&&+\frac{1}{\sqrt{\varepsilon_k}}\int_{j\Delta}^{s\varepsilon_k+j\Delta}g(x,\mu,\tilde{Y}_{r}^{\varepsilon_k,j\Delta,x,\mu,y})d\hat{W}_{r}^{k}
\nonumber \\
=\!\!\!\!\!\!\!\!&&
y+\frac{1}{\varepsilon_k}\int_{0}^{s\varepsilon_k}f(x,\mu,\tilde{Y}_{r+j\Delta}^{\varepsilon_k,j\Delta,x,\mu,y})dr
+\frac{1}{\sqrt{\varepsilon_k}}\int_{0}^{s\varepsilon_k+j\Delta}h(x,\mu,\tilde{Y}_{r}^{\varepsilon_k,j\Delta,x,\mu,y})dr\nonumber\\
&&+\frac{1}{\sqrt{\varepsilon_k}}\int_{0}^{s\varepsilon_k}g(x,\mu,\tilde{Y}_{r+j\Delta}^{\varepsilon_k,j\Delta,x,\mu,y})d\hat{W}_{r}^{k,j\Delta}
\nonumber \\
=\!\!\!\!\!\!\!\!&&
y+\int_{0}^{s}f(x,\mu,\tilde{Y}_{r\varepsilon_k+j\Delta}^{\varepsilon_k,j\Delta,x,\mu,y})dr
+\sqrt{\varepsilon_k}\int_{0}^{s}h(x,\mu,\tilde{Y}_{r\varepsilon_k+j\Delta}^{\varepsilon_k,j\Delta,x,\mu,y})dr\nonumber\\
&&+\int_{0}^{s}g(x,\mu,\tilde{Y}_{r\varepsilon_k
+j\Delta}^{\varepsilon_k,j\Delta,x,\mu,y})dW_{r}^{k,j\Delta},\label{p20}
\end{eqnarray}
where $$\Big\{\hat W_{r}^{k,j\Delta}:=\hat W_{r+j\Delta}^{k}-\hat W_{j\Delta}^{k}\Big\}_{r\geq0}~~\text{and}~~\Big\{W_{r}^{k,j\Delta}:=\frac{1}{\sqrt{\varepsilon_k}}\hat{W}_{r\varepsilon_k}^{k,j\Delta}\Big\}_{r\geq0}.$$

Consider the following frozen equation:
\begin{eqnarray}
Y_{s}^{x,\mu,y,\varepsilon_k}=\!\!\!\!\!\!\!\!&&
y+\int_{0}^{s}f(x,\mu,{Y}_{r}^{x,\mu,y,\varepsilon_k})dr+\sqrt{\varepsilon_k}\int_{0}^{s}h(x,\mu,{Y}_{r}^{x,\mu,y,\varepsilon_k})dr\nonumber\\
&&+\int_{0}^{s}g(x,\mu,{Y}_{r}^{x,\mu,y,\varepsilon_k})dW_{r}.\label{p21}
\end{eqnarray}
Hence the uniqueness of the solution of \eref{p20} and \eref{p21} implies that $\left\{\tilde{Y}_{s\varepsilon_k+j\Delta}^{\varepsilon_k,j\Delta,x,\mu,y}\right\}_{0\leq s\leq\frac{\Delta}{\varepsilon_k}}$ coincides in distribution with $\left\{{Y}_{s}^{x,\mu,y,\varepsilon_k}\right\}_{0\leq s\leq\frac{\Delta}{\varepsilon_k}}.$

Then it is easy to see for $k$ large enough,
$$\sup_{t\geq 0}\hat\EE|Y_{t}^{x,\mu,y,\varepsilon_k}|^{6}\leq C(1+\hat\EE|\hat \zeta|^{6}).$$
By the Markov property, it follows that

\begin{eqnarray}
\Psi_j(s,r)=\!\!\!\!\!\!\!\!&&
\hat\EE\left\{\EE\Big[\Big\langle F\Big(x,\mathscr{L}_{\hat X^{\varepsilon_k}_{j\Delta}},{Y}^{x,\mathscr{L}_{\hat X^{\varepsilon_k}_{j\Delta}},y,\varepsilon_k}_{s}\Big)-\bar{F}\Big(x,\mathscr{L}_{\hat X^{\varepsilon_k}_{j\Delta}}\Big)\right.,
\nonumber \\
\!\!\!\!\!\!\!\!&&~~~~~~
\left.F\Big(x,\mathscr{L}_{\hat X^{\varepsilon_k}_{j\Delta}},
{Y}^{x,\mathscr{L}_{\hat X^{\varepsilon_k}_{j\Delta}},y,\varepsilon_k}_{r}\Big)-\bar{F}\Big(x,
\mathscr{L}_{\hat X^{\varepsilon_k}_{j\Delta}}\Big)\Big\rangle\Big]\Big|_{\{(x,y)=(\hat X^{\varepsilon_k}_{j\Delta},\bar{Y}^{\varepsilon_k}_{j\Delta})\}}\right\}
\nonumber \\
=\!\!\!\!\!\!\!\!&&
\hat\EE\left\{\EE\Big[\Big\langle \EE\Big[F\Big(x,\mathscr{L}_{\hat X^{\varepsilon_k}_{j\Delta}},{Y}^{x,\mathscr{L}_{\hat X^{\varepsilon_k}_{j\Delta}},y,\varepsilon_k}_{s}\Big)-\bar{F}\Big(x,\mathscr{L}_{\hat X^{\varepsilon_k}_{j\Delta}}\Big)\big|\mathscr{F}_{r}\Big]\right.,
\nonumber \\
\!\!\!\!\!\!\!\!&&~~~~~~\left.F\Big(x,\mathscr{L}_{\hat X^{\varepsilon_k}_{j\Delta}},
{Y}^{x,\mathscr{L}_{\hat X^{\varepsilon_k}_{j\Delta}},y,\varepsilon_k}_{r}\Big)-\bar{F}\Big(x,
\mathscr{L}_{\hat X^{\varepsilon_k}_{j\Delta}}\Big)\Big\rangle\Big]\Big|_{\{(x,y)=(\hat X^{\varepsilon_k}_{j\Delta},\bar{Y}^{\varepsilon_k}_{j\Delta})\}}\right\}
\nonumber \\
=\!\!\!\!\!\!\!\!&&
\hat \EE\left\{\EE\Big[\Big\langle \EE\Big[F\Big(x,\mathscr{L}_{\hat X^{\varepsilon_k}_{j\Delta}},{Y}^{x,\mathscr{L}_{\hat X^{\varepsilon_k}_{j\Delta}},z,\varepsilon_k}_{s-r}\Big)-\bar{F}\Big(x,\mathscr{L}_{\hat X^{\varepsilon_k}_{j\Delta}}\Big)\Big]\Big|_{z=Y^{x,\mathscr{L}_{\hat X^{\varepsilon_k}_{j\Delta}},y,\varepsilon_k}_r}\right.
\nonumber \\
\!\!\!\!\!\!\!\!&&~~~~~~\left.F\Big(x,\mathscr{L}_{\hat X^{\varepsilon_k}_{j\Delta}},
{Y}^{x,\mathscr{L}_{\hat X^{\varepsilon_k}_{j\Delta}},y,\varepsilon_k}_{r}\Big)-\bar{F}\Big(x,
\mathscr{L}_{\hat X^{\varepsilon_k}_{j\Delta}}\Big)\Big\rangle\Big]\Big|_{\{(x,y)=(\hat X^{\varepsilon_k}_{j\Delta},\bar{Y}^{\varepsilon_k}_{j\Delta})\}}\right\}.
\nonumber
\end{eqnarray}
Note that by a straightforward computation, it is easy to prove for $k$ large enough,
$$
|\EE F(x,\mu, Y^{x,\mu,y})-\EE F(x,\mu, Y^{x,\mu,y,\varepsilon_k})|\leq  C\sqrt{\varepsilon_k},
$$
which combines with \eref{F4.18} and \eref{Ergodicity},  we obtain
\begin{eqnarray}
\Psi_j(s,r)
\leq\!\!\!\!\!\!\!\!&&
C_T\hat\EE\left\{\EE\left[1+|x|^6+\hat \EE|\hat X^{\varepsilon_k}_{j\Delta}|^6+|{Y}^{x,\mathscr{L}_{\hat X^{\varepsilon_k}_{j\Delta}},y,\varepsilon_k}_r|^6\right]\Big|_{\{(x,y)=(\hat X^{\varepsilon_k}_{j\Delta},\bar{Y}^{\varepsilon_k}_{j\Delta})\}}e^{-\frac{(s-r)\beta}{2}}\right\}
\nonumber \\
&&+C_T\sqrt{\varepsilon_k}\hat\EE\left\{\EE\left[1+|x|^6+\hat \EE|\hat X^{\varepsilon_k}_{j\Delta}|^6+|{Y}^{x,\mathscr{L}_{\hat X^{\varepsilon_k}_{j\Delta}},y,\varepsilon_k}_r|^6\right]\Big|_{\{(x,y)=(\hat X^{\varepsilon_k}_{j\Delta},\bar{Y}^{\varepsilon_k}_{j\Delta})\}}\right\}\nonumber\\
\leq\!\!\!\!\!\!\!\!&& C_T(1+\hat\EE|\hat\xi|^{9}+\hat\EE|\hat \zeta|^{18})\left(e^{-\frac{(s-r)\beta}{2}}+\sqrt{\varepsilon_k}\right).\label{w2}
\end{eqnarray}
By \eref{w1} and \eref{w2}, it is easy to see \eref{w3} holds. The proof is complete.
\end{proof}

\vspace{5mm}
\noindent\textbf{Acknowledgements} { W. Hong is supported by NSFC (No.~12171354);  S. Li is supported by NSFC (No.~12001247) and
 NSF of Jiangsu Province (No.~BK20201019);  X. Sun is supported by NSFC (No.~11931004,
12090011), the QingLan Project of Jiangsu Province and the Priority Academic Program Development of Jiangsu Higher Education Institutions.}

\end{document}